%% file: randompoly.tex
\documentclass[11pt]{article}

\usepackage[utf8x]{inputenc}
\usepackage[english]{babel}
\usepackage[a4paper,margin=3cm]{geometry}
\usepackage[colorlinks=true,linkcolor=black,citecolor=black,hypertexnames=false]{hyperref}
\usepackage{amsmath,amssymb,amsthm}
\usepackage{enumerate}
\usepackage{tikz}
\usepackage{microtype}
\usepackage{charter}
\usepackage{dsfont}

\definecolor{fields}{rgb}{0,0,0.8}
\definecolor{deg}{rgb}{0,0.5,0}
\definecolor{K}{rgb}{1,1,0.5}
\definecolor{K2}{rgb}{0,0.2,1}
\definecolor{E2}{rgb}{0,0.7,1}
\definecolor{E2b}{rgb}{0.4,1,1}
\definecolor{K3}{rgb}{1,0.2,0.2}
\definecolor{E3}{rgb}{1,0.2,0.8}

\newtheorem{theointro}{Theorem}

\newtheorem{theo}{Theorem}[section]
\newtheorem{prop}[theo]{Proposition}
\newtheorem{cor}[theo]{Corollary}
\newtheorem{lem}[theo]{Lemma}

\theoremstyle{definition}
\newtheorem{deftn}[theo]{Definition}
\newtheorem{rem}[theo]{Remark}

\newcommand{\ZZ}{\mathbb Z}
\newcommand{\Zp}{\ZZ_p}
\newcommand{\QQ}{\mathbb Q}
\newcommand{\Qp}{\QQ_p}
\newcommand{\Qpbar}{\bar \QQ_p}
\newcommand{\FF}{\mathbb F}
\newcommand{\Fq}{\FF_q}
\newcommand{\Fqbar}{\bar \FF_q}
\newcommand{\RR}{\mathbb R}

\newcommand{\PP}{\mathbb P}

\newcommand{\dist}{\text{\rm dist}}
\renewcommand{\mod}{\hspace{0.6ex}\%\hspace{0.6ex}}
\newcommand{\modtau}{\hspace{0.6ex}\%^\tau\hspace{0.6ex}}

\newcommand{\Prob}{\text{\rm Prob}}
\newcommand{\EE}{\mathbb E}
\newcommand{\1}{\mathds 1}
\newcommand{\Cov}{\text{\rm Cov}}

\newcommand{\calB}{\mathcal B}
\newcommand{\calG}{\mathcal G}

\renewcommand{\O}{\mathcal O}
\newcommand{\m}{\mathfrak m}

\newcommand{\OF}{\O_F}
\newcommand{\OK}{\O_K}
\renewcommand{\OE}{\O_E}
\newcommand{\mF}{\m_F}
\newcommand{\mK}{\m_K}

\newcommand{\Et}{\text{\bf Ét}}
\newcommand{\Ex}{\text{\bf Ex}}

\newcommand{\GL}{\text{\rm GL}}

\newcommand{\card}{\#} 
\newcommand{\Aut}{\text{\rm Aut}\:}
\newcommand{\Tr}{\text{\rm Tr}}

\newcommand{\new}{\text{\rm new}}

\newcommand{\Mat}{\text{\rm Mat}}
\newcommand{\Hom}{\text{\rm Hom}}
\newcommand{\surj}{\text{\rm surj}}
\newcommand{\Falg}{F\text{\rm -alg}}

\title{Where are the zeroes of\\a random $p$-adic polynomial?}

\author{Xavier Caruso}

\date\today

\begin{document}

\maketitle

\begin{abstract}
We study the repartition of the roots of a random $p$-adic polynomial
in an algebraic closure of $\Qp$.
We prove that the mean number of roots generating a fixed finite 
extension $K$ of $\Qp$ depends mostly on the discriminant of $K$,
an extension containing less roots when it gets more ramified.
We prove further that, for any positive integer~$r$, 
a random $p$-adic polynomial of sufficiently large degree has about~$r$ 
roots on average in extensions of degree at most~$r$.

Beyond the mean, we also study higher moments and correlations
between the number of roots in two given subsets of $\Qp$ (or, more
generally, of a finite extension of $\Qp$). In this perspective, we
notably establish results highlighting that the roots tend to 
repel each other and quantify this phenomenon.
\end{abstract}

\setcounter{tocdepth}{1}
\tableofcontents

\section*{Introduction}

The distribution of roots of a random real polynomial is a classical 
subject of research that has been thoroughly studied since the pioneer 
work of Bloch and Polya~\cite{BP}, Littlewood and Offord~\cite{LO1, LO2, 
LO3} and Kac's famous paper~\cite{kac}, in which an \emph{exact} 
formula giving the avegare number of roots of a random polynomial with
gaussian coefficients appears for the first time.

Investigating similarly the behaviour of roots of $p$-adic random 
polynomials is a natural question which have recently received some
attention.
The story starts in 2006 when Evans published the article~\cite{evans}, 
in which he managed to 
adapt Kac's strategy and eventually compute the average number of zeros 
in $\Zp$ of a random polynomial of degree~$n$ with coefficients 
uniformly\footnote{By uniform distribution, we mean the distribution 
coming from the Haar measure on the compact group $(\Zp, +)$. It turns 
out that it is the correct $p$-adic analogue of the normal 
distribution.} distributed in $\Zp$.
The same year, Buhler, Goldstein, Moews and Rosenberg~\cite{BGMR}
found formulas for the probability that a random $p$-adic polynomial
has all its roots in $\Qp$.
After about ten years without further significant contributions, the
subject was revived a couple of years ago by Lerario and his collaborators
who started to undertake a systematical study of these phenomena. With
Kulkarni~\cite{LK}, they notably extend Crofton's formula to the $p$-adic 
setting and derive new estimations on the number of roots of a 
$p$-adic polynomial, establishing in particular that a uniformly 
distributed random polynomial of fixed degree over $\Zp$ has exactly one 
root in $\Qp$ on average, independently from~$p$ and from the degree.
On a slightly different note, Ait El 
Mannsour and Lerario~\cite{AL} obtain formulas counting the average 
number of lines in random projective $p$-adic varieties.
More recently, the case of nonuniform distributions has also been
addressed by Shmueli~\cite{shmueli}, who came up with sharp estimations 
on the average number of roots.

Most of the aforementioned works are concerned with the \emph{mean}
of the random variable $Z_n$ couting the number of roots of a
$p$-adic polynomial of degree~$n$.
Beyond the mean (for which one can rely on Kac's techniques), 
obtaining more information about the $Z_n$'s is a fundamental
question that has been recently addressed and elegently solved by Bhargava, 
Cremona, Fisher and Gajović~\cite{BCFG}. In their paper, they 
set up a general strategy to compute all probabilities 
$\Prob[Z_n{=}r]$ with $n$ and $r$ running over the
integers.
In addition, they observed that the formulas they obtained are
all rational functions in $p$ which are symmetric under the 
transformation $p \leftrightarrow p^{-1}$.
This beautiful and fascinating property remains nowadays quite
mysterious.

Apart from the distinction between archimedean and nonarchimedean, $\Qp$ 
differs from $\RR$ in that its arithmetic is definitely much richer;
while the absolute Galois group of $\RR$ is somehow boring, that of 
$\Qp$ is large, intricated and encodes much arithmetical subtle
information. 
In other words, the set of finite extensions of~$\Qp$ has a prominent
structure which is part of the strength and the complexity of the
$p$-adic world.
Therefore, looking at the roots of a random $p$-adic polynomial 
not only in~$\Qp$ but in an algebraic closure~$\Qpbar$ of~$\Qp$ 
sounds like a very natural and appealing question, which is the
one we address in the present paper.

To this end, we fix a finite extension $F$ of~$\Qp$ together with an 
algebraic closure $\bar F$ of~$F$.
We endow $\bar F$ with the $p$-adic norm $\Vert \cdot \Vert$ 
normalized by $\Vert p \Vert = p^{-[F:\Qp]}$ and use the letter $q$ 
to denote the cardinality of the residue field of $F$.
Given a positive integer $n$, a finite extension $K$ of $F$ and a 
compact
open subset $U$ of $K$, we introduce the random variable $Z_{U,n}$ 
counting the number of roots in $U$ of a random polynomial of degree $n$ 
with coefficients in the ring of integers $\OF$ of $F$.
Our first theorem gives an integral expression of the expected values
of the $Z_{U,n}$'s.

\begin{theointro}
\label{theo:mean}
There exists a family of functions
$\rho_{K,n} : K \to \RR^+$ ($K$ running over the set of finite 
extensions of $F$ included in $\bar F$ and $n$ running the set of
positive integers) satisfying the following property:
for any positive integer $n$,
any extension finite extension $K$ of $F$ sitting inside $\bar F$
and any open subset $U$ of $E$, we have:
$$\EE[Z_{U,n}] \,\, = \,\, 
  \sum_{K' \subset K} \,\,
  \int_{U\cap K'} \rho_{K',n}(x)\:dx$$
where the sum runs over all extensions $K'$ of $F$ included in $K$.
\end{theointro}

The functions $\rho_{K,n}$'s are called the \emph{density functions} as 
their values at a given point $x$ reflect the number of roots one may 
expect to find in a small neighborhood of $x$. Our second theorem
provides rather precise information about the density functions.

\begin{theointro}
\label{theo:density}
Let $n$ be a positive integer, let $K \subset \bar F$ be a finite
extension of $F$ and let $x \in K$.
Write $r$ for the degree of the extension $K/F$.
\begin{enumerate}
\item \emph{(Vanishing)}
If $F[x] \neq K$ or $n < r$, then $\rho_{K,n}(x) = 0$.
\item \emph{(Continuity)}
The function $\rho_{K,n}$ is continuous on $K$.
\item \emph{(Invariance under isomorphisms)}
Given a second finite extension $L$ of $F$
and an isomophism of $F$-algebras
$\sigma : K \to L$, we have
$\rho_{K,n}(x) = \rho_{L,n}\big(\sigma(x)\big)$.
\item \emph{(Transformation under homography)}
For $\left(\begin{matrix} a & b \\ c & d \end{matrix}\right)
\in \GL_2(\OF)$, we have:
$$\rho_{K,n}\left(\frac{ax+b}{cx+d}\right) = 
\Vert cx+d\Vert^{2r} \cdot \rho_{K,n}(x).$$
\item \emph{(Monotony)}
We have
$\rho_{K,n}(x) \leq \rho_{K,n+1}(x)$
and the inequality is strict if and only if $F[x] = K$ and
$r \leq n < 2r - 1$.
\item \emph{(Formulas for extremal degrees)}
If $F[x] = K$ and $x \in \OK$, then
$$\begin{array}{rr@{\hspace{0.5ex}}l}
&
\rho_{K,r}(x) & = \displaystyle \Vert D_K\Vert \cdot \frac 1{\card\big(\OK/\OF[x]\big)}
  \cdot \frac{q^{r+1} - q^r}{q^{r+1} - 1} \medskip \\
\text{for } n \geq 2r - 1, &
\rho_{K,n}(x) & = \displaystyle \Vert D_K\Vert 
  \cdot \int_{\OF[x]} \Vert t\Vert^r\: dt
\end{array}$$
where $D_K$ is the discrimant of the extension $K/F$.
\end{enumerate}
\end{theointro}

A first remarkable consequence of Theorem~\ref{theo:density} is that the 
functions $\rho_{K,n}$ are independant of $n$ provided that $n \geq 
2r - 1$. A similar behaviour was already noticed 
in \cite{BCFG} for higher moments of the random variables $Z_{F,n}$.
Besides, it is in theory feasible to derive from Theorem~\ref{theo:density} 
closed formulas for $\rho_{K,n}$ and its integral over $K$, at least when 
$n = r$ or $n \geq 2r - 1$.
For example, Theorem~\ref{theo:quadratic} below covers the case of 
quadratic extensions. Before stating it, it is convenient to introduce
the notation:
\begin{equation}
\label{eq:rhod}
\rho_n(K) = \int_K \, \rho_{K,n}(x) \: dx.
\end{equation}
By Theorem~\ref{theo:mean}, $\rho_n(K)$ counts the number of roots of
a random polynomial of degree $n$ which fall inside $K$ but outside
all strict subfields of $K$ containing~$F$.

\begin{theointro}
\label{theo:quadratic}
Let $K$ be a quadratic extension of $F$.
\begin{enumerate}[(i)]
\item If $K/F$ is unramified, we have:

\smallskip

\hspace{5em}%
$\begin{array}{rr@{\hspace{0.5ex}}l}
& \rho_2(K) & = 
  \displaystyle \frac{q^2 - q + 1}{q^2 + q + 1},
\medskip \\
\text{for } n \geq 3,
& \rho_n(K) & =
  \displaystyle \frac{q^4 + 1}{q^4 + q^3 + q^2 + q + 1}.
\end{array}$

\smallskip

\item If $K/F$ is totally ramified, we have:

\smallskip

\hspace{5em}%
$\begin{array}{rr@{\hspace{0.5ex}}l}
& \rho_2(K) & = 
  \displaystyle \Vert D_K \Vert \cdot \frac{q^2}{q^2 + q + 1},
\medskip \\
\text{for } n \geq 3,
& \rho_n(K) & =
  \displaystyle \Vert D_K \Vert \cdot \frac{q^2(q^2 + 1)}{q^4 + q^3 + q^2 + q + 1}.
\end{array}$

\bigskip

\end{enumerate}
\end{theointro}

When $K$ is the quadratic unramified extension of $F$, we notice that 
$\rho_n(K)$ is a rational function in~$q$ which is self-reciprocal, 
\emph{i.e.} invariant under the transformation $q \leftrightarrow
q^{-1}$. As recalled previously, this remarkable property also holds
for all higher moments of the random variables $Z_{F,n}$.
On the contrary, when $K/F$ is totally ramified, the function 
$\rho_d(K)$ is not self-reciprocal. 
One can nevertheless recover the expected symmetry by summing up
over all totally ramified quadratic extensions of $F$. Indeed, using
Serre's mass formula~\cite{serre}, we end up with:

\bigskip

\hspace{5em}%
$\begin{array}{rr@{\hspace{0.5ex}}l}
& \displaystyle \sum_K \rho_2(K) & = 
  \displaystyle \frac{2q}{q^2 + q + 1},
\medskip \\
\text{for } n \geq 3,
& \displaystyle \sum_K \rho_n(K) & =
  \displaystyle \frac{2q(q^2 + 1)}{q^4 + q^3 + q^2 + q + 1}.
\end{array}$

\bigskip

\noindent
(where both sums run over all totally ramified quadratic extensions of 
$F$ sitting inside~$\bar F$) which are indeed self-reciprocal rational 
functions in $q$.

Another panel of interesting corollaries of Theorem~\ref{theo:density} 
concerns the orders of magnitude of the functions $\rho_{K,n}$'s. 
Roughly speaking, Theorem~\ref{theo:density} tells us that the size of 
$\rho_{K,n}$ is controlled by the $p$-adic norm of $D_K$. It is in fact 
even more transparent when we integrate over the entire space.

\begin{theointro}
\label{theo:estimation}
Let $K \subset \bar F$ be a finite extension of $F$. 
Write $r$ for the degree of $K/F$ and $f$
for its residuel degree. We have the estimations:
$$\begin{array}{rr@{\hspace{0.5ex}}l}
& \displaystyle
  \frac{\rho_r(K)}{\Vert D_K \Vert} & = 
  \displaystyle
  \left( 1 - \frac 1 q \right) \cdot
  \sum_{m|f}\: \mu\!\left(\frac f m\right) q^{m-f} \,+\, 
  O\left(\frac 1{q^f}\right)
\medskip \\
\text{for } n \geq 2r-1,
& \displaystyle
  \frac{\rho_n(K)}{\Vert D_K \Vert} & = 
  \displaystyle
  \sum_{m|f}\: \mu\!\left(\frac f m\right) q^{m-f} \,+\,
  O\left(\frac 1{q^f}\right)
\end{array}$$
where $\mu$ denotes the Moebius function and 
the constants hidden in the $O(-)$ are absolute.
\end{theointro}

The dominant term in the two sums above is the summand
corresponding to $m = f$ and is equal to~$1$. Hence, $\rho_n(K)$
is roughly equal to $\Vert D_K \Vert$ for $n = r$ or $n \geq 2r-1$.
More precisely, one finds $\rho_n(K) = \Vert D_K \Vert \cdot (1 + 
O(q^{-1}))$ in both cases. It turns out that
this conclusion continues to holds for all $n \geq r$ thanks to
the monotony property of Theorem~\ref{theo:density}.

One can also sum up the estimations of Theorem~\ref{theo:estimation}
over all extensions of a fixed degree. Doing so, we obtain the
following theorem.

\begin{theointro}
\label{theo:sum}
For any positive integers $r$ and $n$ with $n \geq 2r-1$,
we have the estimation:
$$\sum_{K \in \Ex_r} \rho_n(K) = 
  \sum_{m|r}\: \varphi\!\left(\frac r m\right) q^{m-r} \,+\, 
  O\left(\frac{r \cdot \log \log r}{q^r}\right)$$
where $\Ex_r$ denotes the set of all extensions of $F$ 
of degree $r$ inside $\bar F$ 
and $\varphi$ is the Euler's totient function.
\end{theointro}

Again, the dominant term in the sum of Theorem~\ref{theo:sum}
corresponds to $m = r$ and its value is~$1$. Therefore, we conclude
that a random polynomial of degree $n$ has, on average, one root in
the ground field $F$, one more root in the union of extensions of
degree $2$, one more root in the union of extensions of degree $3$,
\emph{etc.} until the degree $n$ where all roots have been found.
Many variations on this theme are possible; for example, one can prove that
all roots of a random polynomial lie in the maximal unramified 
extension of $F$ expect $\frac 2 q + O\big(\frac 1{q^2}\big)$ of
them. On the other hand, we deduce from Theorem~\ref{theo:quadratic} that
the quadratic totally ramified extensions of $F$ contain $\frac 2 q + 
O\big(\frac 1{q^2}\big)$ roots outside~$F$. We then conclude that
there is no more than $O\big(\frac 1{p^2}\big)$ new roots in 
ramified extensions of degree at least~$3$.

On a different note, it is also quite instructive to study the 
fluctuations of the density 
functions $\rho_{K,n}$. Theorem~\ref{theo:density} indicates that they are 
governed by the size of the $\OF$-algebra generated by~$x$. As a 
consequence, we deduce that elements which generate a large extension 
$K$ but are close for the $p$-adic distance to a strict subfield of $K$ 
have less chance to show up as a root of a random polynomial. In other 
words, if a root $x$ of a random polynomial is congruent to an element 
of a given extension $K$ modulo a large power of $p$, it is very likely 
that $x$ actually lies in $K$. In some sense, subfields attract all 
roots in a neighborhood.

Beyond the mean, it is important to understand higher moments of the 
$Z_{U,n}$'s to draw a more precise picture of the behaviour of these 
random variables. We address this question by enlarging a bit our 
setting: instead of restricting ourselves to finite extensions of 
$F$, we consider more generally products of such extensions, \emph{i.e.} 
finite étale algebras over $F$. The nice observation is that 
Theorem~\ref{theo:mean} admits a straightforward generalization to this 
extended framework.
Applying it with $E = K^r$ (for some given finite extension $K$
of $F$) provides information about the
$r$-th moment of $Z_{K,d}$ and, more generally, sheds some light on
the distribution of $r$-tuples of roots in $K^r$.
For $K = F$ and $r = 2$, this yoga has already interesting consequences 
as it permits to compute the covariances between the~$Z_{U,n}$'s.

\begin{theointro}
\label{theo:cov}
Let $U$ and $V$ be two balls in $\OF$ that do not meet.
Pick $u \in U$ and $v \in V$. We have:
$$\frac{\Cov\big(Z_{U,n}, Z_{V,n}\big)}
{\EE[Z_{U,n}] \cdot \EE[Z_{V,n}]} \,=\, 
- 1
\,+\, \frac{(q+1)^2}{q^2 + q + 1} {\cdot} \Vert u{-}v \Vert
\,-\, \frac q{q^2 + q + 1} {\cdot} \Vert u{-}v \Vert^4$$
for all $n \geq 3$.
\end{theointro}

Although the above formula might look unattractive at first glance,
it is quite instructive. Indeed, to begin with, it indicates
that $\Cov\big(Z_{U,n}, Z_{V,n}\big)$ vanishes if and only
if $\Vert u{-}v \Vert = 1$. In other words, the random
variables $Z_{U,n}$ and $Z_{V,n}$ are uncorrelated if and only 
if $U$
and $V$ are sufficiently far away. Otherwise, $Z_{U,n}$ and 
$Z_{V,n}$ are correlated and the covariance is always negative
(still assuming that $U \cap V = \emptyset$). Moreover, the
correlation gets more and more significant when $U$ and $V$
gets closer. This tends to show that roots repel each other.
This conclusion can been understood as a consequence of the 
general principle that subalgebras attract roots; indeed, noticing
that $F$ embeds diagonally into $F^2$, the above principle tells
us that if we are given two nearby roots in $F$ of a random 
polynomial, there is a huge chance that those roots actually
coincide,
which exactly means that it is unlikely to get nearby distinct roots.

Another amazing benefit of working with étale extensions is the 
existence of mass formulas for the density functions in the spirit
of Bhargava's extension to étale algebras of the classical Serre 
mass formula~\cite{bhargava}.
Given a finite étale $F$-algebra $E$, define
$\rho_n(E)$ by the integral of Eq.~\eqref{eq:rhod} and 
let $\Aut_{\!\Falg}(E)$ denote the group of $F$-automorphisms of~$E$.

\begin{theointro}
\label{theo:mass}
For any positive integers~$r$ and $n$ with $n \geq r$, we have:
\begin{equation}
\label{eq:mass}
\sum_{E \in \Et_r} \, \frac{\rho_n(E)}{\card \Aut_{\!\Falg}(E)}  \,=\, 1
\end{equation}
where the summation set $\Et_r$ consists of all isomorphism classes of 
étale extensions $E$ of $F$ of degree~$r$ (and the notation $\card$
refers to the cardinality).
\end{theointro}

When $r = 1$, Eq.~\eqref{eq:mass} reduces to $\rho_n(F) = 1$ and so 
asserts that a random polynomial of degree at least $1$ has exactly 
one root in $F$ on average; we then recover Lerario and Kulkani's 
result in this case.
When $r$ grows, Theorem~\ref{theo:mass} roughly says that the above
remarkable property continues to hold if we count (weighted) roots in 
extensions of a fixed degree provided that we pay attention to include 
all étale algebras, and not only fields!
Notice however that Theorem~\ref{theo:estimation} shows that the
contribution of actual extensions to the sum in Eq.~\eqref{eq:mass}
is about $1/r$. The most significant part of the mass then comes
from nontrivial products of smaller degree extensions.

\paragraph*{Organization of the article.}

The plan of the article follows closely the progression of the
introduction.
In Section~\ref{sec:density}, we prove Theorems~\ref{theo:mean}
and~\ref{theo:density}. In Section~\ref{sec:closed}, we study
examples and obtain closed formulas for the density functions in
several simple cases. In addition of treating completely the case 
of quadratic extension (in line with Theorem~\ref{theo:quadratic}),
we obtain partial results for extensions of prime degrees and
for unramified extensions.
Section~\ref{sec:magnitude} is devoted to finding estimations of
orders of magnitude of the density functions and their integrals;
we notably prove Theorems~\ref{theo:estimation} and~\ref{theo:sum}
there.
Finally, in Section~\ref{sec:etale}, we present the setup of
étale algebras and
extend Theorems~\ref{theo:mean} and~\ref{theo:density} to this
setting. We then discuss applications to higher moments and mass
formulas for density functions, 
establishing Theorems~\ref{theo:cov} and~\ref{theo:mass}.

\paragraph*{Notations.}

Throughout the article, we fix a prime number $p$,
a finite extension $F$ of $\Qp$ and
an algebraic closure $\bar F$ of $F$. We use the letter $q$ to
denote the cardinality of the residue field of $F$.
We write $\Vert \cdot \Vert$ for the $p$-adic norm on $\bar F$, 
normalized by $\Vert p \Vert = p^{-[F:\Qp]}$. 

We let $\Omega_n$ be the space of polynomials of degree at most
$n$ with coefficients in $\OF$; we call $\mu_n$ the probability 
measure on $\Omega_n$ corresponding to $\lambda_F^{\otimes{n+1}}$ 
under the canonical identification $\Omega_n \simeq \OF^{n+1}$.
In a slight abuse of notations, we continue to write $\Vert 
\cdot \Vert$ for the norm on
$\Omega_n$ corresponding to the sup norm on $\OF^{n+1}$ (it is 
the so-called \emph{Gauss norm}). 

Throughout the article, all finite extensions of $F$ are implicitely
supposed to be contained in $\bar F$. If $K$ is such an extension, we 
denote by $\OK$ its ring of integers and by $\OK^\times$ the group of 
invertible elements of $\OK$. We let $\lambda_K$ be the Haar measure on 
$K$ normalized by $\lambda_K (\OK) = 1$. Our normalization choices lead 
to the transformation formula $\lambda_K(aH) = \Vert a \Vert^r \cdot 
\lambda_K(H)$ where $r$ is the degree of the extension $K/F$.

Finally, we use the notation $\card A$ to denote the cardinality
of a set~$A$.

\section{Density functions}
\label{sec:density}

The aim of this section is to define the density functions
$\rho_{K,n}$ and to prove Theorems~\ref{theo:mean} and~\ref{theo:density}.
The main ingredient we shall need is a $p$-adic version of the famous 
Kac-Rice formula which gives an integral expression for a number of 
roots of a polynomial. We will establish it in \S \ref{ssec:Kac}.
In \S \ref{ssec:limit}, we carry out a key computation
which will allow us to construct the density functions and prove
Theorem~\ref{theo:density} in~\S \ref{ssec:constrho}. We finally
move to the computation of expected number of roots and prove
Theorem~\ref{theo:mean} is \S \ref{ssec:meanZ}.

\subsection{The $p$-adic Kac-Rice formula}
\label{ssec:Kac}

A $p$-adic version of the Kac-Rice formula already appears 
in the pioneer work of Evans~\cite{evans}.
Nevertheless, for the purpose of this article, it will be more
convenient to use a different formulation from that of Evans (the
latter being
actually closer to what we usually call the ``area formula'').
For this reason, we prefer taking some time to establish our version 
of the $p$-adic Kac-Rice formula and giving a complete proof of it.
We refer to~\cite[Chapter~5]{robert} for the definition of strictly 
differentiable functions of the $p$-adic variable.

\begin{theo}
\label{theo:pKac}
Let $K$ be a finite extension of $F$ of degree $r$.
Let $U$ be a compact open subset of $K$ and 
let $f : U \to K$ be a strictly differentiable function. We assume
that $(f(x), f'(x)) \neq (0,0)$ for all $x \in U$. Then:
\begin{equation}
\label{eq:pKac}
\card f^{-1}(0) = \lim_{s \to \infty} \,
q^{sr} \cdot \int_U \Vert f'(x) \Vert^r \cdot 
\1_{\{\Vert f(x) \Vert \leq q^{-s}\}} \, dx.
\end{equation}
\end{theo}

\begin{proof}
Throughout the proof, we denote by $B_s$ the closed ball of $K$
of radius $q^{-s}$ and center~$0$. If $\pi$ denotes a uniformizer 
of $K$, the set $B_s$ can be alternatively defined by $B_s = \pi^s
\OK$. We deduce from the latter equality that $\lambda_K(B_s) = 
\Vert \pi \Vert^{sr} = q^{-sr}$.

We consider an element $a \in U$ such that $f(a) = 0$. From our 
assumption, we know that $f'(a) \neq 0$. Therefore, applying
\cite[Lemma~3.4]{CRV}, we get the existence of a positive integer $S_a$ having
the following property: for any integer $s \geq S_a$, the function 
$f$ induces a bijection from $a + f'(a)^{-1} B_s$ to $B_s$.
Up to enlarging $S_a$, we can further assume that
$\Vert f'(x) \Vert = \Vert f'(a) \Vert$ for all $x \in B_{S_a}$.
We deduce for these two facts that:
\begin{equation}
\label{eq:pKac:1}
\int_{a + B_{S_a}} \Vert f'(x) \Vert^r  \cdot
  \1_{\{\Vert f(x) \Vert \leq q^{-s}\}} \, dx
= \Vert f'(a) \Vert^r \cdot \lambda_K\big(f'(a)^{-1} B_s\big)
= q^{-sr}
\end{equation}
for all $s \geq S_a$.

From the previous discussion, we also derive that $a$ is the
unique zero of $f$ in $B_{S_a}$. In other words, the set of zeros
of $f$ is discrete. By compacity, it follows that $f$ has only 
finitely many zeros in~$U$. Let us call them $a_1, \ldots, a_m$.
Set $S = \max (S_{a_1}, \ldots, S_{a_m})$ and, for $i \in \{1, 
\ldots, m\}$, write $U_i = a_i + B_{S_{a_i}}$.
Up to enlarging again the $S_{a_i}$'s, we can assume that the
$U_i$'s are pairwise disjoint. 
Summing up the equalities~\eqref{eq:pKac:1}, we find:
\begin{equation}
\label{eq:pKac:2}
\sum_{i=1}^m \int_{U_i} \Vert f'(x) \Vert^r  \cdot
  \1_{\{\Vert f(x) \Vert \leq q^{-s}\}} \, dx
= m \cdot q^{-sr}
\end{equation}
provided that $s \geq S$.
Let $V$ be the complement in $U$ of $U_1 \sqcup \cdots \sqcup
U_m$. It is compact and the function $f$ does not vanish on it.
Hence, if $s$ is large enough, we have $\Vert f(x) \Vert > q^{-s}$
for all $x \in V$. For those $s$, we thus get:
\begin{equation}
\label{eq:pKac:3}
\int_V \Vert f'(x) \Vert^r  \cdot
  \1_{\{\Vert f(x) \Vert \leq q^{-s}\}} \, dx
= 0.
\end{equation}
Combining Eqs.~\eqref{eq:pKac:2} and~\eqref{eq:pKac:3}, we find
that the equality:
$$q^{sr} \cdot \int_U \Vert f'(x) \Vert^r  \cdot
  \1_{\{\Vert f(x) \Vert \leq q^{-s}\}} \, dx = m$$
holds true when $s$ is sufficiently large. Passing to the limit, we
get the theorem.
\end{proof}

We underline that the compacity assumption in
Theorem~\ref{theo:pKac} cannot be relaxed. For example, taking
simply $U = \Zp \backslash \{0\}$ and $f : x \mapsto x$, one sees
that $f$ have no zero in $U$ while the right hand side of
Eq.~\eqref{eq:pKac} converges to $1$. Roughly speaking, the
integral continues to see the missing zero at the origin, which 
is expected because removing one point from the domain of
integration does not alter the value of the integral.

Similarly, the assumption that the zeros of $f$ are nondegenerate
(\emph{i.e.} that the derivative does not vanish at these points) 
is definitely necessary. For example, if we take the function
$f : \Zp \to \Qp$, $x \mapsto x^2$, a simple calculation shows
that the right hand side of Eq.~\eqref{eq:pKac} converges to 
$\frac p {p+1} < 1$. More generally, one can prove that, if 
$f$ is a polynomial whose roots in $U$ are $a_1, \ldots, a_m$
and have multiplicity $\mu_1, \ldots, \mu_m$ respectively, then
the right hand side of Eq.~\eqref{eq:pKac} converges to:
$$\sum_{i=1}^m \frac{q^{\mu_i} - q^{\mu_i-1}}{q^{\mu_i} - 1}.$$
In other words, a root of multiplicity $\mu$ does not contribute
for $1$ but for $\frac{q^{\mu} - q^{\mu-1}}{q^{\mu} - 1} < 1$.

\subsection{A key computation}
\label{ssec:limit}

Let $K$ be a finite of $F$ of degree $r$ and 
let $U$ be an open subset of $K$. 
We aim at computing the expected value of random variable $Z_{U,n} :
\Omega_n \to \ZZ \cup \{+\infty\}$ defined by:
\begin{align*}
Z_{U,n}(P) 
 & = \card\,\big\{ \, x \in U \text{ s.t. } f(x) = 0 \,\big\} \\
 & = \lim_{s \to \infty} \,
     q^{sr} \cdot \int_U \Vert P'(x) \Vert^r \cdot
     \1_{\{\Vert P(x) \Vert \leq q^{-s}\}} \, dx
\end{align*}
the second equality coming from Theorem~\ref{theo:pKac}. 
For this, roughly speaking, we would like to write down the following
calculation:
\begin{align*}
\EE[Z_{U,n}] = 
\int_{\Omega_n} Z_{U,n}(P) dP
& = \int_{\Omega_n}
    \lim_{s \to \infty} \,
     q^{sr} \cdot \int_U \Vert P'(x) \Vert^r \cdot
     \1_{\{\Vert P(x) \Vert \leq q^{-s}\}} \, dx\, dP \\
& = \int_U\,
    \lim_{s \to \infty} \,
     q^{sr} \cdot \int_{\Omega_n} \Vert P'(x) \Vert^r \cdot
     \1_{\{\Vert P(x) \Vert \leq q^{-s}\}} \, dP\, dx
\end{align*}
and introduce the density function defined by:
\begin{equation}
\label{eq:limit}
x \, \mapsto \, 
  \lim_{s \to \infty} \,
  q^{sr} \cdot \int_{\Omega_n} \Vert P'(x) \Vert^r \cdot
  \1_{\{\Vert P(x) \Vert \leq q^{-s}\}} \, dP.
\end{equation}
However, we have to be a bit more careful because the above limit does 
not behave quite well everywhere: it takes infinite values on certain 
subspaces but it turns out that those parts lead to a finite positive
contribution when we integrate. The next proposition shows that these
issues are somehow localized on strict subfields.

\begin{prop}
\label{prop:limit}
If $n \geq r$ and if $x$ lies in $\OK$ and generates $K$ over $F$,
the limit in Eq.~\eqref{eq:limit} exists and is equal to:
$$\frac{\Vert D_K \Vert}{\card\big(\OK/\OF[x]\big)} \cdot 
  \int_{\Omega_{n-r}} \hspace{-1em} \Vert Q(x) \Vert^r \: dQ$$
where $D_K$ denotes the discriminant of the extension $K/F$.
\end{prop}

\begin{proof}
For simplicity, write:
$$I_s = q^{sr} \cdot
  \int_{\Omega_n} \Vert P'(x) \Vert^r \cdot
  \1_{\{\Vert P(x) \Vert \leq q^{-s}\}} \, dP.$$
Let $Z$ be the minimal monic polynomial of $x$ over $F$. By our
assumptions, $Z$ has degree $r$ and coefficients in $\OF$.
This implies that the map
$\Omega_{n-r} \times \Omega_{r-1} \to \Omega_n$ taking
$(Q, R)$ to $QZ + R$ preserves the measure.
Performing the corresponding change of
variables, we end up with the equality:
\begin{align*}
I_s 
& = q^{sr} \cdot
  \int_{\Omega_{r-1}} \int_{\Omega_{n-r}} \hspace{-1em}
  \Vert Q(x)Z'(x)+R'(x) \Vert^r \cdot
  \1_{\{\Vert R(x) \Vert \leq q^{-s}\}} \, dQ\,dR \\
& = q^{sr} \cdot
  \int_{\Omega_{r-1}} \left(\int_{\Omega_{n-r}} \hspace{-1em}
  \Vert Q(x)Z'(x)+R'(x) \Vert^r \,dQ \right)
  \1_{\{\Vert R(x) \Vert \leq q^{-s}\}} \, dR
\end{align*}
We consider the evaluation morphism $\alpha_x : 
F \otimes_{\OF} \Omega_{r-1} \to K$ taking a polynomial $R$ to 
$R(x)$.
It is $F$-linear and bijective since the domain of $\alpha_x$ is
restricted to polynomials of degree strictly less than $r$. Its
inverse $\alpha_x^{-1}$ is $F$-linear and so, it is continuous.
Thus, there exists a positive constant $\gamma$ such that
$\Vert R(x) \Vert \geq \gamma \cdot \Vert R \Vert$
for all $R \in \Omega_{r-1}$.

Let us assume for a moment that we are given a polynomial $R \in 
\Omega_{r-1}$ such that $\Vert R(x) \Vert \leq q^{-s}$. By what 
precedes, we find that $\Vert R \Vert \leq \gamma^{-1} q^{-s}$, from 
what we further deduce that $\Vert R'(x) \Vert \leq \gamma^{-1} 
q^{-s}$. Since $Z'(x)$ does not vanish, we conclude that $Z'(x)$ 
must divide $R'(x)$ provided that $s$ is large enough. 
One can then perform the change of variables $Q \mapsto Q - 
\frac{R'(x)}{Z'(x)}$ in the inner integral and get:
\begin{align*}
\int_{\Omega_{n-r}} \hspace{-1em} \Vert Q(x)Z'(x)+R'(x) \Vert^r \, dQ 
& = \int_{\Omega_{n-r}} \hspace{-1em} \Vert Q(x)Z'(x) \Vert^r \, dQ \\
& = \Vert Z'(x) \Vert^r \cdot 
    \int_{\Omega_{n-r}} \hspace{-1em}\Vert Q(x) \Vert^r \, dQ.
\end{align*}
We are then left with:
\begin{equation}
\label{eq:Ieps}
I_s = q^{sr} \cdot
\Vert Z'(x) \Vert^r \cdot 
\int_{\Omega_{n-r}} \hspace{-1em} \Vert Q(x) \Vert^r \, dQ\, \cdot 
\int_{\Omega_{r-1}} \hspace{-1em} 
\1_{\{\Vert R(x) \Vert \leq q^{-s}\}} \, dR.
\end{equation}
In order to estimate the last factor, we come back to the evaluation
morphism $\alpha_x$. Since it is a $F$-linear isomorphism, it must 
act on the measures by multiplication by some scalar (namely, its
determinant). In other words, there exists a positive constant 
$\delta$ such that
$\lambda_K(\alpha_x(H)) = \delta \cdot \mu_n(H)$ for all measurable
subset $H$ of $\Omega_{r-1}$. Taking $H = \Omega_{r-1}$, we find
$$\delta
= \lambda_K\big(\alpha_x(\Omega_{r-1})\big)
= \lambda_K\big(\OF[x]\big)
= \frac 1{\card\big(\OK/\OF[x]\big)}.$$
As in the proof of Theorem~\ref{theo:pKac}, we let $B_s$ be the closed 
ball of $K$ of radius $q^{-s}$ centered at~$0$. By definition 
$\alpha_x^{-1}(B_s)$ consists of polynomials $R$ such that $\Vert R(x) 
\Vert \leq q^{-s}$. Moreover, if $s$ is sufficiently large, 
$\alpha_x^{-1}(B_s)$ sits inside $\Omega_{r-1}$, and so:
\begin{align*}
\int_{\Omega_{r-1}} \hspace{-1em} 
\1_{\{\Vert R(x) \Vert \leq q^{-s}\}} \, dR
& = \mu_n\big(\alpha_x^{-1}(B_s)\big) \\
& = \card\big(\OK/\OF[x]\big) \cdot \lambda_K(B_s)
  = \card\big(\OK/\OF[x]\big) \cdot q^{-sr}.
\end{align*}
Plugging this input in Eq.~\eqref{eq:Ieps}, we end up with:
\begin{equation}
\label{eq:Is}
I_s = \Vert Z'(x) \Vert^r \cdot 
\card\big(\OK/\OF[x]\big) \cdot
\int_{\Omega_{n-r}} \hspace{-1em} \Vert Q(x) \Vert^r \, dQ
\end{equation}
when $s$ is sufficiently large. The sequence $(I_s)_{s \geq 0}$
is then eventually constant and converges to the limit given by
the above formula.
In order to conclude the proof of the proposition, it remains to 
relate the norm of $Z'(x)$ with that of the discriminant of $K$.
We endow $K$ with the symmetric $F$-bilinear form $b : K \times K 
\to F$, $(u,v) \mapsto \Tr_{K/F}(uv)$. Let also $\calB_x = (1, x, \ldots, 
x^{r-1})$ be the canonical basis of $\OF[x]$ over $\OF$ and 
set $\calB'_x = \big(\frac {x^{r-1}}{Z'(x)}, 
\frac {x^{r-2}}{Z'(x)}, \ldots, \frac 1 {Z'(x)}\big)$. Both
$\calB_x$ and $\calB'_x$ are $F$-basis of $K$ and it follows
from \cite[\S III.6, Lemma~2]{serre2} that the matrix of $b$ in 
the basis $\calB_x$ and
$\calB'_x$ is lower-triangular with all diagonal entries equal
to~$1$. Its determinant is then~$1$ as well. Performing a change
of basis, we find that:
$$\det \Mat_{\calB_x}(b) = \pm\:N_{K/F}\big(Z'(x)\big)$$
where, by definition, $\Mat_{\calB_x}(b)$ is the matrix of $b$ in 
$\calB_x$ and $N_{K/F}$ is the norm of $K$ over $F$.
We now consider a basis $\calB$ of $\OK$ over $\OF$.
By definition, the discriminant of $K$ is the determinant of
$\Mat_{\calB}(b)$. Therefore, if $P$ denotes the transition matrix 
from $\calB$ to $\calB_x$, we derive from the change-of-basis formula
that
$\det \Mat_{\calB_x}(b) = (\det P)^2 \cdot D_K$ and so:
$$N_{K/F}\big(Z'(x)\big) = \pm\: (\det P)^2 \cdot D_K.$$
On the other hand, noticing that $P$ is also the matrix of
$\alpha_x$ in the canonical basis, we deduce that:
$$\Vert\!\det P \Vert = \lambda_K\big(\OF[x]\big) = 
\frac 1 {\card\big(\OK/\OF[x]\big)}.$$
Combining the two previous equalities, we end up with:
$$\Vert Z'(x) \Vert^r = 
\Vert N_{K/F}\big(Z'(x)\big) \Vert = 
\frac 1 {\card\big(\OK/\OF[x]\big)^2} \cdot \Vert D_K \Vert.$$
Plugging this relation in Eq.~\eqref{eq:Is}, we obtain the proposition.
\end{proof}

\subsection{Construction and properties of the density functions}
\label{ssec:constrho}

In this subsection, we construct the density functions $\rho_{K,n}$
and establish Theorem~\ref{theo:density}.

\begin{deftn}
\label{def:rho}
Let $n$ be a positive integer and $K$ be a finite extension 
of $F$ of degree $r$.

\noindent
For $x \in \OK$, we set:
$$\begin{array}{r@{\hspace{0.5ex}}l@{\qquad}l}
\rho_{K,n}(x) 
 & = \displaystyle
     \frac{\Vert D_K \Vert}{\card\big(\OK/\OF[x]\big)} \cdot
     \int_{\Omega_{n-r}} \hspace{-1em} \Vert Q(x) \Vert^r \: dQ
 & \text{if } n \geq r \text{ and } F[x] = K, \medskip \\
 & = 0
 & \text{otherwise.}
\end{array}$$

\noindent
For $x \in K$, $x \not\in \OK$, we set $\rho_{K,n}(x) = 
\Vert x \Vert^{-2r} \cdot \rho_{K,n}(x^{-1})$.
\end{deftn}

The first part of Definition~\ref{def:rho} is exactly what we expect 
after Proposition~\ref{prop:limit}. As for the second part, it is 
motivated by the observation that the transformation $P(X) \mapsto X^n 
P(X^{-1})$ preserves the measures on $\Omega_n$ and changes a root $x$ 
into $x^{-1}$.
In any case, we notice that, since $K$ is a discrete valuation field, we 
have $x \in \OK$ or $x^{-1} \in \OK$ for all $x \in K$. 
Defintion~\ref{def:rho} then makes sense and leads to a well-defined 
function $\rho_{K,n} : K \to \RR^+$, which is called the \emph{density 
function} on $K$ of degree~$n$.

The rest of this subsection is devoted to the proof of
Theorem~\ref{theo:density}. 
The vanishing property and the invariance under isomorphisms (Statements 
1 and 3 respectively) are clear from the definitions. In what follows, 
we address the other items of Theorem~\ref{theo:density} one by one (in 
a slightly different order).

\subsubsection{Continuity}
\label{sssec:continuity}

The function
$$x \mapsto \int_{\Omega_{n-r}} \hspace{-1em} \Vert Q(x) \Vert^r \: dQ$$
is continuous on $\OK$ as the integrand is obviously bounded, positive 
and continuous on $\Omega_{n-r} \times \OK$.

Let $x \in \OK$ such that $F[x] = K$.
In this case, the index of $\OF[x]$ in $\OK$ is the inverse of the
$p$-adic norm of
the determinant of the matrix $M(x)$ whose $i$-th column is filled
with the coordinates of $x^i$ in a fixed basis of $\OK$.
If $y$ is a second element of $\OK$ which is congruent to $x$ modulo
$p^s$ for some integer $s$, then the matrices $M(x)$ and $M(y)$ are
also congruent modulo $p^s$ and so are their determinants. In
other words, if $\Vert x - y \Vert \leq \varepsilon$ for some 
positive real number $\varepsilon$, then 
$\Vert\!\det M(x) - \det M(y) \Vert \leq \varepsilon$ as well.
It follows that $\Vert\!\det M(x) \Vert = \Vert\!\det M(y) \Vert$,
that is $\card\big(\OK/\OF[x]\big) = \card\big(\OK/\OF[y]\big)$,
as soon as $\varepsilon < \Vert\!\det M(x) \Vert$.
This proves that the function $y \mapsto \card\big(\OK/\OF[y]\big)$
is constant in some neighborhood of $x$; in particular, it is
continuous at~$x$ and so is $\rho_{K,n}$.

We now consider the case where $F[x] \neq K$. Set $L = F[x]$
and let $d$ denote the degree of the extension $L/F$.
We consider a second element $y \in \OK$ such that $F[y] = K$ and 
$x \equiv y \pmod{p^s}$ for some given integer~$s$. We
then have the inclusion $\OF[y] \subset \O_L + p^s \OK$, from what
we derive the estimation:
$$\card\big(\OK/\OF[y]\big) \geq \card\big(\OK/(\O_L + p^s \OK)\big)
= \Vert p \Vert^{s(d-r)} \geq \Vert x - y \Vert^{d-r}.$$
Since $d < r$, we deduce that the quantity
$\card\big(\OK/\OF[y]\big)$ goes to $\infty$ when $y$ approches~$x$,
from what we conclude that $\rho_{K,n}$ is continuous at~$x$.

Finally,
the continuity on $K$ immediately follows from that on $\OK$.

\subsubsection{Transformation under homography}
\label{sssec:homography}

It is enough to prove the transformation formula for the matrices
$$
\left(\begin{matrix} a & 0 \\ 0 & a \end{matrix}\right)
\, (a \in \OF^\times)
\quad ; \quad
\left(\begin{matrix} 1 & a \\ 0 & 1 \end{matrix}\right)
\, (a \in \OF)
\quad ; \quad
\left(\begin{matrix} 0 & 1 \\ 1 & 0 \end{matrix}\right)
$$
since they generate $\GL_2(\OF)$. It is obvious for scalars
matrices. 

For the second family of matrices, we have to check
$\rho_{K,n}(x+a) = \rho_{K,n}(x)$ for $x \in K$ and $a \in
\OF$. Let us assume first that $x \in \OK$.
Then $\OF[x] = \OF[x+a]$. Moreover, the map $\Omega_{n-r} \to 
\Omega_{n-r}$, $Q(X) \mapsto Q(X+a)$ preserves the measure,
implying that:
$$\int_{\Omega_{n-r}} \hspace{-1em} \Vert Q(x) \Vert^r \: dQ
= \int_{\Omega_{n-r}} \hspace{-1em} \Vert Q(x+a) \Vert^r \: dQ$$
We conclude that the equality $\rho_{K,n}(x+a) = \rho_{K,n}(x)$
is correct in this case. 
We assume now that $x \not\in \OK$, \emph{i.e.} $\Vert x \Vert
> 1$. Since $a$ lies in $\OF$, we have $\Vert a \Vert \leq 1$
and so $\Vert x + a \Vert = \Vert x \Vert$. Setting $y = x^{-1}$
we are then reduced to prove that
$\rho_{K,n}(y) = \rho_{K,n}(z)$ with $z = \frac y{1+ay}$.
Observing that $\Vert y \Vert < 1$, we can write a series 
expansion for the inverse of $1 + ay$ and eventually get:
$$z = \sum_{i \geq 0} (-1)^i a^i y^{i+1} 
\in \OF[y].$$
Similarly, we prove that $y \in \OF[z]$ and hence $\OF[y] = \OF[z]$.
We consider the transformation $\tau : \Omega_{n-r} \to \Omega_{n-r}$
defined by $\tau\big(Q(X)\big) = X^{n-r} Q(X^{-1})$. 
It is explicity given by the formula:
$$\tau\big(a_{n-r} X^{n-r} + \cdots + a_1 X + a_0\big) =
a_0 X^{n-r} + a_1 X^{n-r-1} + a_{n-r}$$
and hence preserves the measure. Therefore:
\begin{align*}
\int_{\Omega_{n-r}} \hspace{-1em} \Vert Q(y) \Vert^r \: dQ
 & = \Vert y \Vert^{n-r} \cdot 
     \int_{\Omega_{n-r}} \hspace{-1em} \Vert Q(x) \Vert^r \: dQ \\
 & = \Vert y \Vert^{n-r} \cdot
     \int_{\Omega_{n-r}} \hspace{-1em} \Vert Q(x+a) \Vert^r \: dQ \\
 & = \Vert y \Vert^{n-r} \cdot \Vert x+a \Vert^{n-r} \cdot 
     \int_{\Omega_{n-r}} \hspace{-1em} \Vert Q(z) \Vert^r \: dQ.
\end{align*}
In addition, we remark that $y(x+a) = 1 + ay$ has norm~$1$. We
then end up with:
$$\int_{\Omega_{n-r}} \hspace{-1em} \Vert Q(y) \Vert^r \: dQ
= \int_{\Omega_{n-r}} \hspace{-1em} \Vert Q(z) \Vert^r \: dQ$$
which eventually lefts us with the desired equality $\rho_{K,n}(y) = 
\rho_{K,n}(z)$.

It remains to treat the case of the antidiagonal matrix, which
amounts to proving that $\rho_{K,n}(x^{-1}) = \Vert x \Vert^{2r}
\rho_{K,n}(x)$. It is obvious from the definition when $x \not\in
\OK$ or $x^{-1} \not\in\OK$. We may then assume that $x$ is
invertible in $\OK$, \emph{i.e.} $\Vert x \Vert = 1$. In this
situation, using again that $\tau$ preserves the measure, we find:
$$\int_{\Omega_{n-r}} \hspace{-1em} \Vert Q(x) \Vert^r \: dQ
= \int_{\Omega_{n-r}} \hspace{-1em} \Vert Q(x^{-1}) \Vert^r \: dQ.$$
Let $Z$ be the minimal polynomial of $x$. 
The constant coefficient of $Z$ is the norm of $x$ up to a sign.
Hence it has norm~$1$, which means that it is invertible in $\OF$.
It follows from this observation that $x^{-1}$ can be expressed as
a polynomial in $x$ with coefficients in $\OF$, that is $x^{-1}
\in \OF[x]$. Similarly, one has $x \in \OF[x^{-1}]$. Combining
both results, we get $\OF[x] = \OF[x^{-1}]$, from what we deduce
the desired equality.

\subsubsection{Formulas for extremal degrees}
\label{sssec:formulas}

When $n = r$, the density function $\rho_{r,K}$ is defined by:
$$\rho_{r,K}(x) = 
  \frac{\Vert D_K \Vert}{\card\big(\OK/\OF[x]\big)} \cdot
  \int_{\OF} \Vert t \Vert^r \: dt.$$
Let $\pi$ be a fixed uniformizer of $F$ and set $U_s = \pi^s
\OF^\times$. We observe that $U_s$ consists exactly of elements
of norm $q^{-s}$ and that $\OF \backslash \{0\}$ is the disjoint 
union of the $U_s$. It follows from these observations that:
$$\int_{\OF} \Vert t \Vert^r \: dt
= \sum_{s=0}^\infty \int_{U_s} \Vert t \Vert^r \: dt
= \sum_{s=0}^\infty \lambda_F(U_s) \cdot q^{-sr}.$$
Besides, the measure of $U_s$ can be computed as follows:
$$\lambda_F(U_s) = \lambda_F(\pi^s \OF^\times) = 
\Vert \pi^s \Vert \cdot \lambda_F(\OF^\times) = 
q^{-s} \cdot \left(1 - \frac 1 q\right).$$
We then conclude that:
\begin{equation}
\label{eq:intnorm}
\int_{\OF} \Vert t \Vert^r \: dt
= \left(1 - \frac 1 q\right) \cdot 
  \sum_{s=0}^\infty q^{-sr - s}
= \frac{q^{r+1}-q^r}{q^{r+1} - 1}
\end{equation}
and the claimed formula follows.

We now move to the case $n \geq 2r - 1$. We set $m = n-r$.
Looking at the definition of $\rho_{K,n}$, we realize that it is
enough to prove that:
$$\int_{\Omega_m} \Vert Q(x) \Vert^r \: dQ
= \card\big(\OK/\OF[x]\big) \cdot 
  \int_{\OF[x]} \Vert t \Vert^r \: dt.$$
As in the proof of Proposition~\ref{prop:limit}, let $Z$ denote 
the minimal polynomial of $x$ and consider the measure-preserving
morphism $\Omega_{m-r} \times \Omega_{r-1} \to \Omega_m$ mapping
$(S, T)$ to $SZ + T$. (We agree that $\Omega_{-1} = \{0\}$ when 
$m = r-1$.)
Performing the corresponding change of variables, we obtain:
$$\int_{\Omega_m} \Vert Q(x) \Vert^r \: dQ
= \int_{\Omega_{r-1}} \hspace{-0.5em} \Vert T(x) \Vert^r \: dT.$$
In other words, we may assume that $m = r-1$, \emph{i.e.} $n = 2r-1$.
Following again the proof of Proposition~\ref{prop:limit}, we
consider the evaluation map $\alpha_x : \Omega_{r-1} \to \OK$,
$T(X) \mapsto T(x)$. We have seen that it acts on the measures by
multiplication by $\card\big(\OK/\OF[x]\big)^{-1}$. Therefore, performing
the change of variables $t = \alpha_x(T)$, we obtain:
$$\int_{\Omega_{r-1}} \hspace{-0.5em} \Vert T(x) \Vert^r \: dT
= \card\big(\OK/\OF[x]\big) \cdot 
  \int_{\OF[x]} \Vert t \Vert^r \: dt$$
which concludes the proof.

\subsubsection{Monotony}
\label{sssec:monotony}

From what precedes, it is clear that $\rho_{K,n}(x) = \rho_{K,n+1}(x)$
when $F[x] \neq K$ or $n < r$ or $n \geq 2r - 1$. It is then enough
to prove that $\rho_{K,n}(x) < \rho_{K,n+1}(x)$ in the remaining cases.
Coming back to the definition of the density functions, this amounts
to showing that:
$$\int_{\Omega_m} \Vert Q(x) \Vert^r \: dQ \,<\,
\int_{\Omega_{m+1}} \hspace{-0.5em} \Vert Q(x) \Vert^r \: dQ$$
provided that $0 \leq m < r-1$ and $F[x] = K$.
Let $a \in \OF$.
By compacity, the function $\Omega_m \to \RR$, $Q \mapsto 
\Vert a x^{m+1} + Q(x) \Vert$ attains its minimum. In other 
words, there exists a polynomial $Q_a \in \Omega_m$ such that
$\Vert a x^{m+1} + Q_a(x) \Vert \leq \Vert a x^{m+1} + Q(x) \Vert$
for all $Q \in \Omega_m$. Write $\delta_a = \Vert a x^{m+1} + Q_a(x) 
\Vert$.
Notice that $\delta_a > 0$; indeed, otherwise, $x$ would be annihilated
by a polynomial over $F$ of degree $m+1 < r$, contradicting $F[x] = K$.
Note also that $\delta_a$ depends only on the norm of $a$; in particular
the function $a \mapsto \delta_a$ is measurable.

\begin{lem}
With the above notations, we have:
$$\Vert a x^{m+1} + Q(x) \Vert = 
\max\big( \delta_a,\, \Vert (Q - Q_a)(x) \Vert \big)$$
for all $a \in \OF$ and $Q \in \Omega_m$.
\end{lem}

\begin{proof}
Putting $S = Q - Q_a$, the ultrametric triangle inequality gives:
\begin{equation}
\label{eq:Qa}
\Vert a x^{m+1} + Q(x) \Vert \leq
\max\big( \delta_a,\, \Vert S(x) \Vert \big)
\end{equation}
with equality provided that $\Vert S(x) \Vert \neq \delta_a$. 
We then get the lemma under this additional assumption.
On the contrary, when $\Vert S(x) \Vert = \delta_a$, we derive 
from the definition of $Q_a$ that $\Vert a x^{m+1} + Q(x) \Vert \geq
\delta_a$. We deduce that Eq.~\eqref{eq:Qa} is an equality in this
case as well, which establishes the lemma.
\end{proof}

Separating the leading coefficient in the integral over $\Omega_{m+1}$,
we obtain:
\begin{align*}
\int_{\Omega_{m+1}} \hspace{-0.5em} \Vert Q(x) \Vert^r \: dQ
 & = \int_{\OF} \int_{\Omega_m} \Vert a x^{m+1} + Q(x) \Vert^r \: dQ \: da \\
 & = \int_{\OF} \int_{\Omega_m} \max\big( \delta_a^r,\, 
      \Vert (Q-Q_a)(x) \Vert^r\big) \: dQ \: da \\
 & = \int_{\OF} \int_{\Omega_m} \max\big( \delta_a^r,\, 
      \Vert Q(x) \Vert^r\big) \: dQ \: da
\end{align*}
which eventually shows that:
$$\int_{\Omega_{m+1}} \hspace{-0.5em} \Vert Q(x) \Vert^r \: dQ
\, \geq \,
\int_{\Omega_m} \Vert Q(x) \Vert^r \: dQ$$
with strict inequality provided that the set of pairs $(a,Q) 
\in \OF \times \Omega_m$ 
for which $\Vert Q(x)  \Vert < \delta_a$ has positive measure.
But, the latter property holds always true, being a
consequence of the facts that $\delta_a = \delta_1 > 0$ for all $a 
\in \OF^\times$ and that $\OF^\times$ has positive measure in~$\OF$.
This concludes the proof of Theorem~\ref{theo:density}.

\subsection{From density functions to average number of roots}
\label{ssec:meanZ}

We now focus on Theorem~\ref{theo:mean}.
After Proposition~\ref{prop:limit}, we are encouraged to treat
separately roots lying in strict subextensions of $K$.
We materialize this idea by introducing the notion of new roots.

\begin{deftn}
\label{def:newroot}
Let $K$ be a finite extension of $F$.
An element $x \in \bar F$ is \emph{new} in $K$ if it is in $K$ 
but not in any strict subextension of $K$.
\end{deftn}

Given an open subset $U$ of a finite extension $K$ of $F$ together 
with a positive integer $n$, we define the random variable $Z^\new_{U,f} : 
\Omega_n \to \ZZ$ taking a polynomial of $P$ to the number of roots of 
$P$, which lie in $U$ and are new in $K$.
After what we have achieved so far, one strongly expects the mean value 
of $Z^\new_{U,n}$ to be related to the integral of the expression of the 
limit which appears in Proposition~\ref{prop:limit}.

\begin{theo}
\label{theo:meanZnew}
Let $n$ be a positive integer.
Let $K$ be a finite extension of $F$ and let $U$ be an open subset
of $K$. Then:
\begin{equation}
\label{eq:meanZnew}
\EE\big[Z^\new_{U,n}\big] = \int_U \rho_{K,n}(x) \: dx.
\end{equation}
\end{theo}

\begin{proof}
As usual, let $r$ be the degree of the extension $K/F$.
Let $K^\new$ be the subset of $K$ consisting of elements $x$ for
which $F[x] = K$.
Since $K$ contains only finitely many subextensions and each 
subextension is a closed subspace of $K$, we deduce that $K^\new$ is 
open in $K$. We set $\OK^\new = \OK \cap K^\new$.

To start with, we assume that $U$ is compact and included in
$\OK^\new$. We then
have $Z^\new_{U,n} = Z_{U,n}$. Since a random polynomial has almost
surely only simple roots, we deduce from the
$p$-adic Kac-Rice formula (\emph{cf} Theorem~\ref{theo:pKac}) that:
\begin{align*}
\EE[Z_{U,n}] 
& = \int_{\Omega_n} \lim_{s \to \infty} \,
    \int_U I_s(x, P) \, dx\, dP \\
\text{with} \quad
I_s(x, P) 
& = q^{sr} \cdot \Vert P'(x) \Vert^r \cdot
    \1_{\{\Vert P(x) \Vert \leq q^{-s}\}}.
\end{align*}
It is clear that $I_s(x,P)$ is everywhere nonnegative.
As in the proof of Theorem~\ref{theo:pKac},
let $B_s$ be the closed ball of $K$ of radius $q^{-s}$ and center~$0$.
From \cite[Proposition~3.2]{evans}, we deduce that:
$$\int_{P^{-1}(B_s)} \hspace{-0.2em} \Vert P'(x) \Vert^r \, dx 
\,=\, \int_{B_s} \card P^{-1}(y) \, dy 
\,\leq\, n \lambda_K(B_s) \,=\, n q^{-rs}.$$
for any polynomial $P \in \Omega_n$. Hence:
$$\int_U I_s(x, P) \, dx 
= q^{sr} \int_{U \cap P^{-1}(B_s)} \Vert P'(x) \Vert^r \, dx
\,\leq\, n.$$
We can then apply Lebesgue's dominated convergence theorem and 
get:
$$\EE[Z_{U,n}] 
 = \lim_{s \to \infty} \, \int_{\Omega_n}
    \int_U I_s(x, P) \, dx\, dP
 = \lim_{s \to \infty} \, \int_U
    \int_{\Omega_n} I_s(x, P) \, dP\, dx$$
the second equality coming from Fubini's theorem. We now want to use a 
similar argument to swap the limit on $s$ and the integral over $U$. In 
order to proceed, we fix an element $x \in U$ and denote by $\alpha_x : F 
\otimes_{\OF} \Omega_n \to K$ the evaluation morphism at $x$ already 
considered in the proof of Proposition~\ref{prop:limit}. Noticing that 
$\Vert P'(x) \Vert^r \leq 1$ for all $P \in \Omega_n$, we find:
\begin{align*}
\int_{\Omega_n} I_s(x, P) \, dP
 & \leq q^{sr} \cdot \mu_n\big(\Omega_n \cap \alpha_x^{-1}(B_s)\big) \\
 & \leq q^{sr} \cdot \mu_n\big(\alpha_x^{-1}(B_s)\big) 
   = \card\big(\OK/\OF[x]\big).
\end{align*}
We have seen in \S \ref{sssec:continuity} that the 
function $x \mapsto \card\big(\OK/\OF[x]\big)$ is continuous; hence it is
integrable on the compact set~$U$. The dominated convergence theorem
then again applies and gives:
$$\EE[Z_{U,n}] 
 = \int_U \, \lim_{s \to \infty} 
    \int_{\Omega_n} I_s(x, P) \, dP\, dx
 = \int_U \rho_{K,n}(x) \, dx.$$
Theorem~\ref{theo:meanZnew} is then proved when $U$ is compact and 
included in $\OK^\new$.

One can extend the result to any open subset of $\OK$ using a
standard limit argument. 
Precisely, if $U$ is open in $\OK$, one 
can construct an increasing sequence $(U_m)_{m \geq 0}$ of compact open 
subsets of $\OK^\new$ such that $\bigcup_{m \geq 0} U_m = 
U \cap \OK^\new$.
Applying what we have done before with $U_m$, we find:
\begin{equation}
\label{eq:ZnewUm}
\EE[Z_{U_m,n}] = \int_{U_m} \rho_{K,n}(x) \: dx.
\end{equation}
Moreover the sequence $Z_{U_m,n}$ is nondecreasing and simply
converges to $Z^\new_{U,n}$. By the monotone convergence theorem,
we find that $\EE[Z_{U_m,n}]$ converges to $\EE[Z^\new_{U,n}]$.
Passing to the limit in Eq.~\eqref{eq:ZnewUm}, we obtain the theorem
for~$U$.

For a general $U$, we write $U = U_0 \sqcup U_\infty$ with $U_0 =
U \cap \OK$ and $U_\infty = U \backslash U_0$. Since both sides
of Eq.~\eqref{eq:meanZnew} are additive with respect to $U$, it is
enough to prove the theorem for $U_\infty$. For this, as in 
\S \ref{sssec:homography}, we consider
the measure-preserving map $\tau : \Omega_n \to \Omega_n$ defined by:
$$\tau\big(a_n X^n + \cdots + a_1 X + a_0\big) =
a_0 X^n + \cdots + a_{n-1} X + a_n.$$
An element $x \in K$ is a root of
a polynomial $P$ if and only if $x^{-1}$ is a root of $\tau(P)$.
It is also obvious that $x \in K^\new$ if and only if $x^{-1} \in
K^\new$. Thus we get
$\EE\big[Z^\new_{U_\infty,n}\big] = \EE\big[Z^\new_{V_\infty,n}\big]$
where $V_\infty$ is the image of $U_\infty$ under the map $x 
\mapsto x^{-1}$. Besides $V_\infty \subset \OK$; we can then apply
the theorem with $V_\infty$ and conclude that:
$$\EE\big[Z^\new_{U_\infty,n}\big]
 = \int_{V_\infty} \rho_{K,n}(x) \: dx 
 = \int_{U_\infty} \rho_{K,n}(x^{-1}) \cdot \Vert x \Vert^{-2r} \: dx
 = \int_{U_\infty} \rho_{K,n}(x) \: dx$$
which finally proves the theorem in all cases.
\end{proof}

We conclude this section by explaining how Theorem~\ref{theo:mean}
can be derived from Theorem~\ref{theo:meanZnew}. It is actually quite
easy once we have notice that, given a finite extension $E$ of $F$, 
any root $x \in K$ of a polynomial $P \in 
\Omega_n$ is new in a unique subextension $K'$ of $E$, namely 
$K' = F[x]$. 
Hence, if $U$ is an open subset of $E$, one has:
$$Z_{U,n} = \sum_{K' \subset K} Z^\new_{U \cap K',n}$$
and Theorem~\ref{theo:mean} follows by additivity of the mean.

\section{Examples and closed formulas}
\label{sec:closed}

As we have seen in Section~\ref{sec:density}, the distribution of roots 
in~$\bar F$ of a random polynomial of degree~$n$ over $\OF$ are governed 
by the density functions $\rho_{K,n}$. However, it is not clear so far 
how useful could be this result for deriving explicit formulas, given 
that the density functions are defined by somehow intricated integral 
expressions which do not look easily tractable at first glance.

The aim of this section is to get more familiar with those integrals 
and fully compute them in certain simple cases. The case of quadratic 
extensions will be covered in full generality in \S 
\ref{ssec:quadratic}, culminating with a proof of 
Theorem~\ref{theo:quadratic}. Partial results in the 
case of prime degree extensions and unramified extensions will be given 
in \S \ref{ssec:primedeg} and \S \ref{ssec:unram} respectively.

Before getting to the heart of the matter and dealing with extensions,
it is important to elucidate the case of the ground field $F$ itself.
This base case is addressed by the following proposition.

\begin{prop}
\label{prop:rhoF}
For all positive integer $n$ and all $x \in F$, we have:
$$\begin{array}{r@{\hspace{0.5ex}}l@{\qquad}l}
\rho_{F,n}(x) 
& = \displaystyle \frac q{q+1} 
  & \text{if } x \in \OF \medskip \\
& = \displaystyle \frac q{q+1} \cdot \Vert x \Vert^{-2}
  & \text{otherwise.}
\end{array}$$
\end{prop}

\begin{proof}
It follows from Theorem~\ref{theo:density} that $\rho_{F,n}$ does
not depend on~$n$ provided that $n \geq 1$.
The values of $\rho_{F,n}$ on $\OF$ aren the given by the explicit
formula for $\rho_{F,1}(x)$ which appears again in 
Theorem~\ref{theo:density}.
The values on $F\backslash\OF$ are deduced from that on $\OF$,
coming back to Definition~\ref{def:rho}.
\end{proof}

Integrating over $F$, we recover (one more time) the fact that a
random polynomial over $\OF$ has exactly one root on average in~$F$.
In a similar fashion, if $U$ is an open subset of $\OF$, we find
$\EE[Z_{U,n}] = \frac q{q+1} \cdot \lambda(U)$.
In particular, when $F = \Qp$ and $U = \Zp$, we recover Evans' 
theorem~\cite{evans} which states that a random polynomial over the
$p$-adics has $\frac p{p+1}$ roots in $\ZZ_p$ on average.

\subsection{Quadratic extensions}
\label{ssec:quadratic}

Throughout this subsection, we fix a quadratic extension $K$
of $F$ together with an integer $n \geq 1$. We aim at finding a 
closed expression for $\rho_{K,n}(x)$ for $x \in \OK$.
Our starting point is Theorem~\ref{theo:density} which provides
us with the following expression:
\begin{itemize}
\item if $n = 2$, then
$\displaystyle \rho_{K,2}(x) = 
   \Vert D_K \Vert \cdot 
   \frac{1}{\card\big(\OK/\OF[x]\big)} \cdot
   \frac{q^2}{q^2 + q + 1}$,
\item if $n \geq 3$, then
$\displaystyle \rho_{K,n}(x) = 
   \Vert D_K \Vert \cdot 
   \int_{\OF[x]} \Vert t \Vert^2 \: dt$.
\end{itemize}
We now distinguish between two cases depending on the fact that
$K/F$ is ramified or not.

\begin{prop}
\label{prop:rhounram}
Let $K$ be the unramified quadratic extension of $F$.
Then, for all $x \in \OK$, we have:
$$\begin{array}{lr@{\hspace{0.5ex}}l}
& \rho_{K,2}(x) & \displaystyle
  = \frac{q^2}{q^2 + q + 1} \cdot \dist(x, F),
\medskip \\
\text{for } n \geq 3,
& \rho_{K,n}(x) & \displaystyle
  = \frac{q^2}{q^2 + q + 1} \cdot \dist(x, F)
  + \frac{q^3}{(q^2 + 1) \: (q^2 + q + 1)} \cdot \dist(x,F)^4
\end{array}$$
where $\dist(x, F)$ denotes the distance from $x$ to $F$.
\end{prop}

\begin{proof}
Since $K/F$ is unramified, its discriminant has norm~$1$ 
and thus does not contribute.
We fix an element $\zeta \in
\OK$ with norm~$1$ such that $\OK = \OF[\zeta]$. The family
$\calB = (1, \zeta)$ is a basis of $\OK$ over $\OF$.
Let $x \in \OK$, $x \not\in \OF$ and
write $x = a + b \zeta$ with $a, b \in \OF$, $b \neq 0$.
We have $\OF[x] = \OF[b \zeta]$, which shows that a basis of
$\OF[x]$ over $\OF$ is simply $\calB_x = (1, b\zeta)$.
Comparing $\calB$ and $\calB_x$, we find
$\card\big(\OK/\OF[x]\big) = \card\big(\OF/b\OF\big) =
\Vert b \Vert^{-1}$.
Observing in addition 
that $a$ is the closest element of $x$ in $F$, we can reinterpret
the norm of $b$ as the distance of $x$ to~$F$. Putting all
ingredients together, we obtain the announced formula for
$\rho_{K,2}(x)$.
This formula has been established when $x \not\in \OF$; however, it 
obviously also holds true when $x \in \OF$ since $\rho_{K,2}(x)$ 
vanishes in this case by definition.

We now move to the computation of $\rho_{K,n}$ for $n \geq 3$.
We continue to consider an element $x \in \OK$, $x \not\in \OF$
and to write $x = a + b \zeta$ with $a, b \in \OF$, $b \neq 0$.
Performing a change of variables or, more simply, coming back to
Definition~\ref{def:rho}, we have:
$$\rho_{K,n}(x)
 = \frac 1 {\card\big(\OK/\OF[x]\big)} \cdot
   \int_{\OF^2} \Vert u + vx \Vert^2 \: du \: dv
 = \Vert b \Vert \cdot
   \int_{\OF^2} \Vert u + vx \Vert^2 \: du \: dv.$$
Replacing $u$ by $u - va$, this reduces to:
$$\rho_{K,n}(x)
 = \Vert b \Vert \cdot 
   \int_{\OF^2} \Vert u + vb\zeta \Vert^2 \: du \: dv
 = \Vert b \Vert \cdot 
   \int_{\OF^2} \max\big(\Vert u \Vert^2, \Vert vb \Vert^2 \big) \: du \: dv.$$
As in \S \ref{sssec:formulas}, we decompose $\OF$ as the disjoint
union $\OF = \{0\} \sqcup \bigcup_{s \geq 0} U_s$ where $U_s$ consists 
of elements of norm~$q^{-s}$. Decomposing the integral accordingly and 
writing $\Vert b \Vert = q^{-v}$, we find:
$$\rho_{K,n}(x)
 = \left( 1 - \frac 1 q\right)^2 q^{-v} \cdot
   \sum_{s=0}^\infty \sum_{t=0}^\infty q^{-s-t-2\min(s, t+v)}.$$
Computing the latter double sum is painful but straightforward.
We split the domain of summation into three regions, namely:
\begin{align*}
D_1 & 
 = \big\{ \, (s,t) \in \ZZ_{\geq 0}^2 
   \quad \text{such that} \quad
   s < v \,\big\}, \\
D_2 & 
 = \big\{ \, (s,t) \in \ZZ_{\geq 0}^2
   \quad \text{such that} \quad
   v \leq s \leq t + v \,\big\}, \\
D_3 & 
 = \big\{ \, (s,t) \in \ZZ_{\geq 0}^2
   \quad \text{such that} \quad
   s > t + v \,\big\}.
\end{align*}
We then compute the double sum separately on each domain:
\begin{itemize}
\item over $D_1$: $\displaystyle
  \sum_{s=0}^{v-1} \sum_{t=0}^\infty q^{-3s-t}
= \frac{q^4}{(q-1)(q^3-1)}\cdot (1 - q^{-3v})$,
\item over $D_2$: $\displaystyle
  \sum_{s=v}^\infty \sum_{t=s-v}^\infty q^{-3s-t}
= q^{-3v} \sum_{s=0}^\infty \sum_{t=s}^\infty q^{-3s-t}
= \frac{q^5}{(q-1)(q^4-1)} \cdot q^{-3v}$,
\item over $D_3$: $\displaystyle
  \sum_{t=0}^\infty \sum_{s=t+v+1}^\infty q^{-s-3t-2v}
= q^{-3v-1} \sum_{t=0}^\infty \sum_{s=t}^\infty q^{-s-3t}
= \frac{q^4}{(q-1)(q^4-1)} \cdot q^{-3v}$.
\end{itemize}
Summing up all contributions and noticing $q^{-v} =
\Vert b \Vert = \dist(x,F)$, we finally find the expression
given in the statement of the proposition.
As previously, we notice that this formula continues to
hold when $x \in \OF$ given that $\rho_{K,n}(x)$ vanishes by
definition in this case.
\end{proof}

\begin{prop}
\label{prop:rhoram}
Let $K$ be a totally ramified quadratic extension of $F$.
Then, for all $x \in \OK$, we have:
$$\begin{array}{lr@{\hspace{0.5ex}}l}
& \displaystyle \frac{\rho_{K,2}(x)}{\Vert D_K \Vert} & \displaystyle
  = \frac{q^{3/2}}{q^2 + q + 1} \cdot \dist(x, F),
\medskip \\
\text{for } n \geq 3,
& \displaystyle \frac{\rho_{K,n}(x)}{\Vert D_K \Vert} & \displaystyle
  = \frac{q^{3/2}}{q^2 + q + 1} \cdot \dist(x, F)
  + \frac{1}{q \:(q + 1) \: (q^2 + q + 1)} \cdot \dist(x,F)^4
\end{array}$$
where $\dist(x, F)$ denotes the distance from $x$ to $F$.
\end{prop}

\begin{proof}
The proof is quite similar to that of 
Proposition~\ref{prop:rhounram}, so we only sketch the argument.
We fix a uniformizer $\pi$ of $K$. The family $(1,\pi)$ forms a 
basis of $\OK$ over $\OF$.
Let $x \in \OK$, $x \not\in \OF$ and write $x = a + b\pi$ with 
$a, b \in \OF$, $b \neq 0$. From the 
fact that $(1, b\pi)$ is a $\OF$-basis of $\OF[x]$, we deduce that the 
cardinality of $\OK/\OF[x]$ is $\Vert b \Vert^{-1}$. Noticing that
$\dist(x, F) = \Vert b \pi \Vert = \sqrt{q} \cdot \Vert b \Vert$,
we get the announced formula for $\rho_{K,2}(x)$.

When $n \geq 3$, we start with the formula:
$$\rho_{K,n}(x)
 = \frac {\Vert D_K \Vert}{\card\big(\OK/\OF[x]\big)} \cdot
   \int_{\OF^2} \Vert u + vx \Vert^2 \: du \: dv.$$
Decomposing the integral into slices where $\Vert u \Vert$ and
$\Vert v \Vert$ are constant, we obtain:
$$\frac{\rho_{K,n}(x)}{\Vert D_K \Vert}
 = \left( 1 - \frac 1 q\right)^2 q^{-v} \cdot
   \sum_{s=0}^\infty \sum_{t=0}^\infty q^{-s-t-\min(2s, 2t+2v+1)}$$
where $v$ is defined by $\Vert b \Vert = q^{-v}$. Finally, splitting 
the previous double sum into three parts exactly as we did in the 
unramified case, we end up after some calculations with the
formula displayed in the statement of the proposition.
\end{proof}

Theorem~\ref{theo:quadratic} can be deduced from 
Propositions~\ref{prop:rhounram} and~\ref{prop:rhoram} by integrating
over~$K$. For this, the first ingredient is the observation that the
transformation
law of Theorem~\ref{theo:density}.3 permits to relate the integral
of $\rho_{K,n}$ outside $\OK$ to its integral over another domain
sitting inside $\OK$. Precisely, applying it with the homography
$h : x \mapsto x^{-1}$, we get
$\rho_{K,n}\big(h(x)\big) = \Vert x \Vert^4 \cdot \rho_{K,n}(x)$.
Noticing in addition that $h$ maps bijectively $K\backslash\OK$
to the maximal ideal $\mK$ of $\OK$, we obtain:
$$\int_{K\backslash\OK} \hspace{-0.4em} \rho_{K,n}(x)\: dx 
 = \int_{\mK} \rho_{K,n}\big(h(x)\big) {\cdot} \Vert h'(x) \Vert^2 \: dx 
 = \int_{\mK} \rho_{K,n}(x) \: dx.$$
Summing up the contributions over $\OK$ and $K\backslash\OK$,
we end up with:
$$\int_K \, \rho_{K,n}(x)\: dx = 
\int_{\OK} \rho_{K,n}(x)\: dx \, + \, \int_{\mK} \rho_{K,n}(x)\: dx.$$
Theorem~\ref{theo:quadratic} now easily follows from the next lemma.

\begin{lem}
Let $d$ be a positive integer and set 
$\displaystyle
 \alpha_d = \frac{q-1}{q^{d+1} - 1} = \frac 1{q^d + q^{d-1} \cdots + 1}$.
\begin{enumerate}[(i)]
\item If $K/F$ is unramified, we have:
$$\int_{\OK} \dist(x,F)^d \: dx = q^d \alpha_d 
\quad \text{and} \quad
\int_{\mK} \dist(x,F)^d \: dx = q^{-2} \alpha_d.$$

\item If $K/F$ is totally ramified, we have:
$$\int_{\OK} \dist(x,F)^d \: dx = q^{d/2} \alpha_d 
\quad \text{and} \quad
\int_{\mK} \dist(x,F)^d \: dx = q^{d/2\,-\,1} \alpha_d.$$
\end{enumerate}
\end{lem}

\begin{proof}
We first assume that $K/F$ is unramified. Writing $\OK = \OF[\zeta]$
as in the proof of Proposition~\ref{prop:rhounram}, we find:
$$\int_{\OK} \dist(x,F)^d \: dx 
 \,=\, \int_{\OF} \Vert b \Vert^d \: db = q^d \alpha_d$$
the last equality being nothing but Eq.~\eqref{eq:intnorm} 
established in \S \ref{sssec:formulas}.
Similarly noticing that $x = a + b \zeta$ lies in $\OK$ if and only
if both $a$ and $b$ falls in the maximal ideal $\mF$ of $\OF$, we
obtain:
$$\int_{\mK} \dist(x,F)^d \: dx 
 \,=\, \lambda(\mF) \cdot \int_{\mF} \Vert b \Vert^d \: db
 \,=\, q^{-1} \int_{\mF} \Vert b \Vert^d \: db.$$
Fixing a uniformizer $\pi_F$ of $F$ and performing the change of
variables $b = \pi_F t$, we finally get:
$$\int_{\mK} \dist(x,F)^d \: dx 
 \,=\, q^{-d-2} \int_{\OF} \Vert t \Vert^d \: dt = q^{-2} \alpha_d.$$

The argument in the totally ramified case is similar. We pick a
uniformizer $\pi$ of $K$ and writing $\OK = \OF[\pi]$, we find:
$$\int_{\OK} \dist(x,F)^d \: dx 
 \,=\, \int_{\OF} \Vert b \pi \Vert^d \: db 
 \,=\, \Vert \pi \Vert^d \cdot q^d \alpha_d
 \,=\, q^{d/2} \alpha_d.$$
For the integral over $\mK$, we observe that $a + b\pi \in \OK$
if and only if $a \in \mF$. Therefore:
$$\int_{\mK} \dist(x,F)^d \: dx 
 \,=\, \lambda(\mF) \cdot \int_{\OF} \Vert b \pi \Vert^d \: db 
 \,=\, q^{d/2\,-\,1} \alpha_d$$
which concludes the proof.
\end{proof}

\subsection{Prime degree extensions}
\label{ssec:primedeg}

The strategy we have presented above in the case of quadratic
extensions actually extends to all extensions of prime degree,
the crucial point being that $K/F$ does not admit any nontrivial
subextension. Nonetheless, in this generality, the computations 
become much longer and painful although they remain feasible in 
theory.
The case of polynomials of minimal degree remain however reasonable.

\begin{prop}
\label{prop:rhor}
Let $r$ be a prime number and let $K$ be an extension of $F$
of degree $r$.
\begin{enumerate}[(i)]
\item If $K/F$ is unramified then, for all $x \in \OK$, we have:
$$\rho_{r,K}(x) = \frac{q^{r+1}-q^r}{q^{r+1}-1}
  \cdot \dist(x,F)^{r(r-1)/2}.$$
\item If $K/F$ is totally ramified then, for all $x \in \OK$, we have:
$$\rho_{r,K}(x) = \Vert D_K \Vert \cdot
  \frac{q^{(r+3)/2}-q^{(r+1)/2}}{q^{r+1}-1} 
  \cdot \dist(x,F)^{r(r-1)/2}.$$
\end{enumerate}
\end{prop}

\begin{proof}
We assume first that $K/F$ is unramified. Let $x \in \OK$, $x \not
\in \OF$. By compacity, there exists $a \in \OF$ such that $\Vert x 
- a \Vert = \dist(x, F)$. We pick such an element~$a$ and write
$x - a = \pi^v \zeta$ where $\pi$ is a fixed uniformizer of $F$ and
$\zeta \in \OK$ has norm~$1$. Let $k_F$ and $k_K$ denote the residue
fields of $F$ and $K$ respectively and let $\bar\zeta$ be the image
of $\zeta$ in $k_K$. We claim that $\bar\zeta \not\in k_F$; indeed,
otherwise, there would exist $b \in \OF$ with $b \equiv \zeta \pmod
\pi$ and so $\Vert x - (a + \pi^v b) \Vert \leq q^{-v-1} < q^{-v} = 
\Vert x - a \Vert$, contradicting the minimality property of~$a$.
Given that the extension $k_K/k_F$ has prime degree, we deduce 
that $k_K = k_F[\bar\zeta]$ and, consequently, that $\OK = \OF[\zeta]$.
The family $(1, \zeta, \ldots, \zeta^{r-1})$ is then a basis 
of $\OK$ over $\OF$, while $(1, \pi^v \zeta, \ldots,
(\pi^v \zeta)^{r-1})$ is a basis of $\OF[x]$ over $\OF$. This
shows that:
$$\card\big(\OK/\OF[x]\big) = q^{v + 2v + \cdots + (r-1)v}
 = q^{vr(r-1)/2} = \dist(x,F)^{-r(r-1)/2}$$
from what we deduce the claimed formula for $\rho_{r,K}(x)$.

We now move to the totally ramified case.
Let $\pi_F$ (resp. $\pi_K$) denote a fixed uniformizer of $F$ (resp.
$K$). Then $\OK = \OF[\pi_K]$ and the
family $\calB = (1, \pi_K, \ldots, \pi_K^{r-1})$ is a $\OF$-basis of
$\OK$.
Let $x \in \OK$ and let $a \in \OF$ such that $\Vert x - a \Vert =
\dist(x,F)$. Write $x - a = \pi_K^v u$ with $v \in \ZZ_{\geq 0}$ and
$u \in \OK^\times$. The minimality condition in the definition of~$a$
implies that $v \not\equiv 0 \pmod r$. Besides the family $\calB_x =
(1, \pi_K^v u, \ldots, (\pi_K^v u)^{r-1})$ is a $\OF$-basis of $\OF[x]$.
For $i \in \{0, \ldots, r{-}1\}$, we decompose $(\pi_K^v u)^i$ in the
basis $\calB$, \emph{i.e.} we write:
\begin{equation}
\label{eq:onbasisB}
(\pi_K^v u)^i = \sum_{j=0}^{r-1} a_{ij} \pi_K^j
\end{equation}
with $a_{ij} \in \OF$.
Set $c = q^{1/r}$. Taking norms in Eq.~\eqref{eq:onbasisB}, we
get $c^{-vi} = \max_{1 \leq j < r} \big(\Vert a_{ij} \Vert {\cdot} c^{-j}\big)$.
In other words, $\Vert a_{ij} \Vert \leq c^{j-vi}$ for all $j$ and
the equality holds for at least one index~$j$. On the other hand,
the inequality is certainly strict as soon as $r$ does not divides 
$j-vi$ because $\Vert a_{ij} \Vert$ is a negative power of $q =
c^r$. Therefore, if $j_i$ denotes the remainder in the division
of $vi$ by $r$, we conclude that
$\Vert a_{ij} \Vert \leq c^{j-vi}$
for all $j$ with equality if and only if $j = j_i$.
Let $A$ be the change-of-basis matrix from $\calB$ to $\calB_x$.
By definition, its entries are exactly the $a_{ij}$'s. 
Let $P$ be the permutation matrix associated to $i \mapsto j_i$
and define:
$$B = \left(\begin{matrix}
\pi_F^{-v_0} \\
& \ddots \\
& & \pi_F^{-v_{r-1}}
\end{matrix} \right) \cdot P^{-1} \cdot A
\qquad \text{with} \quad
v_i = \frac{vi - j_i}r.$$
The estimations on $\Vert a_{ij} \Vert$ we have obtained previously
shows that $B$ has entries in $\OF$ and that $B \text{ mod } \pi_F$
is lower-triangular with nonzero diagonal entries.
Hence $B$ is invertible over $\OF$, which gives $\Vert\!\det B \Vert 
= 1$. Since $P$ is clearly invertible as well, we conclude that:
$$\card\big(\OK/\OF[x]\big) = \Vert\!\det A \Vert^{-1} 
 = q^{v_0 + \cdots + v_r} = q^{(v-1)(r-1)/2}.$$
Remembering that $\dist(x,F) = \Vert x-a \Vert = \Vert \pi_K^v u \Vert
= q^{-v/r}$, we finally find the formula given in the proposition.
\end{proof}

Extending Proposition~\ref{prop:rhor} to larger degrees does not
require more conceptual arguments but leads to much more laborious
calculations. Precisely, it follows from Definition~\ref{def:rho}
that, for $n \geq r$ and $x \in \OK \backslash \OF$, one has:
$$\rho_{n,K}(x) 
 = \frac{\Vert D_K \Vert}{\card\big(\OK/\OF[x]\big)} \cdot
  \int_{\OF^{m+1}} \max\big( 
  \Vert u_0 \Vert^r, \,
  \delta^r \Vert u_1 \Vert^r,\,
  \ldots, \,
  \delta^{rm} \Vert u_m \Vert^r \big) \: du_0 \cdots du_m$$
where $\delta = \dist(x,F)$ and $m = \min(r{-}1, n{-}r)$.
Besides, after the proof of Proposition~\ref{prop:rhor}, we know
an explicit formula for $\card\big(\OK/\OF[x]\big)$.
The computation of the integral can be carried out by splitting the
domain of integration, namely $\OF^{m+1}$, into $m{+}1$ subdomains depending on 
the index at which the maximum is reached. Computing separately all 
contributions and summing them up, we can conclude. As far as we tried, 
it seems that the final formula does not take a simple form in full
generality.

\subsection{Unramified extensions}
\label{ssec:unram}

For general unramified extensions~$K$, it is possible to compute 
at a lower cost some values of $\rho_{K,n}$ (for any~$n$).

\begin{prop}
\label{prop:unram}
Let $r$ be a positive integer and let $K$ be the unramified
extension of $F$ of degree~$r$. Let $x \in \OK$.
We assume that $\OF[x] = \OK$. Then:
$$\rho_{n,K}(x) = \frac{q^{n'+1} - q^r}{q^{n'+1} - 1}
\qquad \text{with} \quad n' = \min(n, 2r{-}1)$$
for all $n \geq r$.
\end{prop}

\begin{proof}
Set $m = n'-r$.
Let $k_F$ (resp. $k_K$) denote the residue field of $F$ (resp.
$K$). Let $\xi \in k_K$ be the image of $x$. Our assumption
implies that $k_F[\xi] = k_K$. Therefore, if
an expression of the form $u_0 + u_1 x + \cdots + u_m x^m$ with 
$u_i \in \OF$ vanishes in $k_K$, then all $u_i$ have to vanish
in $k_F$. We deduce from this property that:
$$\Vert u_0 + u_1 x + \cdots + u_m x^m \Vert = 
\max\big(\Vert u_0 \Vert, \, \Vert u_1 \Vert,\,
  \ldots, \, \Vert u_m \Vert \big)$$
for all $u_0, \ldots, u_m \in \OF$. It then follows from
Definition~\ref{def:rho} and our assumptions that:
$$\rho_{n,K}(x) = 
  \int_{\OF^{m+1}} \max\big( 
  \Vert u_0 \Vert^r, \,
  \Vert u_1 \Vert^r,\,
  \ldots, \,
  \Vert u_m \Vert^r \big) \: du_0 \cdots du_m.$$
Contrarily to what we said at the end of \S \ref{ssec:primedeg},
it turns out that the latter integral can be easily computed.
Indeed, let us fix a uniformizer $\pi$ of $F$ and, for each 
nonnegative integer
$s$, set $U_s = (\pi^s \OF)^{m+1}$. The $U_s$'s then form a
decreasing sequence of open subsets of $\OF^{m+1}$. Moreover, it
is easy to check that $U_s$ is exactly the domain on which the
integrand $\max\big(\Vert u_0 \Vert^r, \, \ldots, \,
\Vert u_m \Vert^r \big)$ is less than or equal to $q^{-sr}$.
We deduce from this that:
$$\rho_{n,K}(x) = \sum_{s=0}^\infty 
\big(\lambda_F^{\otimes{m+1}}(U_s) -
     \lambda_F^{\otimes{m+1}}(U_{s+1})\big) \cdot q^{-sr}.$$
Observing that $\lambda_F^{\otimes{m+1}}(U_s) = \lambda_F(\pi^s
\OF)^{m+1} = q^{-s(m+1)}$, we end up with:
$$\rho_{n,K}(x) = 
\big( 1 - q^{-m-1}\big) \cdot
\sum_{s=0}^\infty q^{-s(m+1+r)} = 
\frac{q^{m+1+r} - q^r}{q^{m+1+r} - 1}$$
which is what we wanted to prove.
\end{proof}

\subsection{Numerical simulations}
\label{ssec:simulations}

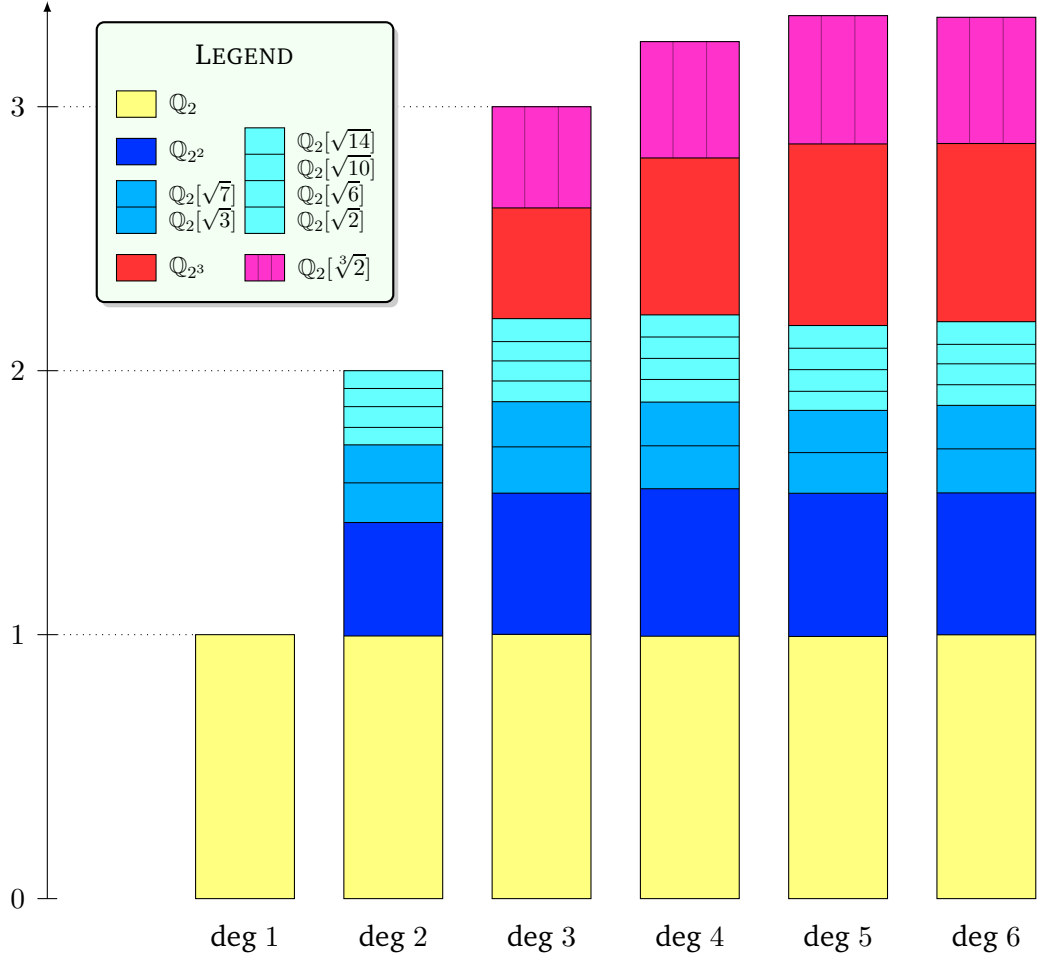
\begin{figure}[t]
\hfill%
\input{stats2}
\hfill\null
\caption{Average number of roots of a polynomial in various extensions. 
Sample of $500,\!000$ polynomials over $\ZZ_2$ picked uniformly at random.}
\label{fig:stats2}
\end{figure}

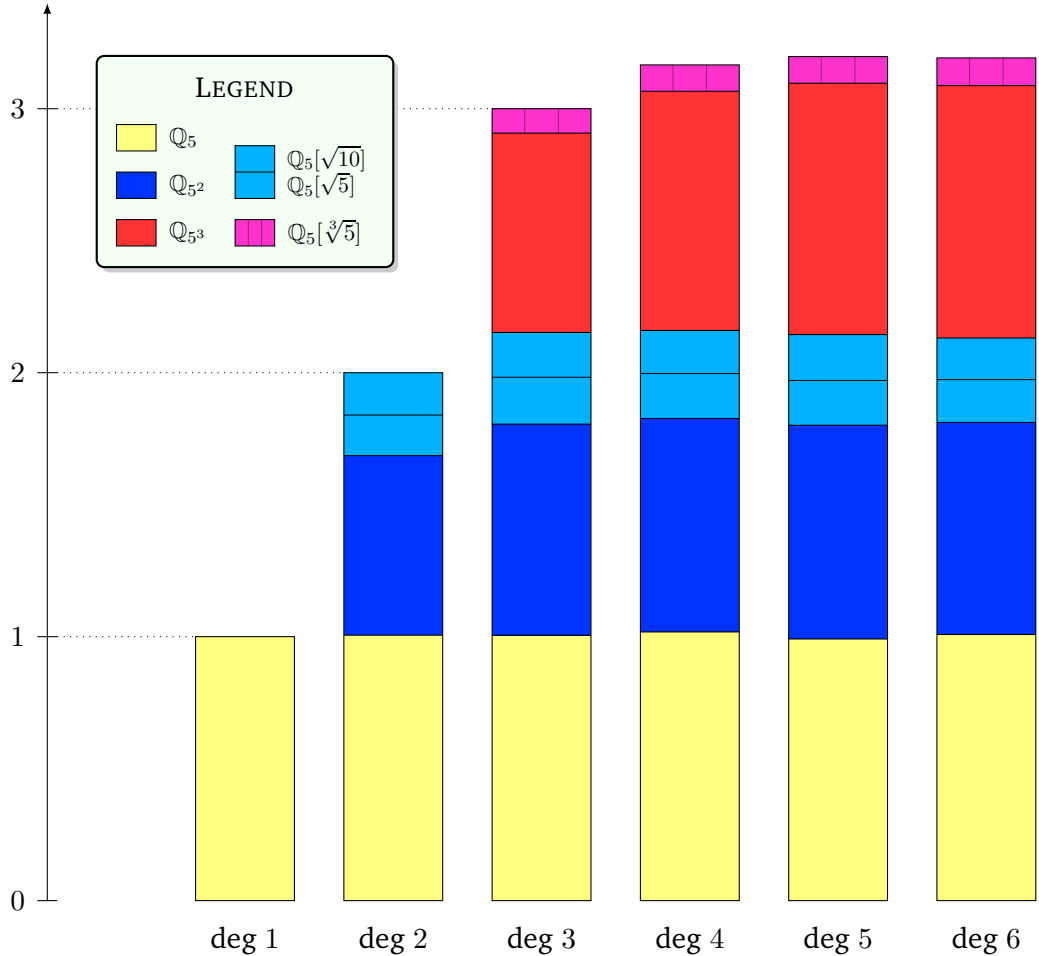
\begin{figure}[t]
\hfill%
\input{stats5}
\hfill\null
\caption{Average number of roots of a polynomial in various extensions. 
Sample of $500,\!000$ polynomials over $\ZZ_5$ picked uniformly at random.}
\label{fig:stats5}
\end{figure}

We have conducted several numerical experiments illustrating the 
theoretical results obtained in \S \ref{sec:density} 
and \S \ref{sec:closed}. First of all, for $p \in \{2, 5\}$,
we have picked a sample of $500,000$ random polynomials over
$\ZZ_p$ of degree up to~$5$ and count the number of (new) roots
these polynomials have in $\Qp$ and all its extensions of degree~$2$
and~$3$. The results are reported on Fig.~\ref{fig:stats2} for
$p = 2$ and on Fig.~\ref{fig:stats5} for $p = 5$.
In both cases, we observe that the empirical average number of
roots in~$\QQ_p$ is~$1$ as predicted by Proposition~\ref{prop:rhoF}.

We can also check that the number of new roots heavily depends on
the discriminant of the extension.
When $p = 5$, for example, the discrimant of $\QQ_{5^2}$ has norm~$1$ 
whereas the discrimant of the two other quadratic extensions, namely
$\QQ_5[\sqrt 5]$ and $\QQ_5[\sqrt{10}]$, has norm $1/5$; looking at
the picture of Fig.~\ref{fig:stats5}, we see that the height of
the dark blue area is roughly $5$ times larger that the height of
the light blue one. Similarly, when $p = 2$, the discriminant of
$\QQ_4$ has norm $1$, the discrimiant of $\QQ_2[\sqrt 3]$ and 
$\QQ_2[\sqrt 7]$ has norm $1/4$ and that of the four remaining
quadratic extensions has norm $1/8$. Again, we way check on
Fig.~\ref{fig:stats2} that
one has approximately the same ratios for the heights of the
corresponding areas.

Another property we can visualize on Fig.~\ref{fig:stats2} 
and Fig.~\ref{fig:stats5} is the monotony statement of 
Theorem~\ref{theo:density}. Indeed, we see that the heights of 
the areas tend to slightly increase with the degree until they 
stabilize at degree~$3$ for quadratic extensions and degree~$5$ 
for cubic extensions.

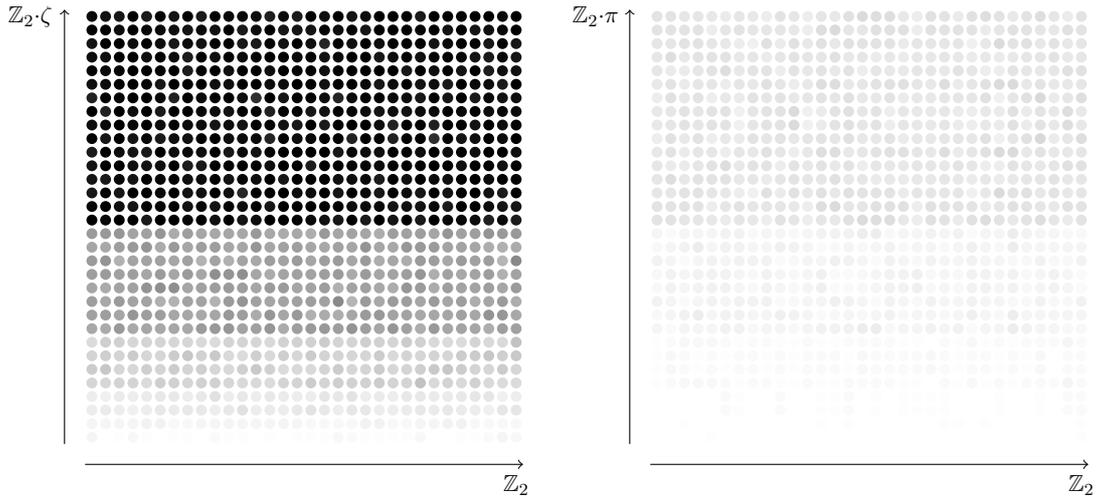
\begin{figure}[t]
\hfill%
\input{statsunram}
\hfill%
\input{statsram}
\hfill\null
\caption{Repartition of new roots of a polynomial in $\ZZ_4$ (on the 
left) and in $\ZZ_2[\pi]$ with $\pi^2 + 2\pi - 2 = 0$ (on the right).
Sample of $500,\!000$ polynomials over $\ZZ_2$ picked uniformly at random.}
\label{fig:statsext}
\end{figure}

In a slightly different direction, Fig.~\ref{fig:statsext} shows the 
empiric repartition of new roots of a sample of $500,000$ random 
polynomials of degree~$5$ over $\ZZ_2$ in the ring of integers of
two quadratic extensions, 
namely $\ZZ_4$, presented as $\ZZ_2[\zeta]$ with $\zeta^2 + \zeta + 
1 = 0$ (on the left) and $\ZZ_2[\pi]$ where $\pi$ is a 
root of the Eisenstein polynomial $X^2 + 2X - 2$ (on the right).
In our pictures, each bullet corresponds to a class modulo $2^5$.
The darkness of the bullet encodes the number of roots we got in the 
corresponding class: the bullet is black if we observed more than 
$400$ roots (in total, on our sample of $500,000$ polynomials), it
is left white if we have not got any root and grayscale colors are
used for intermediate number of hits. 
Moreover, on each $\ZZ_2$-line, the classes are ordered as follows:
$$\begin{array}{c}
  0, 16, 8, 24, 4, 20, 12, 28,
  2, 18, 10, 26, 6, 22, 14, 30, \\
  1, 17, 9, 25, 5, 21, 13, 29,
  3, 19, 11, 27, 7, 23, 15, 31,
\end{array}$$
\emph{i.e.} the even classes come first, followed by the odd
ones and inside each subgroup, the classes are gathered according 
to their congruence modulo $4$, then modulo $8$, \emph{etc.} This 
organisation is appropriate for our purpose but it somehow reflects
the $2$-adic distance.

On each picture, it is striking that the bullets are becoming much
brighter when we are approaching the $\ZZ_2$-line on the bottom. 
This empirical observation is in perfect compliance with the explicit 
formulas of Propositions~\ref{prop:rhounram} and~\ref{prop:rhoram}
that clearly indicate that the size of $\rho_{K,n}(x)$ is governed
by the distance of $x$ to $\QQ_2$.
Another unmissable point is that the bullets on the picture on the
right are much brighter than the bullets on the left. Again, this 
is due to the fact that the discriminant of $\QQ_2[\pi]$ has a 
smaller norm than the discriminant of $\QQ_4$; we then expect
more roots, and so darker bullets, in the case of $\QQ_4$.

\section{Some orders of magnitude}
\label{sec:magnitude}

In the previous section, we have tried to write down \emph{exact}
formulas for the density function $\rho_{K,n}$'s. However, in many
cases, it is sufficient---and sometimes even more useful---to know
their orders of magnitude.
In this section, we focus on the quantities:
$$\rho_n(K) = \int_K \rho_{K,n}(x) \: dx$$
which counts the average number of new roots in~$K$, \emph{i.e.}
$\rho_n(K) = \EE[Z^\new_{K,n}]$ (see Definition~\ref{def:newroot}
and Theorem~\ref{theo:meanZnew}).
We obtain sharp asymptotic estimations of them when $q$ and $r$
grows, proving in particular Theorems~\ref{theo:estimation} 
and~\ref{theo:sum}.

\subsection{Counting generators over finite fields}
\label{ssec:genFqf}

This subsection gathers several preliminary lemmas concerning
the number of generators of an extension of finite fields.
Write $\Fq$ (resp. $\Fqbar$) for the residue field of $F$ (resp.
$\bar F$). It is well-known that $\Fqbar$ is an algebraic closure
of $\Fq$. For any given a positive integer $f$, let $\FF_{q^f}$ denote
the unique extension of $\Fq$ of degree~$f$ sitting inside $\Fqbar$.
Let $G_f$ be the number of elements $x \in \FF_{q^f}$ such that
$\Fq[x] = \FF_{q^f}$. From the obvious fact that each 
$x \in \FF_{q^f}$ generates some $\FF_{q^m}$ for a divisor $m$
of $f$, we deduce the relation:
\begin{equation}
\label{eq:sumGm}
\sum_{m|f} G_m = q^f.
\end{equation}
This relation can be ``inverted'' using Moebius inversion formula.
We recall that the Moebius function $\mu : \ZZ_{> 0} \to \{0, 
1\}$ is defined by $\mu(p_1 \ldots p_s) = (-1)^s$ if $p_1, \ldots,
p_s$ are distinct prime numbers and $\mu(n) = 0$ as soon as $n$ is
divisible by a square. Eq.~\eqref{eq:sumGm} then becomes:
\begin{equation}
\label{eq:Gf}
G_f = \sum_{m|f} \mu\left(\frac f m \right) q^m.
\end{equation}
From the latter formula, one easily derive 
that $G_f = q^f + O\big(q^{f/2}\big)$.
Here is an other estimation of the same type we shall need in the
sequel.

\begin{lem}
\label{lem:sumGm}
For all positive integer $f$, we have:
$$\sum_{\substack{m|f\\m < f}} G_m q^m 
  \, \leq \, 2 q^f.$$
\end{lem}

\begin{proof}
A strict divisor of $f$ cannot certainly exceed $f/2$. On the other
hand, it is obvious from the definition that $G_m \leq q^m$ for
all~$m$. Hence, the sum of the lemma is bounded from above by:
$$\sum_{m < f/2} q^{2m} 
\leq \frac{q^{f+2} - 1}{q^2 - 1}
\leq \frac{q^f}{1 - q^{-2}} \leq 2 q^f$$
given that $q \geq 2$.
\end{proof}

Let $\varphi$ be the Euler's totient function. We recall that
$\varphi(n)$ is by definition the number of integers which are
less than $n$ and coprime with $n$. If the decomposition of $n$
in prime factors reads $n = p_1^{\alpha_1} \cdots p_s^{\alpha_s}$, one 
has the formula:
$$\varphi(n) = n \cdot \left( 1 - \frac 1{p_1}\right) \cdots
\left( 1 - \frac 1{p_s}\right).$$
The beautiful next lemma, which relates harmoniously the Moebius
and the Euler functions, is somehow classical; we nevertheless
include a proof for completeness.

\begin{lem}
\label{lem:moebiuseuler}
For any positive integer $n$, one has the relation:
$$\sum_{m|n} \frac{\mu(m)} m \,=\, \frac{\varphi(n)} n.$$
\end{lem}

\begin{proof}
Write $n = p_1^{\alpha_1} \cdots p_s^{\alpha_s}$ where the $p_i$'s 
are pairwise distinct prime numbers.
The divisors $m$ of $n$ are exactly the integers of the form
$m = p_1^{\beta_1} \cdots p_s^{\beta_s}$ with $\beta_i \leq \alpha_i$
for all~$i$. We then deduce from the definition of the Moebius
function that:
$$\sum_{m|n} \frac{\mu(m)} m 
 = \sum_{\beta_1 = 0}^1 \cdots \sum_{\beta_s = 0}^1
   \frac{(-1)^{\beta_1 + \cdots + \beta_s}}{p_1^{\beta_1}\ldots p_s^{\beta_s}}
 = \prod_{i=1}^s \left( 1 - \frac 1 {p_i}\right) = \frac{\varphi(n)}n$$
which proves the lemma.
\end{proof}

\begin{lem}
\label{lem:sumGfoverf}
For any positive integer $r$, one has the relation:
$$r \cdot \sum_{f|r} \frac{G_f}f 
 \,=\, \sum_{m|r} \varphi\left(\frac r m \right) q^m.$$
\end{lem}

\begin{proof}
The lemma follows from the following sequence of equalities:

\bigskip

\hfill%
\raisebox{3em}{%
$\begin{array}{@{}r@{\hspace{0.5ex}}l@{\qquad}l@{}}
\displaystyle
r \cdot \sum_{f|r} \frac{G_f}f 
 & = \displaystyle
   r \cdot \sum_{m|f|r} \mu\left(\frac f m\right)\frac{q^m}f.
 & \text{by Eq.~\eqref{eq:Gf}} \smallskip \\
 & = \displaystyle
   r \cdot \sum_{m|r} \, \sum_{n| \frac r m} \frac{\mu(n)} n \, \cdot \frac{q^m} m.$$
 & \text{(change of variables $n = \frac f m$)} \medskip \\
 & = \displaystyle
   \sum_{m|r} \, \varphi\left(\frac r m\right) q^m$$
 & \text{by Lemma~\ref{lem:moebiuseuler}}.
\end{array}$}
\hfill\null
\end{proof}

\subsection{Isolating the main contribution to $\rho_n(K)$}

We come back to the $p$-adic situation and to our problem of finding a 
sharp asymptotic of $\rho_n(K)$.
In what follows, we fix a finite extension $K$ of $F$. We denote
its degree by~$r$ and its residual degree by~$f$. The residue field of $K$ is
then $\FF_{q^f}$. Let $\calG_K$ be the set of generators of $\OK$,
that is:
$$\calG_K = \big\{\, x \in \OK \quad \text{such that} \quad
\OF[x] = \OK \, \big\}.$$
Similarly, for a positive integer~$f$, we define $\calG_f$ as
the set of generators of $\FF_{q^f}$ over $\FF_q$. By definition,
the number $G_f$ we have introduced in \S \ref{ssec:genFqf} is
the cardinality of $\calG_f$.
For each $\alpha \in \FF_{q^f}$, we let $U_\alpha$ denote the
open subset of $\OK$ consisting of elements whose image in the
residue field is~$\alpha$. The $U_\alpha$'s are then pairwise
disjoint and $\lambda_K(U_\alpha) = q^{-f}$ for all~$\alpha$.

\begin{lem}
\label{lem:mesGK}
For $\alpha \in \FF_{q^f}$, we have:
$$\begin{array}{r@{\hspace{0.5ex}}l@{\qquad}l}
q^{-f} \cdot \big(1 - q^{-f}\big)
   \,\leq\, \lambda_K(\calG_K \cap U_\alpha) 
 & \,\leq\, q^{-f}
 & \text{if } \alpha \in \calG_f,
\medskip \\
\calG_K \cap U_\alpha
 & \,=\, \emptyset
 & \text{if } \alpha \not\in \calG_f. 
\end{array}$$
\end{lem}

\begin{proof}
If $x \in \OK$ generates $\OK$ over $\OF$, its image in the
residue field must generate $\FF_{q^f}$ over $\FF_q$. This proves
the lemma when $\alpha \not\in \calG_f$.

We now assume that $\alpha \in \calG_f$.
First of all, it is clear that $\lambda_K(\calG_K \cap U_\alpha)
\leq \lambda_K(U_\alpha) = q^{-f}$.
Let $\bar Z$ be the minimal polynomial of $\alpha$ over $\FF_q$. Let 
also $a \in \OK$ be a lifting of~$\alpha$ and $Z \in \Omega_f$ be a 
polynomial lifting $\bar Z$. We fix in addition a uniformizer $\pi$ of 
$K$. Then $Z(a)$ is a multiple of $\pi$ and we write $Z(a) = \pi b$. 
Besides $\bar Z'(a)$ does not vanish given that $\FF_{q^f}/ \FF_q$ is a 
separable extension. For $x \in \OK$, the congruence:
$$Z(a + \pi x) 
\,\equiv\, Z(a) + Z'(a) \pi x 
\,\equiv\, \pi \cdot \big(b + Z'(a)\:x\big) \pmod{\pi^2}$$
shows that $Z(a + \pi x)$ is a uniformizer of $K$ as soon as the
image of $x$ in the residue field is different from
$\frac{-b}{Z'(a)}$. When this occurs, $\OF[x]$ then contains a
uniformizer of $F$ together with a generator of the residue field,
implying that $\OF[x] = \OK$. The lemma follows.
\end{proof}

\begin{prop}
\label{prop:boundrhon}
For all $n \geq r$, we have:
$$\int_{K\backslash\calG_K} \rho_{K,n}(x) \: dx 
 \, \leq \, 4 \cdot \Vert D_K\Vert \cdot q^{-f}$$
\end{prop}

\begin{proof}
Set $V = K \backslash \calG_K$ and $V_\alpha = V \cap U_\alpha$ for
$\alpha \in \FF_{q^f}$. Set also $V_\infty = K\backslash \OK$.
The set $V$ then appears as the disjoint union of the $V_\alpha$'s
for $\alpha$ varying in $\PP^1(\FF_{q^f})$. We are going to bound
the integral of $\rho_{K,n}$ an each $V_\alpha$ separately.
When $\alpha \in \calG_K$, it follows from Lemma~\ref{lem:mesGK}
that $\lambda_K(V_\alpha) \leq q^{-2f}$. Besides, coming back to
Definition~\ref{def:rho}, we remark that $\rho_{K,n}$ is upper 
bounded by~$\Vert D_K \Vert$ on $\OK$ and so, on $V_\alpha$. We then deduce,
in this case, that
$$\frac 1 {\Vert D_K \Vert} \cdot \int_{V_\alpha} \rho_{K,n}(x) \: dx 
 \, \leq \, q^{-2f}.$$
We now consider the case where $\alpha \in \FF_{q^f}$, $\alpha
\not\in \calG_K$. Write $\ell = \Fq[\alpha]$ and let $m(\alpha)$
be the degree of the extension $\ell/\Fq$.
By assumption $m(\alpha) < f$. On the other hand, the reduction
morphism $\OK \to \FF_{q^f}$ takes $\OF[x]$ to $\ell$. We deduce
that the cardinality of the quotient $\OK/\OF[x]$ is bounded from
below by the cardinaly of $\FF_{q^f}/\ell$, which is 
$q^{f-m(\alpha)}$. Injecting this bound in the defition of
$\rho_{n,K}$, we find:
$$\frac 1 {\Vert D_K \Vert} \cdot \int_{V_\alpha} \rho_{K,n}(x) \: dx 
 \, \leq \, \lambda_K(V_\alpha) \cdot q^{m(\alpha) - f}
 \, = \, q^{m(\alpha) - 2f}.$$
Finally, when $\alpha = \infty$, we perform the change of variables 
$x \mapsto x^{-1}$, which lefts us with:
$$\int_{V_\infty} \rho_{K,n}(x) \: dx
 \,=\, \int_{V_0} \rho_{K,n}(x) \: dx
 \,\leq\, \Vert D_K \Vert \cdot q^{1 - 2f}.$$
Summing up all upper bounds, we find:
$$\frac 1 {\Vert D_K \Vert} \cdot
  \int_{K\backslash\calG_K} \rho_{K,n}(x) \: dx 
 \,\leq\, q^{1-2f} \,+\, G_f q^{-2f} 
      + \sum_{\substack{m|f\\m<f}} G_m q^{m - 2f}$$
given that the number of $\alpha \in \FF_{q^f} \backslash \calG_K$ 
for which $m(\alpha) = m$ is equal to $G_m$ by definition.
Remembering that $G_f \leq q^f$ and using Lemma~\ref{lem:sumGm},
we end up with:
$$\frac 1 {\Vert D_K \Vert} \cdot
  \int_{K\backslash\calG_K} \rho_{K,n}(x) \: dx 
  \,\leq\, q^{1-2f} + q^{-f} + 2 q^{-f}
  \,\leq\, 4 q^{-f}.$$
The proposition is proved.
\end{proof}

The next corollary can be seen as an effective version of 
Theorem~\ref{theo:estimation} (use Eq.~\eqref{eq:Gf}).

\begin{cor}
\label{cor:estimation}
We have the estimations:
$$\begin{array}{rr@{\hspace{0.5ex}}l}
& \displaystyle
  - q^{-f} \, \leq\,
  \frac{\rho_r(K)}{\Vert D_K \Vert} \,-\,
  \frac{q^{r+1} - q^r}{q^{r+1} - 1} \cdot \frac{G_f}{q^f}
  & \, \leq \, 4 q^{-f}
\medskip \\
\text{for } n \geq 2r-1,
& \displaystyle
  - q^{-f} \, \leq\,
  \frac{\rho_n(K)}{\Vert D_K \Vert} \,-\,
  \frac{q^r}{q^r + 1} \cdot \frac{G_f}{q^f}
  & \, \leq \, 4 q^{-f}
\end{array}$$
\end{cor}

\begin{proof}
Let $n = r$ or $n \geq 2r - 1$.
Recall that by definition:
$$\rho_n(K) = \int_K \rho_{K,n}(x) \: dx
  \,=\, \int_{\calG_K} \rho_{K,n}(x) \: dx
  \,+\, \int_{K\backslash\calG_K} \rho_{K,n}(x) \:dx.$$
By Proposition~\ref{prop:boundrhon}, we know that the integral
over $K \backslash \calG_K$ is upper bounded by $4 q^{-f}$. It is 
also obviously nonnegative since the integrand is nonnegative. 
Therefore, it is enough to prove that:
$$0 \, \leq \, 
  c_n \: \frac{G_f}{q^f} \,-\,
  \int_{\calG_K} \frac{\rho_{K,n}(x)}{\Vert D_K \Vert} \: dx
  \, \leq\, q^{-f}$$
with $c_r = \frac{q^{r+1} - q^r}{q^{r+1} - 1}$ and $c_n =
\frac{q^r}{q^r + 1}$ for $n \geq 2r -1$. On the other hand, a 
computation similar to the one we carried out in \S 
\ref{sssec:formulas} gives:
$$\int_{\OK} \Vert t \Vert^r \: dt = \frac{q^r}{q^r + 1}.$$
Hence, it follows from the explicit formulas of 
Theorem~\ref{theo:density} that $\rho_{K,n}$ is constant equal to 
$c_n {\cdot} \Vert D_K \Vert$ on~$\calG_K$. We are then reduced to
check that $0 \leq q^{-f} G_f - \lambda_K(\calG_K) \leq q^{-f} c_n^{-1}$.
This follows from Lemma~\ref{lem:mesGK} after remarking that 
$0 \leq c_n \leq 1$.
\end{proof}

By the monotony result of Theorem~\ref{theo:density}, the previous 
corollary also provides estimations of $\rho_n(K)$ when $n$ varies 
between $r$ and $2r{-}1$. They are however less sharp since the error 
term is \emph{a priori} only in $O\big(q^{-1}\big)$ instead of 
$O\big(q^{-f}\big)$. We can nevertheless recover more accuracy when 
the extension $K/F$ is unramified.

\begin{theo}
If $K/F$ is unramified, we have the estimation:
$$- q^{-r} \, \leq\,
  \frac{\rho_n(K)}{\Vert D_K \Vert} \,-\,
  \frac{q^{n+1} - q^r}{q^{n+1} - 1} \cdot \frac{G_r}{q^r}
  \, \leq \, 4 q^{-r}$$
for all $n$ between $r$ and $2r-1$.
\end{theo}

\begin{proof}
It is exactly the same than the proof of
Corollary~\ref{cor:estimation}, except that the value of $\rho_{K,n}$ 
on $\calG_K$ is now given by Proposition~\ref{prop:unram}.
(Note also that $f = r$ in the unramified case.)
\end{proof}

Replacing $G_r$ by its expression given by Eq.~\eqref{eq:Gf}, we 
end up with the following asymptotic development in the spirit of 
Theorem~\ref{theo:estimation}:
\begin{equation}
\label{eq:rhonunram}
\rho_n(K) = 
  \left( 1 - \frac 1 {q^{n-r+1}} \right) \cdot
  \sum_{m|r}\: \mu\!\left(\frac r m\right) q^{m-r} \,+\, 
  O\left(\frac 1{q^r}\right).
\end{equation}
This estimation holds true for any \emph{unramified} 
extension $K/F$ and an integer $n$ in the range $[r,\: 2r{-}1]$;
moreover, the constant hidden in the $O(-)$ is absolute.

\subsection{Summing up over extensions of fixed degree}

The aim of this subsection is to prove Theorem~\ref{theo:sum}. The 
strategy we will follow is quite simple: we sum up the surroundings of 
Corollary~\ref{cor:estimation} over all extensions~$K$ of a fixed 
degree~$r$.
For this, the main new ingredient we shall need is Serre's mass formula
that we recall below. If $K$ is a finite extension of $F$, we write
$\Aut_{\!\Falg}(K)$ for the group of automorphisms of $F$-algebras of~$K$.

\begin{theo}[Serre's mass formula]
\label{theo:serre}
For any positive integer~$r$, we have:
$$q^{r-1} \sum_K \frac{\Vert D_K \Vert}{\card\:\Aut_{\!\Falg}(K)} = 1$$
where the sum runs over all isomorphism classes of \emph{totally
ramified} extensions $K$ of $F$ of degree~$r$.
\end{theo}

\begin{proof}
See~\cite{serre}.
\end{proof}

We need to be careful that, in the formulation of
Theorem~\ref{theo:serre}, the sum runs over \emph{isomorphism
classes} of extensions and not extensions sitting inside $\bar F$
as we work with usually in this article.
To switch between those two viewpoints, we observe that an abstract
extension $K$ of $F$ of degree~$r$ admits $r$ embeddings into $\bar
F$. However, two such embeddings have the same image (and so define
the same subfield of $\bar F$) when they differ by an automorphism
of~$K$. The number of subfields of $\bar F$ which are isomorphic to
$K$ is then exactly equal to $\frac r {\card\:\Aut_{\!\Falg}(K)}$.
Therefore, Serre's mass formula can be rewritten as follows:
\begin{equation}
\label{eq:serre}
\sum_{K \in \Ex_{r,1}} \Vert D_K \Vert = \frac r{q^{r-1}}
\end{equation}
where the indexation set $\Ex_{r,1}$ consists of all subfields $K$ 
of $\bar F$ which are totally ramified extensions of $F$ of degree~$r$.
More generally, given an auxiliary positive integer~$f$ dividing~$r$, 
we define $\Ex_{r,f}$ as the set of embedded extensions of $F$ of
degree~$r$ and residual degree~$f$.
Serre's mass formula extends without difficulty to extensions in
$\Ex_{r,f}$.

\begin{prop}
\label{prop:serre}
For all positive integers $r$ and all divisors $f$ or $r$, we
have:
$$\sum_{K \in \Ex_{r,f}} \Vert D_K \Vert = 
\frac r{q^r} \cdot \frac{q^f}f.$$
\end{prop}

\begin{proof}
Let $F_f$ be the unique unramified extension of $F$ sitting
inside $\bar F$. Its residue field is~$\FF_{q^f}$ and the
normalized norm on $F_f$ is the $f$-th power of $\Vert \cdot
\Vert$.
Moreover, any extension~$K$ in $\Ex_{r,f}$ canonically contains $F_f$ 
and then appears uniquely as a totally ramified extension of $F_f$. In
addition, the discriminant of $K/F$, still denoted by $D_K$, is the
$f$-th power of the discriminant $D_{K/F_f}$ of $K/F_f$. Hence
$\Vert D_{K/F_f} \Vert^f = \Vert D_K \Vert$.
The formula~\eqref{eq:serre} applied with the base field $F_f$ then
gives:
$$\sum_{K \in \Ex_{r,f}} \Vert D_K \Vert 
 = \sum_{K \in \Ex_{r,f}} \Vert D_{K/F_f} \Vert^f
 = \frac r f \cdot \frac 1 { (q^f)^{r\!/\!f\, -\,1}}
 = \frac r {q^r} \cdot \frac {q^f} f$$
which establishes the proposition.
\end{proof}

We now fix two positive integers $r$ and $n$ with $n \geq 2r-1$.
For a given divisor $f$ or $r$, observing that:
$$ 1 - \frac 1{q^f} 
 \leq 1 - \frac 1{q^r} 
 \leq \frac{q^r}{q^r + 1} \leq 1$$
we derive from Corollary~\ref{cor:estimation} that:
$$\left| \frac{\rho_n(K)}{\Vert D_K \Vert} - \frac{G_f}{q^f} \right|
\leq \frac 5{q^f}$$
for any extension $K \in \Ex_{r,f}$. Summing up over all such 
extensions, we find:
$$\left| \sum_{K \in \Ex_{r,f}} \hspace{-0.5em} \rho_n(K) \,-\,
\frac{G_f}{q^f} \sum_{K \in \Ex_{r,f}} \hspace{-0.2em} \Vert D_K \Vert \, \right|
\, \leq \,
\frac 5{q^f} \sum_{K \in \Ex_{r,f}} \Vert D_K \Vert$$
which gives after Proposition~\ref{prop:serre}:
$$\left| \sum_{K \in \Ex_{r,f}} \hspace{-0.5em} \rho_n(K) \,-\,
\frac r{q^r} {\cdot} \frac{G_f}f \, \right|
\, \leq \,
\frac {5}{q^r}\cdot \frac r f.$$
Writing that $\Ex_r$ is the disjoint union
of the $\Ex_{r,f}$'s for $f$ varying in the set of divisors of~$r$
and summing up all the above estimations, we end up with:
\begin{equation}
\label{eq:estimation:r}
\left| \sum_{K \in \Ex_r} \rho_n(K) \,-\,
\frac r{q^r} \sum_{f|r} \frac{G_f}f \, \right|
\, \leq \, \frac {5}{q^r} \sum_{f|r} \frac r f
\, = \, \frac 5 {q^r} \sum_{f|r} f.
\end{equation}
Finally, the error term is controlled thanks to the following classical
result:
$$\displaystyle \sum_{f|r} f \,=\, O\big(r \cdot \log\log r\big)$$
(see for instance \cite{gronwall}).
Injecting it in Eq.~\eqref{eq:estimation:r} and using the summation
formula of Lemma~\ref{lem:sumGfoverf}, we finally get 
Theorem~\ref{theo:sum}.

\begin{rem}
It is amusing to observe that
Theorems~\ref{theo:estimation} and~\ref{theo:sum} have a very
similar shape, except that former involves the Moebius function
whereas the latter implicates the Euler function.
Comparing both results and writing $\delta = \varphi - \mu$, we
obtain:
$$\sum_{K \in \Ex^{\text{ram}}_r} \rho_n(K) = 
  \sum_{m|r}\: \delta\!\left(\frac r m\right) q^{m-r} \,+\, 
  O\left(\frac{r \cdot \log \log r}{q^r}\right)$$
where the indexation set $\Ex^{\text{ram}}_r$ consists of all
\emph{ramified} extensions of $F$ of degree~$r$ sitting inside
$\bar F$. Since $\delta(1) = 0$,
the dominant term of the sum in the right hand side is obtained
for $m = r/\ell$ where $\ell$ is the smallest divisor 
of~$m$. Its value is $\delta(\ell)\:q^{-r(1-\frac 1 \ell)}$. Since 
$\ell$ is necessarily prime, we have moreover $\mu(\ell) = -1$ and 
$\varphi(\ell) = \ell - 1$, giving $\delta(\ell) = \ell$. As a
consequence, we obtain the approximation:
$$\sum_{K \in \Ex^{\text{ram}}_r} \rho_n(K) \,\approx\,
  \ell \cdot q^{-r(1-\frac 1 \ell)}.$$
For example, when $r = 2$, we find that the order of magnitude of
$\sum_{K \in \Ex^{\text{ram}}_2} \rho_n(K)$ is $2/q$. If $q$ is
odd, the set $\Ex^{\text{ram}}_2$ consists exactly of two 
elements (namely the extensions $F\big[\sqrt \pi\big]$ and 
$F\big[\sqrt {a\pi}\big]$ where $\pi$ is a uniformizer of $F$ 
and $a$ is an element of $\OF$ which is not a square in the residue
field) and each corresponding summand contributes for~$1/q$.
\end{rem}

\section{The setup of étale algebras}
\label{sec:etale}

In the previous section, we have exclusively focused our interest on the 
\emph{mean} of the $Z_{U,n}$'s. However, it is evident that the mean 
captures only a small part of the complexity of the phenomena and, 
beyond it, we would like to study higher moments or correlations between 
the $Z_{U,n}$'s to get a more precise picture of the situation.

In this section, we propose to attack these questions by applying the 
same methods as before in some suitable enlarged framework, which is 
that of finite étale algebras (that are finite products of finite 
extensions of $F$).
On the one hand, allowing this flexibility is somehow harmless
because all the techniques we have developed in the previous 
sections extend without difficulty. However, on the other hand,
it is also quite interesting because it sheds new lights on many
natural questions. For instance, it turns out that the average number 
of roots in power algebras of the form $K^m$ is closely related to 
higher moments of the random variables $Z_{U,n}$ for $U \subset K$.
Similarly, studying the average number of roots in products of type 
$K_1 \times K_2$ provides valuable information on the correlations
between $Z_{K_1,n}$ and $Z_{K_2,n}$.

This section is organized as follows. In \S \ref{ssec:etaleroots},
we introduce the random variables $Z_{E,n}$ and $Z^\new_{E,n}$ for
a finite étale algebra $E$ and relate them to products of the $Z_{K,n}$'s.
In \S \ref{ssec:etaledensity}, we extend Theorems~\ref{theo:mean}
and~\ref{theo:density} to our new setup. In \S \ref{ssec:F2}, we
spend some time on the special case of the étale algebra $F^2$,
making everything explicit in this example. We then derive from 
our calculations precise information about the variance and the
correlations between the $Z_{U,n}$'s where $U$ is an open subset
of~$F$, proving in particular Theorem~\ref{theo:cov}.
Finally, we prove Theorem~\ref{theo:mass} in \S \ref{ssec:mass}.

\subsection{Roots and new roots}
\label{ssec:etaleroots}

Before getting to the heart of the matter, we gather basic standard
facts about finite étale $F$-algebras. To begin with, let us recall
that a finite étale algebra over~$F$ is defined as an finite algebra
over~$F$ without nilpotent elements \cite[Definition~2.1.2]{cohen}.
It is well known (see \cite[Corollary~2.1.6]{cohen}) that any
finite étale algebra over $F$ can be decomposed
as a product $K_1 \times K_2 \times \cdots \times K_m$ where
each $K_i$ is a finite extension of $E$.
In what follows, we will often pick such a decomposition and work
with it. The two next (classical) results will be quite useful to 
check that all our forthcoming constructions are intristic, \emph{i.e.} 
do not depend on the choice of a decomposition as above.

\begin{lem}
\label{lem:etale:surj}
Let $K, K_1, \ldots, K_m$ be finite extensions of $F$ and
set $E = K_1 \times \cdots \times K_m$.
Let $f : E \to K$ be a surjective morphism of $F$-algebras.
Then there exists an index $i$ such that $f = \varphi \circ
\text{pr}_i$ where $\varphi : K_i \to K$ is an isomorphism
and $\text{pr}_i$ denotes the projection on the $i$-th factor.
\end{lem}

\begin{proof}
Let $e_i = (0, \ldots, 0, 1, 0, \ldots, 0) \in E$ with the~$1$ 
in $i$-th position.
If $x$ lies in the kernel of $f$, so does $e_i x$ for all~$i$.
This proves that $\ker f$ decomposed as $I_1 \times \cdots \times 
I_m$ where $I_i$ is some ideal of $K_i$. Since $K_i$ is a field, one
must have $I_i = 0$ or $I_i = K_i$.
Moreover, since $f$ is surjective, $\ker f$ is a maximal ideal.
This shows that there exists a special index $i$ such that $I_i 
= 0$ and $I_j = K_j$ for all $j \neq i$. Hence $f$ factors 
through $\text{pr}_i$ and the lemma follows.
\end{proof}

\begin{cor}
\label{cor:etale:surj}
Let $K_1, \ldots, K_m$ be pairwise nonisomorphic finite extensions 
of $F$. Let $(a_1, \ldots, a_m)$ and $(b_1, \ldots, b_m)$ be two
tuple of nonnegative integers. Let:
$$f :\,
K_1^{a_1} \times \cdots \times K_m^{a_m} \longrightarrow
K_1^{b_1} \times \cdots \times K_m^{b_m}$$
be a surjective morphism of $F$-algebras. Then, for each $i \in
\{1, \ldots, m\}$, there exists an injection $\sigma_i :
\{1, \ldots, b_i\} \to \{1, \ldots, a_i\}$ and a tuple
$(\varphi_{i,1}, \ldots, \varphi_{i,b_i})$ of automorphisms
of $K_i$ such that:
$$f\big((x_{i,j})_{1 \leq i \leq m, 1 \leq j \leq a_i}\big) 
 = \big(\varphi_{i,j}(x_{i,\sigma_i(j)})\big)_{1 \leq i \leq m, 1 \leq j \leq b_i}.$$
\end{cor}

\begin{proof}
Write $x = (x_{i,j})_{1 \leq i \leq m, 1 \leq j \leq a_i}$.
Given $i \in \{1, \ldots, m\}$ and $j \in \{1, \ldots, b_i\}$,
Lemma~\ref{lem:etale:surj} tells us that the $(i,j)$ coordinate
of $f(x)$ must be of the form $\varphi_{i,j}(x_{i,\sigma_i(j)})$
for some automorphism $\varphi_{i,j}$ of $K_i$ and some 
$\sigma_i(j) \in \{1, \ldots, b_i\}$. This establishes the shape
we have given for~$f$. Finally, the fact that the $\sigma_j$'s
are injective follows from the surjectivity of $f$.
\end{proof}

When $a_i = b_i$ for all~$i$, Corollary~\ref{cor:etale:surj}
indicates that the group of automorphisms of $F$-algebras of
$E = K_1^{a_1} \times \cdots \times K_m^{a_m}$, denoted by
$\Aut_{\!\Falg}(E)$, is canonically isomorphic to
$$\prod_{i=1}^m \big(\mathfrak S_{a_i}
\rtimes \Aut_{\!\Falg}(K_i)^{a_i}\big)$$
where $\mathfrak S_{a_i}$ is the symmetric group on $a_i$ letters
and it acts on $\Aut_{\Falg}(K_i)^{a_i}$ by permuting the
automorphisms. In particular we deduce that:
$$\card \Aut_{\!\Falg}(E) = \prod_{i=1}^m \,\, a_i! \:
 \big(\card\Aut_{\!\Falg}(K_i)\big)^{a_i}.$$
Another property of étale algebras it will be important to keep in
mind in the sequel is recorded in the next proposition.

\begin{prop}
\label{prop:etale:subalg}
Any $F$-subalgebra of a finite étale $F$-algebra is finite étale.
\end{prop}

\begin{proof}
It is obvious from the definition.
\end{proof}

We now go back to our topic.
Given a polynomial $P \in F[X]$, a root of $P$ in a finite étale $F$-algebra 
$E$ is, by definition, an element $x \in E$ such that $P(x) = 0$. It is 
a standard fact that the datum of a root of $P$ in $E$ is equivalent to 
the datum of a morphism of $F$-algebras $F[X]/P \to E$: to a root $x$,
we associate the morphism taking $X$ to $x$ and, conversely, to a
morphism $\varphi : F[X]/P \to E$, we associate the root $\varphi(X)$.

The notion of new elements we have introduced in the case of extensions 
in Definition~\ref{def:newroot} extends immediately to finite
étale algebras.

\begin{deftn}
\label{def:etale:newroot}
Let $E$ be a finite étale algebra over $F$.
An element $x \in E$ is \emph{new} in $E$ if it does not belong to
any strict sub-$F$-algebra of $E$, \emph{i.e.} if $F[x] = E$.
\end{deftn}

\begin{lem}
\label{lem:newsurj}
Let $P$ is a polynomial over $F$. Let $x \in E$ is a root of $P$
and $\varphi : F[X]/P \to E$ be its associated morphism.
Then $x$ is new in $E$ if and only if $\varphi$ is surjective.
\end{lem}

\begin{proof}
It follows from the fact that the image of $\varphi$
is the $F$-algebra generated by~$x$.
\end{proof}

Given a positive integers $n$, a finite étale $F$-algebra $E$ and an 
open subset $U \subset E$, we define the random variables
$Z_{U,n} : \Omega_n \to \ZZ$ and $Z^\new_{U,n} : \Omega_n \to
\ZZ$ by:
\begin{align*}
Z_{U,n}(P) & 
  = \text{number of roots of $P$ in $U$} \\
Z^\new_{U,n}(P) & 
  = \text{number of roots of $P$ in $U$, which are new in $E$}.
\end{align*}
It follows from Lemma~\ref{lem:newsurj} that $Z_{U,n}(P)$
(resp. $Z^\new_{U,n}(P)$) is also the number of morphisms (resp.
surjective morphisms) of $F$-algebras $\varphi : F[X]/P \to E$ such
that $\varphi(X) \in U$. In particular 
$Z_{E,n}(P) = \card\: \Hom_{\Falg}\big(F[X]/P, E\big)$ and
$Z^\new_{E,n}(P) = \card\: \Hom^\surj_{\Falg}\big(F[X]/P, E\big)$
(where the notations are transparent).
This reformulation shows directly that $Z^\new_{E,n}$ identically
vanishes when $n < [E:F]$. Besides, the random variables $Z_{U,n}$ 
and $Z^\new_{U,n}$ are related by the formula:
\begin{equation}
\label{eq:ZZnew}
Z_{U,n} 
= \sum_{E' \subset E} Z^\new_{E'\cap U, n}
\end{equation}
which simply comes from the observation that an element $x \in E$
is new in a unique subalgebra $E'$ of $E$, namely $E' = F[x]$. This
algebra is moreover necessarily étale over~$F$ by
Proposition~\ref{prop:etale:subalg}.
We note furthermore that the construction $Z_{U,n}$ is 
multiplicative with respect to the parameter~$U$, \emph{i.e.} that:
$$Z_{U_1 \times U_2, n} = Z_{U_1, n} \cdot Z_{U_2, n}$$
for any positive integer $n$, any finite étale $F$-algebras $E_1$ and $E_2$ 
and any open subsets $U_1$ and $U_2$ of $E_1$ and $E_2$ respectively.
This property is interesting for us because it implies the formula:
\begin{equation}
\label{eq:cov}
\Cov\big(Z_{U_1,n},\, Z_{U_2,n}\big) =
\EE\big[Z_{U_1 \times U_2, n}\big] -
\EE\big[Z_{U_1, n}\big] {\cdot} \EE\big[Z_{U_2, n}\big]
\end{equation}
which shows that the covariance of $Z_{U_1,n}$ and $Z_{U_2,n}$ can 
be computed in terms of means of random variables of the form $Z_{U,n}$.
The next proposition highlights a similar property for the higher 
(factorial) moments of certain random variables $Z^\new_{U,n}$.

\begin{prop}
\label{prop:highermoments}
Let $K$ be a finite extension of $F$ and set 
$a = \card \: \Aut_{\!\Falg}(K)$.
Let $m, n$ be two positive integers and let $U$ be an open subset
of $K$ stable by all morphisms in $\Aut_{\!\Falg}(K)$. Then:
$$Z^\new_{U^m, n} = Z^\new_{U,n} \cdot \big(Z^\new_{U,n} - a\big) 
\cdot \big(Z^\new_{U,n} - 2a\big) \cdots \big(Z^\new_{U,n} - (m{-}1)a\big).$$
\end{prop}

\begin{proof}
Write $r = [K:F]$ and pick a polynomial $P \in \Omega_n$.
Let $x = (x_1, \ldots, x_m) \in U^m$. We claim that $x$ is a new
root of $P$ in $E$ if and only if $x_i$ is a new root of $P$ in $K$
for all~$i$ and the $x_i$'s are pairwise nonconjugate. To prove to
claim, we let $Z$ be the minimal polynomial of $x$ over $F$ and,
similarly, for all $i \in \{1, \ldots, m\}$, we denote by $Z_i$ the
minimal polynomial of $x_i$ over $F$. We then have $Z = \text{lcm}
(Z_1,\ldots, Z_m)$. Notice moreover that all the $Z_i$'s are
irreducible since we have assumed that $K$ is a field.
On the other hand, the fact that $x$ is new in $E$ (resp. $x_i$ is new 
in $K$) is equivalent to the equality $\deg Z = rm$ (resp. $\deg Z_i = 
r$). We deduce from this that $x$ is new in $E$ if and only if $x_i$ is 
new in $K$ for all $i$ and the $Z_i$'s are pairwise coprime. 
By irreducibility, the coprimality condition is equivalent to the fact 
that the $Z_i$'s are pairwise distinct, which is further equivalent to
the fact that the $x_i$'s are pairwise nonconjugate. This establishes
our claim.

We are now ready to count the number of new roots of $P$ in $U^m$. 
Indeed, by what precedes, it is equivalent to count the number of
tuples $(x_1, \ldots, x_m) \in U^m$ of new roots in $K$ which are
pairwise nonconjugate. By definition of $Z^\new_{U,n}$, we have $Z^\new_{U,n}(P)$ possibilities for
$x_1$. The fact that $x_2$ cannot be conjugate to $x_1$ eliminates 
exactly $a$~possibilities because we have assumed that $U$ is stable
under the action of $\Aut_{\!\Falg}(K)$. It then remains $Z^\new_{U,n}(P) - a$
possibilities for~$x_2$. Similarly, we have $Z^\new_{U,n}(P) - 2a$
possibilities for~$x_3$ because it has to be nonconjugate to both
$x_1$ and $x_2$. Repeating this argument $m$ times, we end up with
the formula of the proposition.
\end{proof}

\subsection{Density functions}
\label{ssec:etaledensity}

In this subsection, we aim at extending the definition of density 
functions to the setting of étale algebras and at proving variants of 
Theorems~\ref{theo:mean} and~\ref{theo:density} in this framework. 
For this, the first step is to find an adequate generalization of
the $p$-adic Kac-Rice.

\subsubsection*{Kac-Rice formula}

If $E$ is a finite étale algebra over $F$ presented as $E = K_1 \times K_2
\times \cdots \times K_m$ (where the $K_i$'s are finite extensions
of~$F$), we endow it with the norm $\Vert \cdot \Vert$ defined by:
$$\Vert (x_1, x_2, \ldots, x_m) \Vert 
 = \max\big( \Vert x_1 \Vert, \Vert x_2 \Vert, \ldots, \Vert x_m \Vert \big)
\qquad (x_i \in K_i).$$
We deduce from Corollary~\ref{cor:etale:surj} and the discussion
just after that the above definition is intrinsic in the
sense that it does not depend on the chosen identification $E \simeq
K_1 \times \cdots \times K_m$.
We let $\OE$ be the subring of $E$ consisting
of elements of norm at most~$1$. When $E$ is presented as $E = K_1
\times \cdots \times K_m$, we have 
$\OE = \O_{K_1} \times \cdots \times \O_{K_m}$.

Given $E$ as above, we also define the norm map $N_{E/F} : E \to F$
taking an element $x \in E$ to its so-called norm which is, by
definition, the determinant of the $F$-linear mapping $E \to E$, $y
\mapsto xy$. When $E = K_1 \times \cdots \times K_m$, we have:
$$N_{E/F}\big((x_1, \ldots, x_m)\big) = 
  N_{K_1/F}(x_1) \cdots N_{K_m/F}(x_m)$$
for what we derive:
\begin{align*}
\Vert N_{E/F}\big((x_1, \ldots, x_m)\big) \Vert 
 & = \Vert N_{K_1/F}(x_1) \Vert \cdots \Vert N_{K_m/F}(x_m) \Vert \\
 & = \Vert x_1 \Vert^{[K_1:F]} \cdots \Vert x_m \Vert^{[K_m:F]}.
\end{align*}
The latter formula shows in particular that $N_{E/F}$ maps $\OE$
to~$\OF$.

With the above preparation, the extension of the $p$-adic Kac-Rice
formula to our new setting can be formulated as follows.

\begin{theo}
\label{theo:etalepKac}
Let $E$ be a finite étale algebra over $F$ and set $r = [E:F]$.
Let $U$ be a compact open subset of $E$ and 
let $f : U \to E$ be a strictly differentiable function.
We assume that $f'(x)$ is invertible in $E$ for all $x \in E$ such 
that $f(x) = 0$. Then:
$$\card f^{-1}(0) = \lim_{s \to \infty} \,
q^{sr} \cdot \int_U \Vert N_{E/F}(f'(x)) \Vert \cdot 
\1_{\{\Vert f(x) \Vert \leq q^{-s}\}} \, dx.$$
\end{theo}

\begin{proof}
It is entirely similar to that of Theorem~\ref{theo:pKac}.
\end{proof}

\subsubsection*{Density functions}

We define the discriminant $D_E$ of a finite étale $F$-algebra $E = K_1 
\times \cdots \times K_m$ as the product $D_{K_1} \cdots D_{K_m}$. One 
checks that $D_E$ is also the discriminant of the bilinear form $E 
\times E \to F$, $(x,y) \mapsto \Tr_{E/F}(xy)$ where $\Tr_{E/F}$ is the 
trace map of $E$ over $F$. This alternative definition shows in particular
that $D_E$ 
does not depend (up to multiplication by an invertible element) on the 
choice of the presentation $E = K_1 \times \cdots \times K_m$.

In the case of étale algebras, the density functions cannot be defined
exactly the same way as for extensions (see Definition~\ref{def:rho})
because it may happen that neither $x$ nor $x^{-1}$ falls in the ring
of integers.
We will then proceed in a slightly different manner. We denote by
$\lambda_E$ the Haar measure on $E$ normalized by $\lambda_E(\OE) 
= 1$. If $E = K_1 \times \cdots \times K_m$, we simply have
$\lambda_E = \lambda_{K_1} \otimes \cdots \otimes \lambda_{K_m}$.
A second important ingredient we will need is a \emph{height}
function $H : E \to \RR$; writing again $E = K_1 \times \cdots 
\times K_m$, it is defined as follows:
$$H\big((x_1, \ldots, x_m)\big) =
  \prod_{i=1}^m\, \max\!\Big(1,\, \Vert x_i \Vert^{[K_i:F]}\Big)
\qquad (x_i \in K_i).$$
Using again Corollary~\ref{cor:etale:surj}, we conclude that this 
notion does not depend on the choice of the identification 
$E = K_1 \times \cdots \times K_m$.

Given a positive integer~$n$ and a finite étale $F$-algebra $E$
of degree~$r$, we set:
$$\rho_{E,n}(x) 
 = \frac{\Vert D_E \Vert \cdot \lambda_E\big(\OF[x]\big)}
   {H(x)^{n+1}} \cdot
   \int_{\Omega_{n-r}} \hspace{-1em} \Vert N_{E/F}(Q(x)) \Vert \: dQ$$
The main benefit of the above expression is its validity for any $x
\in E$. In particular, when $x$ is not new in $E$, we observe that
$\OF[x]$ is included in a strict $F$-linear subalgebra of $E$ and
thus has measure zero; $\rho_{E,n}(x)$ then vanishes as well in
this case. In a similar fashion, when $x$ is new in $\OE$, the height
of $x$ is~$1$ and the measure of $\OF[x]$ is the inverse of the 
cardinality of $\OE/\OF[x]$; we then get in this case:
$$\rho_{E,n}(x) 
 = \frac{\Vert D_E \Vert}{\card\big(\OE/\OF[x]\big)} \cdot
     \int_{\Omega_{n-r}} \hspace{-1em} \Vert N_{E/F}(Q(x)) \Vert \: dQ$$
which is exactly the formula of Definition~\ref{def:rho}.
Theorems~\ref{theo:mean} and~\ref{theo:density} now extend almost
\emph{verbatim} with, however, one notable exception: the monotony
property of Theorem~\ref{theo:density} no longer holds in the
framework of étale algebras; only remains the fact that the sequence 
$(\rho_{E,n})_{n \geq 1}$ is eventually constant.

\begin{theo}
\label{theo:etale:density}
For any positive integer $n$, any finite étale $F$-algebra $E$ and
any open subset $U$ of $E$, we have:
\begin{align*}
\EE\big[Z^\new_{U,n}\big] 
  & \,\, = \,\, 
  \int_U \rho_{E,n}(x)\:dx \\
\EE[Z_{U,n}] 
  & \,\, = \,\, 
  \sum_{E' \subset E} \,\,
  \int_{U\cap E'} \rho_{E',n}(x)\:dx
\end{align*}
where the latter sum runs over all $F$-subalgebras $E'$ of $E$.
Moreover, writing $r = [E:F]$, the function $\rho_{E,n}$ satisfies
the following list of properties, in which $x$ denotes an element
of $E$.
\begin{enumerate}
\item \emph{(Vanishing)}
If $F[x] \neq E$ or $n < r$, then $\rho_{E,n}(x) = 0$.
\item \emph{(Continuity)}
The function $\rho_{E,n}$ is continuous on $K$.
\item \emph{(Transformation under homography)}
For $\left(\begin{matrix} a & b \\ c & d \end{matrix}\right)
\in \GL_2(\OF)$, we have:
$$\rho_{E,n}\left(\frac{ax+b}{cx+d}\right) = 
\Vert N_{E/F}(cx+d) \Vert^2 \cdot \rho_{E,n}(x).$$
\item \emph{(Ultimate constancy)}
If $n \geq 2r-1$, then $\rho_{E,n} = \rho_{E,2r-1}$.
\item \emph{(Formulas for extremal degrees)}
If $F[x] = E$ and $x \in \OE$, then
$$\begin{array}{rr@{\hspace{0.5ex}}l}
&
\rho_{E,r}(x) & = \displaystyle 
  \Vert D_K\Vert \cdot \frac 1{\card\big(\OE/\OF[x]\big)}
  \cdot \frac{q^{r+1} - q^r}{q^{r+1} - 1} \medskip \\
\text{for } n \geq 2r - 1, &
\rho_{E,n}(x) & = \displaystyle \Vert D_K\Vert 
  \cdot \int_{\OF[x]} \Vert N_{E/F}(t)\Vert\: dt.
\end{array}$$
\end{enumerate}
\end{theo}

\begin{proof}
The proof follows the same pattern than in the case of extensions.
The first step is to extend Proposition~\ref{prop:limit} and show
that:
\begin{equation}
\label{eq:etale:limit}
\lim_{s \to \infty}\,\,
  \int_{\Omega_n} q^{sr} \cdot \Vert N_{E/F}(P(x)) \Vert \cdot
  \1_{\{\Vert P(x) \Vert \leq q^{-s}\}} \, dP
\,=\, \rho_{E,n}(x)
\end{equation}
when $x$ is new in $E$.
As in the proof of Proposition~\ref{prop:limit}, we write $I_s$ for 
the above integral. Let $Z$ be the minimal monic polynomial of $x$; 
it does not need to be irreducible but it has degree~$r$ since $x$ is
assumed to be new in $E$.
Let $c \in \OF$ be the \emph{content} of $Z$ that is, by definition, the 
gcd of the coefficients of $Z$. Using that the content of a product is 
the product of the contents, we find that a product $QZ \in \OF[x]$ if 
and only if $Q \in c^{-1} \OF[X]$. Moreover, we have seen in the
proof of Proposition~\ref{prop:limit} that there exists a positive
constant $\gamma$ such that $\Vert R(x) \Vert \geq \gamma {\cdot}
\Vert R \Vert$ for all polynomial $R$ of degree at most $r{-}1$.
We deduce from these facts, for any fixed $R$ with $\Vert R 
\Vert \leq \gamma$, the property $QZ + R \in \Omega_n$ is 
equivalent to $Q \in c^{-1} \Omega_{n-r}$. Remarking in addition
that the mapping $(Q,R) \mapsto QZ + R$ preserves the measure
(it is linear and has determinant one in the canonical bases), we 
obtain the equality:
$$I_s 
= q^{sr} \cdot
    \int_{\Omega_{r-1}} \int_{c^{-1} \Omega_{n-r}} \hspace{-1em}
    \Vert N_{E/F}\big(Q(x)Z'(x)+R'(x)\big) \Vert \cdot
    \1_{\{\Vert R(x) \Vert \leq q^{-s}\}} \, dQ\,dR$$
which holds true provided that $s$ is sufficiently large.
Repeating now the argument presented in the proof of
Proposition~\ref{prop:limit}, we end up with:
\begin{align*}
I_s 
& = \Vert D_E \Vert \cdot \lambda_E\big(\OF[x]\big) \cdot
    \int_{c^{-1} \Omega_{n-r}} \hspace{-1em}
    \Vert N_{E/F} (Q(x)) \Vert \cdot dQ \\
& = \Vert D_E \Vert \cdot \lambda_E\big(\OF[x]\big) \cdot
    \Vert c \Vert^{-n-1} \int_{\Omega_{n-r}} \hspace{-1em}
    \Vert N_{E/F} (Q(x)) \Vert \cdot dQ
\end{align*}
for $s$ large enough. 
Consider a
decomposition $E = K_1 \times \cdots \times K_m$ and write $x =
(x_1, \ldots, x_m)$ accordingly. For each index $i$, set $r_i = 
[K_i:F]$ and let $Z_i$ be the minimal monic polynomial of $x_i$.
Then $Z$ is the lcm of the $Z_i$'s and comparing degrees, we
get $Z = Z_1 \cdots Z_m$.
On the other hand, using relations between coefficients and roots,
we find that the coefficient on $Z_i$ in $X^j$ has norm at most
$\Vert x_i \Vert^{r_i-j}$ and equality is reached for $j \in \{0,
r_i\}$.
Therefore, the content $c_i$ of $Z_i$ is~$1$ when $x_i \in \O_{K_i}$
and it is equal to $N_{K_i/F}(x_i)$ otherwise. Hence
$\Vert c_i \Vert = \max\!\big(1,\, \Vert x_i \Vert^{r_i}\big)$
in all cases. By the multiplicativity property of contents, we
conclude that $\Vert c \Vert = H(x)$, which finally establishes
Eq.~\eqref{eq:etale:limit}.
After this result, the proof of the first part of the theorem is 
totally similar to that of Theorem~\ref{theo:meanZnew}.

It remains to establish the properties of the density functions.
Continuity and formulas for extremal degrees are derived exactly
as in the case of extensions (see \S \ref{sssec:continuity} and
\S \ref{sssec:formulas} respectively).
Although the transformation
formula under homography can be tackled as in \S\ref{sssec:homography},
it is probably easier, in the case of étale algebras, to use a
different argument that we present now. First of all, we remark
that, thanks to continuity, the set of equalities:
$$\EE\big[Z^\new_{U,n}\big] \,\, = \,\,
  \int_U \rho_{E,n}(x)\:dx$$
(when $U$ varies) entirely determines the density function 
$\rho_{E,n}$. Given an homography $h : t \mapsto \frac{at+b}
{ct+d}$, it is then enough to prove that
$\EE\big[Z^\new_{U,n}\big] = \EE\big[Z^\new_{h(U),n}\big]$ 
for any open subset $U$ of~$E$, which follows from the fact that the 
transformation:
$$\Omega_n \to \Omega_n, \quad
P(X) \mapsto (cX+d)^n \cdot P\left(\frac{aX+b}{cX+d}\right)$$
preserves the measure.

It finally only remains to prove that $\rho_{E,n} = \rho_{E,2r-1}$ 
as soon as $n \geq 2r-1$. Let $x$ be a new element in $E$. As in
the first part of the proof, we consider the minimal monic polynomial 
$Z \in \OF[X]$ of $x$ and let $c \in \OF$ be its content.
We pick a decomposition $E = K_1 \times \cdots \times K_m$ and we let 
$E_0$ (resp. $E_\infty$) be the product of the $K_i$'s in which $x$ 
falls in the ring of integers (resp. outside the ring of integers).
We thus have the decomposition $E = E_0 \times E_\infty$ and we
write $x = (x_0, x_\infty)$ accordingly. It follows from the
definition of the height function that $H(x) = \Vert N_{E_\infty/F}
(x_\infty) \Vert$. For $t \in \{0, \infty\}$, set $r_t = [E_t:F]$
and let $Z_t$ be the minimal monic polynomial of $x_t$. From the
fact that $x$ is new in $E$, we find $\deg Z = r$, from what we
deduce that $\deg Z_t = r_t$ and $Z = Z_0 Z_\infty$. 
Besides, both polynomials $Z_0$ and $c^{-1} Z_\infty$ have integral 
coefficients. Even better, if we write:
\begin{equation}
\label{eq:Zinfty}
c^{-1} Z_\infty = \lambda_0 + \lambda_1 X + \cdots + 
\lambda_{r_\infty} X^{r_\infty}
\end{equation}
the constant coefficient $\lambda_0$ is invertible in $\OF$ while
the next ones lie in the maximal ideal of
$\OF$, \emph{i.e.} $\Vert \lambda_0 \Vert = 1$ and 
$\Vert \lambda_i \Vert < 1$ for $i \in \{1, \ldots, r_\infty\}$.
For a polynomial $Q \in \Omega_n$, we define
$Q \mod Z_0$ as the remainder of the division of $Q$ by $Z_0$
and $Q \modtau Z_\infty$ as the remainder of the division by 
\emph{increasing power order} of $Q$ by $Z_\infty$. 
The notation$\modtau$comes from the fact that:
$$Q \modtau Z_\infty = \tau\big(\tau(Q) \mod \tau(Z_\infty)\big)$$
where $\tau$ is the involution of $\Omega_n$ taking $P(X)$ to $X^n 
P(X^{-1})$. We derive from the fact that $Z_0 \in \OF[X]$ (resp. from 
Eq.~\eqref{eq:Zinfty}) that $Q \mod Z_0$ (resp. $Q
\modtau Z_\infty$) has integral coefficients when $Q$ has. 
We consider the $\OF$-linear mapping:
$$\begin{array}{rcl}
\varphi : \quad \Omega_n & \longrightarrow
  & (\Omega_{r_0-1}) \times (X^{n-r_\infty+1} \Omega_{r_\infty - 1})
\smallskip \\
Q & \mapsto & \big(Q \mod Z_0, \, Q \modtau Z_\infty\big)
\end{array}$$
We claim that $\varphi$ is surjective. Given $P \in \Omega_{r_0-1}$,
checking that $(P,0)$ is in the image of $\varphi$ amounts to proving 
that there exists a polynomial which is divisible by $Z_\infty$ 
and congruent to $P$ modulo $Z_0$. This follows from the fact that 
$Z_\infty$ is invertible in the quotient $\OF[X]/Z_0$, which is itself 
a consequence of Eq.~\eqref{eq:Zinfty} which implies that the series
$\sum_{i=1}^\infty \big( 1 - c^{-1} Z_\infty \big)^i$
converges in $\OF[X]/Z_0$. Twisting by $\tau$, we 
prove similarly that all elements of the form $(0, X^{n-r_\infty+1}B)$ 
with $B \in \Omega_{r_\infty - 1}$ are attained by $\varphi$. This
gives the surjectivity.

We deduce that $\Omega_n$ can be decomposed (noncanonically) as follows:
$$\Omega_n \simeq \big(\!\ker\varphi\big) 
  \times \big(\Omega_{r_0-1}\big) 
  \times \big(X^{n-r_\infty+1} \Omega_{r_\infty - 1}\big).$$
This isomorphism moreover preserves the measure since it is
$\OF$-linear. Thus, if we set:
$$J_t = 
  \int_{\Omega_{r_t-1}} \hspace{-0.5em} \Vert N_{E_t/F}(Q(x_t)) \Vert \: dQ
\qquad (t \in \{0, \infty\})$$
we get:
\begin{align*}
\frac 1{H(x)^{n+1}}
\int_{\Omega_n} \Vert N_{E/F}(Q(x)) \Vert \: dQ
& \,=\, \frac{\Vert N_{E_\infty/F}(x_\infty) \Vert^{n-r_\infty+1}}{H(x)^{n+1}}
  \cdot J_0 \cdot J_\infty \\
& \,=\, \frac 1{H(x)^{r_\infty}} \cdot J_0 \cdot J_\infty
\end{align*}
which shows that the latter quantity does not depend on $n$ (provided
that $n \geq 2r - 1$) and so neither does $\rho_{n,E}(x)$.
\end{proof}

\subsection{The case of $F^2$}
\label{ssec:F2}

The algebra $E = F^2$ is the simplest example of finite étale algebra 
which is not a field, but it already leads to nontrivial and interesting 
results about the repartition of roots in~$F$ of a random polynomial.
Precisely, it allows us to compute the second momemt and the
covariances between the random variables $Z_{U,n}$ when $U$ is an
open subset of $F$. In what follows, we present a panorama of
results in this direction.

\begin{prop}
\label{prop:rhoF2}
For $x, y \in \OF$, we have:
$$\begin{array}{r@{\quad}r@{\hspace{0.5ex}}l}
& \rho_{F^2,2}(x,y)
& = \displaystyle
    \frac{q^2}{q^2 + q + 1} {\cdot} \Vert x-y \Vert \medskip \\
\text{for } n \geq 3,
& \rho_{F^2,n}(x,y)
& = \displaystyle
    \frac{q^2}{q^2 + q + 1} {\cdot} \Vert x-y \Vert 
  - \frac{q^3}{(q+1)^2\:(q^2+q+1)} {\cdot} \Vert x-y \Vert^4.
\end{array}$$
\end{prop}

\begin{proof}
First of all, observe that $\Vert D_{F^2} \Vert = \Vert D_F \Vert^2
= 1$.
Set $h = x - y$ and let $A$ be the $\OF$-subalgebra of $\OF^2$ 
generated by $(x,y)$. A basis of $A$ is formed by the vectors 
$e_1 = (1, 1)$ and $e_2 = (0, h)$. Hence $\card(\O_{F^2}/A) = \Vert h \Vert^{-1}$ 
and the proposition follows when $n = 2$ using the explicit formulas of 
Theorem~\ref{theo:etale:density}.
For $n \geq 3$, we have to compute the integral of $\Vert N_{F^2/F}(t) 
\Vert$ over~$A$. Applying the linear change of variables $\OF^2
\stackrel\sim\to A$, $(u,v) \mapsto ue_1 + v e_2 = (u, u+hv)$, we
obtain:
\begin{equation}
\label{eq:intA}
\int_A \Vert N_{F^2/F}(t) \Vert \:dt
  = \Vert h \Vert 
    \int_{\OF^2} \Vert u \Vert {\cdot} \Vert u + hv \Vert \: du \: dv.
\end{equation}
In order to compute the latter integral, we first integrate with
respect to the variable~$v$. For a fixed~$u \in \OF$, we claim that:
$$\begin{array}{r@{\hspace{0.5em}}l@{\qquad}l}
\displaystyle \int_{\OF} \Vert u + hv \Vert \: dv
 & = \Vert u \Vert
 & \text{if } \Vert u \Vert > \Vert h \Vert \\
 & = \displaystyle \frac q{q+1} \cdot \Vert h \Vert
 & \text{otherwise.}
\end{array}$$
Indeed, when $\Vert u \Vert > \Vert h \Vert$, the integrand is 
constant equal to $\Vert u \Vert$. On the contrary, when $\Vert u
\Vert \leq \Vert h \Vert$, we can perform the change to variables
$v \mapsto v - h^{-1}u$ and conclude by using Eq.~\eqref{eq:intnorm}.
Injecting the above result in Eq.~\eqref{eq:intA} and decomposing
the integral according to the values of $\Vert u \Vert$, we obtain:
$$\int_A \Vert N_{F^2/F}(t) \Vert \:dt
 = \Vert h \Vert \cdot \left( 1 - \frac 1 q\right) {\cdot}
   \left( \sum_{s=0}^{v-1} q^{-3s}
 \,+\, \frac q{q+1} {\cdot} \Vert h \Vert {\cdot} \sum_{s=v}^\infty q^{-2s}\right)$$
where $v$ is defined by $\Vert h \Vert = q^{-v}$.
A straightforward computation now gives:
$$\int_A \Vert N_{F^2/F}(t) \Vert \:dt
  = \frac{q^2}{q^2 + q + 1} {\cdot} \Vert h \Vert 
  - \frac{q^3}{(q+1)^2\:(q^2+q+1)} {\cdot} \Vert h \Vert^4.$$
which concludes the proof thanks to the formulas for extremal
degrees of Theorem~\ref{theo:etale:density}.
\end{proof}

Using the transformation formulas reported in
Theorem~\ref{theo:etale:density}, we can derive from 
Proposition~\ref{prop:rhoF2} the values of $\rho_{F^2,n}$ on
the whole domain $F^2$. 
Indeed, if $x \in \OF$ and $y \in F \backslash \OF$,
considering the homography $t \mapsto \frac 1{1-x+t}$, we get
$$\rho_{F^2,n}(x,y) 
  = \Vert 1{-}x{+}y \Vert^{-2} \cdot \rho_{F^2,n}(1,y')
  = \Vert y \Vert^{-2} \cdot \rho_{F^2,n}(1,y')$$
with $y' = \frac 1 {1-x+y}$. From the fact that $1-x+y' \not\in \OF$, 
we deduce that $y'$ is in the maximal ideal of $F$ and so $\Vert 1 - 
y'\Vert = 1$. Applying Proposition~\ref{prop:rhoF2}, we finally find:
$$\begin{array}{r@{\quad}r@{\hspace{0.5ex}}l}
& \rho_{F^2,2}(x,y)
& = \displaystyle
    \frac{q^2}{q^2 + q + 1} \cdot \Vert y \Vert^{-2} \medskip \\
\text{for } n \geq 3,
& \rho_{F^2,n}(x,y)
& = \displaystyle
    \frac{q^2}{(q+1)^2} \cdot \Vert y \Vert^{-2}
\end{array}$$
in this case.
Similarly when both $x$ and $y$ do not belong to $\OF$, we use
the homography $t \mapsto t^{-1}$ and get:
$$\begin{array}{r@{\quad}r@{\hspace{0.5ex}}l}
& \rho_{F^2,2}(x,y)
& = \displaystyle
    \frac{q^2}{q^2 + q + 1} {\cdot} 
    \frac{\Vert x-y \Vert}{\Vert x \Vert^2 {\cdot} \Vert y \Vert^2} \medskip \\
\text{for } n \geq 3,
& \rho_{F^2,n}(x,y)
& = \displaystyle
    \frac{q^2}{q^2 + q + 1} {\cdot} 
    \frac{\Vert x-y \Vert}{\Vert x \Vert^2 {\cdot} \Vert y \Vert^2}
  - \frac{q^3}{(q+1)^2\:(q^2+q+1)} {\cdot} 
    \frac{\Vert x-y \Vert}{\Vert x \Vert^5 {\cdot} \Vert y \Vert^5}.
\end{array}$$

\subsubsection*{Applications to covariances}

Proposition~\ref{prop:highermoments} tells us that the integral of 
$\rho_{F^2,n}$ over $F^2$ gives twice the second factorial of the
random variable $Z^\new_{F,n} = Z_{F,n}$ (which was already computed 
in~\cite{BCFG}; it is the value called $\rho(2,n)$ in \emph{loc. cit.}). 
Similarly, it turns out that we can obtain information 
about the variances and covariances of the random variables $Z_{U,n}$
for $U \subset F$ by integrating over smaller domains.

\begin{prop}
\label{prop:meanprod}
For any positive integer $n$ and any open subsets $U$ and $V$
of $F$, we have:
$$\EE\big[Z_{U,n}{\cdot}Z_{V,n}\big] \,=\,
 \int_U \int_V \rho_{F^2,n}(x,y)\:dx\:dy
 \,\,+\,\, \int_{U \cap V} \rho_{F,n}(x)\:dx.$$
\end{prop}

\begin{proof}
The étale algebra $F^2$ admits a unique subalgebra, which is $F$
embedded diagonally. Hence, applying Theorem~\ref{theo:etale:density}
with the open subset $W = U \times V \subset F^2$, we get:
$$\EE\big[Z_{U,n}{\cdot}Z_{V,n}\big] = \EE\big[Z_{W,n}\big] =
 \int_W \rho_{F^2,n}(x)\:dx
 \,\,+\,\, \int_{W \cap F} \rho_{F,n}(x)\:dx.$$
Given that $F$ is embedded diagonally in $F^2$, the
intersection $W \cap F$ is the set of elements $x \in F$ such 
that $(x,x) \in W = U \times V$, \emph{i.e.} $W \cap F = U \cap V$.
The proposition follows.
\end{proof}

When $U$ and $V$ are open balls of $\OF$, one can fully compute 
the integrals of Proposition~\ref{prop:meanprod} and come up 
with closed formulas for the mean of $Z_{U,n}{\cdot} Z_{V,n}$ and then 
for the covariance of $Z_{U,n}$ and $Z_{V,n}$.
The easiest case occurs when $U$ are $V$ are disjoint; indeed, under
this additional assumption, the distance $\Vert x - y \Vert$ does
not vary when $x$ runs over $U$ and $V$ runs over $V$. According
to Proposition~\ref{prop:rhoF2}, the function $\rho_{F^2,n}$ is
then constant over $U \times V$ and it is easy to integrate it.
Doing so, we end up with the formula given in Theorem~\ref{theo:cov}
(in the introduction).

On the contrary, when $V \subset U$, the computation is a bit
more painful but can nevertheless be carried out without trouble.
For $n \geq 3$, the final result we obtain reads:
\begin{align*}
\EE\big[Z_{U,n}{\cdot}Z_{V,n}\big] 
& \,=\, 
  \frac q{q+1} {\cdot} \lambda(V)
  \,+\, 
  \frac {q^3}{(q+1)\:(q^2+q+1)} {\cdot} \lambda(U)^2{\cdot}\lambda(V) \\
& \hspace{3em}-\,
  \frac {q^7}{(q+1)^2\:(q^2+q+1)\:(q^4+q^3+q^2+q+1)} {\cdot} 
    \lambda(U)^5{\cdot}\lambda(V)
\end{align*}
where we recall that $\lambda$ denotes the Haar measure on~$F$.

\subsubsection*{Numerical simulations}

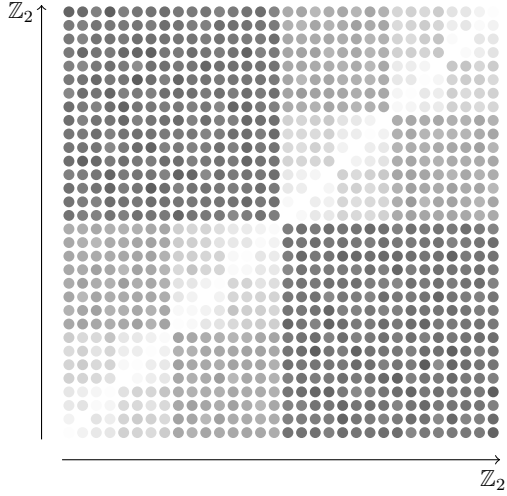
\begin{figure}[t]
\hfill%
\input{statspairs}
\hfill\null
\caption{Repartition of pairs of distinct roots of a polynomial.
Sample of $500,\!000$ polynomials over $\ZZ_2$ picked uniformly at random.}
\label{fig:statspairs}
\end{figure}

As we did in \S \ref{ssec:simulations}, we conducted some numerical 
experients illustrating the results of this subsection: we picked a 
sample of $500,000$ random polynomials over $\ZZ_2$ and located the 
pairs of \emph{distincts} roots of those polynomials (which are exactly 
their new roots in $\QQ_2^2$) in the $2$-adic plane.
The results we obtained are reported on Fig.~\ref{fig:statspairs};
we refer to \S \ref{ssec:simulations} for some explanations on the
way of reading this figure.
In agreement with Proposition~\ref{prop:rhoF2}, we observe that
pairs of roots are less and less numerous when we gettting closer
to the diagonal.

Beyond this, it is also quite interesting to compare 
Fig.~\ref{fig:statspairs} with Fig.~\ref{fig:statsext} (on page
\pageref{fig:statsext}). Indeed, given than $\QQ_2^2$ and 
$\QQ_4$ have both discriminant of norm~$1$, we might expect
at first glance the number of new roots in both algebras to be
comparable. However, looking at the pictures, we clearly 
see that Fig.~\ref{fig:statspairs} is much brighter than its
counterpart. In other words, the conclusion of our numerical
experiments is that there are significantly much more new roots 
in $\QQ_4$ than in $\QQ_2^2$. 
Coming back to Propositions~\ref{prop:rhounram} and~\ref{prop:rhoF2},
we realize that the aforementioned phenomenon can be explained by
looking at the term of higher order (that is the term in
$\dist(x,\QQ_2)^4$ in the case of $\QQ_4$ and the term in
$\Vert x -y\Vert^4$ in the case of $\QQ_2^2$); indeed, this term
contributes positively (\emph{i.e.} it comes with a plus sign) in 
the case of $\QQ_4$ whereas it contributes negatively in the
case of $\QQ_2^2$. When we are sufficiently far away from $\QQ_2$,
this error term is no longer negligible and the change of sign
makes a big difference.

\subsection{A mass formula}
\label{ssec:mass}

In this last subsection, we prove Theorem~\ref{theo:mass}.
For a positive integer $r$, we denote by $\Et_r$ the set of 
isomorphism classes of étale algebras of degree~$r$ over $F$.
To simplify notation, we set:
$$\Sigma(r,n) = 
 \sum_{E \in \Et_r} \frac{\rho_n(E)}{\card \Aut_{\!\Falg}(E)}$$
for all integers $r$ and $n$.
Our goal is to prove that $\Sigma(r,n) = 1$ provided that $n 
\geq r$.
We first consider the extreme case where $n = r$, for which 
we have the following nice reinterpretation of the summands above.

\begin{prop}
\label{prop:proba}
For a finite étale $F$-algebra $E$ of degree~$r$, we have:
$$\frac{\rho_r(E)}{\card \Aut_{\!\Falg}(E)} 
 \,=\, \int_{\Omega_r} \1_{\{F[X]/P \,\simeq\, E\}}\:dP.$$
\end{prop}

\begin{proof}
From Theorem~\ref{theo:etale:density}, we know that:
\begin{equation}
\label{eq:proba}
\rho_r(E) = \EE\big[Z^\new_{E,r}\big] 
 = \int_{\Omega_r} \card \Hom^\surj_{\!\Falg}\big(F[X]/P,\,E\big) \: dP.
\end{equation}
If $P$ is a polynomial of degree~$r$, the quotient algebra
$F[X]/P$ has degree~$r$ as well and so, any surjective 
homomorphism of $F$-algebras
$F[X]/P \to E$ has to be an isomorphism. Consequently, the integrand in 
Eq.~\eqref{eq:proba} is equal to $\card \Aut_{\!\Falg}(E)$ if $F[X]/P 
\simeq E$, and it vanishes otherwise. The proposition follows. 
\end{proof}

Proposition~\ref{prop:proba} tells us that $\rho_r(E)/\card
\Aut_{\!\Falg}(E)$ is exactly the probability that a random polynomial 
$P \in \Omega_r$ satisfies $F[X]/P \simeq E$. On the other hand, note
that the quotient $F[X]/P$ is an étale $F$-algebra of degree~$r$ as soon
as $P$ is separable. This event then occurs almost surely. Therefore
the probabilities that $F[X]/P \simeq E$ sum up to~$1$ when $E$ runs
over~$\Et_r$, \emph{i.e.} $\Sigma(r,r) = 1$.
Theorem~\ref{theo:mass} is then proved when $n = r$.

For higher $n$, the key ingredient of the proof is the following
symmetry result.

\begin{lem}
\label{lem:duality}
For $n > r$, we have $\Sigma(r,n) = \Sigma(n{-}r, n)$.
\end{lem}

\begin{proof}
Unrolling the definitions of $\Sigma(r,n)$ and $\rho_n(E)$, we
obtain:
$$\Sigma(r,n) = \int_{\Omega_n} \sum_{E \in \Et_r}
\frac{\card \Hom^\surj_{\!\Falg}\big(F[X]/P,\,E\big)}
     {\card \Aut_{\!\Falg}(E)} \: dP.$$
Pick a separable polynomial $P \in \Omega_n$ and set $E_P = 
F[X]/P$; it is an étale $F$-algebra of degree $n$. Define
$\Et_r[P]$ as the subset of $\Et_r$ consisting of étale algebras
$E$ for which there exists a surjective morphism of $F$-algebras
$E_P \to E$. Let $E \in \Et_r[P]$
and choose writings $E_P = K_1^{a_1} \times \cdots \times K_m^{a_m}$ 
and $E = K_1^{b_1} \times \cdots \times K_m^{b_m}$ where the $K_i$'s 
are pairwise nonisomorphic finite extensions of $F$ and $a_i, b_i \in 
\ZZ_{\geq 0}$.
From Corollary~\ref{cor:etale:surj}, we deduce that $a_i \geq b_i$
for all $i \in \{1, \ldots, m\}$ and that:
$$\frac{\card \Hom^\surj_{\!\Falg}(E_P,\,E)}
       {\card \Aut_{\!\Falg}(E)}
= \binom{a_1}{b_1} \cdot \binom{a_2}{b_2} \cdots \binom{a_m}{b_m}.$$
Hence if we define 
$E^\vee = K_1^{a_1-b_1} \times \cdots \times K_m^{a_m-b_m}$, we find:
$$\frac{\card \Hom^\surj_{\!\Falg}(E_P,\,E)}
       {\card \Aut_{\!\Falg}(E)}
= \frac{\card \Hom^\surj_{\!\Falg}(E_P,\,E^\vee)}
       {\card \Aut_{\!\Falg}(E^\vee)}$$
Moreover, it follows again from Corollary~\ref{cor:etale:surj}
that the association $E \mapsto E^\vee$ induces a well-defined function
$\Et_r[P] \to \Et_{n-r}[P]$. This function is moreover bijective
because its inverse can be built in a similar fashion. 
We conclude that, for a fixed separable polynomial $P \in 
\Omega_n$, we have:
$$\sum_{E \in \Et_r[P]}
     \frac{\card \Hom^\surj_{\!\Falg}(E_P,\,E)}
          {\card \Aut_{\!\Falg}(E)}
\,= \sum_{E^\vee \in \Et_{n-r}[P]} \hspace{-0.5em}
     \frac{\card \Hom^\surj_{\!\Falg}(E_P,\,E^\vee)}
          {\card \Aut_{\!\Falg}(E^\vee)}$$
which further gives:
$$\sum_{E \in \Et_r}
     \frac{\card \Hom^\surj_{\!\Falg}(E_P,\,E)}
          {\card \Aut_{\!\Falg}(E)}
\,= \sum_{E^\vee \in \Et_{n-r}}
     \frac{\card \Hom^\surj_{\!\Falg}(E_P,\,E^\vee)}
          {\card \Aut_{\!\Falg}(E^\vee)}$$
since all the additional summands which appear in both sums are zero.
Taking finally the integral over $\Omega_n$ and remembering that a
polynomial in $\Omega_n$ is almost surely separable, we obtain the
lemma.
\end{proof}

After Lemma~\ref{lem:duality}, it is easy to conclude the proof of 
Theorem~\ref{theo:mass} by induction on $n$. When $n = 1$, the 
condition $n \geq r$ indicates that $r = 1$ as well and we fall in
the case where $n = r$, which has been already treated.
We now pick an integer $n > 1$ and assume that $\Sigma(r,m) = 1$
provided that $1 \leq r \leq m < n$. Let also $r \in \{1, \ldots, n\}$.
If $n = r$, we have already seen that $\Sigma(r,n) = 1$ and there
is nothing more to prove.
If $r \leq n/2$, it follows from the fourth property of 
Theorem~\ref{theo:etale:density} that $\Sigma(r,n) = \Sigma(r, 
2r{-}1)$, from what we deduce that $\Sigma(r,n) = 1$ thanks to 
our induction hypothesis. Finally, if $n/2 \leq r < n$,
we use Lemma~\ref{lem:duality} to write $\Sigma(r,n) = 
\Sigma(n{-}r, n)$ and, observing that $n-r \leq n/2$, we conclude
by applying the previous case.

\end{document}

%% file: stats2.tex
\begin{tikzpicture}[xscale=1.3,yscale=0.035]
\draw[-latex] (-1.5,0)--(-1.5,340);
\draw (-1.6,0)--(-1.4,0);
\draw (-1.6,100)--(-1.4,100);
\draw (-1.6,200)--(-1.4,200);
\draw (-1.6,300)--(-1.4,300);
\draw[dotted] (-1.4,100)--(0,100);
\draw[dotted] (-1.4,200)--(1.5,200);
\draw[dotted] (-1.4,300)--(3,300);
\node at (-1.8,0) { $0$ };
\node at (-1.8,100) { $1$ };
\node at (-1.8,200) { $2$ };
\node at (-1.8,300) { $3$ };

\draw[transparent] (9,0)--(9,330);

\fill[K] (0,0) rectangle (1,100);
\draw (0,0) rectangle (1,100);
\node at (0.5,-15) { deg $1$ };

\begin{scope}[xshift=1.5cm]
\fill[K] (0,0) rectangle (1,99.50);
\fill[K2] (0,99.50) rectangle (1,142.48);
\fill[E2] (0,142.48) rectangle (1,171.90);
\fill[E2b] (0,171.90) rectangle (1,200);
\draw (0,0) rectangle (1,200);
\draw (0,99.50)--(1,99.50);
\draw (0,142.48)--(1,142.48);
\draw[very thin] (0,157.50)--(1,157.50);
\draw (0,171.90)--(1,171.90);
\draw[very thin] (0,178.52)--(1,178.52);
\draw[very thin] (0,186.38)--(1,186.38);
\draw[very thin] (0,193.26)--(1,193.26);
\node at (0.5,-15) { deg $2$ };
\end{scope}

\begin{scope}[xshift=3cm]
\fill[K] (0,0) rectangle (1,100.12);
\fill[K2] (0,100.12) rectangle (1,153.60);
\fill[E2] (0,153.60) rectangle (1,188.26);
\fill[E2b] (0,188.26) rectangle (1,219.72);
\fill[K3] (0,219.72) rectangle (1,261.66);
\fill[E3] (0,261.66) rectangle (1,300);
\draw (0,0) rectangle (1,300);
\draw (0,100.12)--(1,100.12);
\draw (0,153.60)--(1,153.60);
\draw[very thin] (0,171.14)--(1,171.14);
\draw (0,188.26)--(1,188.26);
\draw[very thin] (0,196.10)--(1,196.10);
\draw[very thin] (0,203.70)--(1,203.70);
\draw[very thin] (0,211.02)--(1,211.02);
\draw (0,219.72)--(1,219.72);
\draw (0,261.66)--(1,261.66);
\draw[very thin, opacity=0.4] (0.33,261.66)--(0.33,300);
\draw[very thin, opacity=0.4] (0.67,261.66)--(0.67,300);
\node at (0.5,-15) { deg $3$ };
\end{scope}

\begin{scope}[xshift=4.5cm]
\fill[K] (0,0) rectangle (1,99.40);
\fill[K2] (0,99.40) rectangle (1,155.28);
\fill[E2] (0,155.28) rectangle (1,188.10);
\fill[E2b] (0,188.10) rectangle (1,221.14);
\fill[K3] (0,221.14) rectangle (1,280.59);
\fill[E3] (0,280.59) rectangle (1,324.66);
\draw (0,0) rectangle (1,324.66);
\draw (0,99.40)--(1,99.40);
\draw (0,155.28)--(1,155.28);
\draw[very thin] (0,171.54)--(1,171.54);
\draw (0,188.10)--(1,188.10);
\draw[very thin] (0,196.68)--(1,196.68);
\draw[very thin] (0,204.66)--(1,204.66);
\draw[very thin] (0,212.76)--(1,212.76);
\draw (0,221.14)--(1,221.14);
\draw (0,280.59)--(1,280.59);
\draw[very thin, opacity=0.4] (0.33,280.59)--(0.33,324.66);
\draw[very thin, opacity=0.4] (0.67,280.59)--(0.67,324.66);
\node at (0.5,-15) { deg $4$ };
\end{scope}

\begin{scope}[xshift=6cm]
\fill[K] (0,0) rectangle (1,99.31);
\fill[K2] (0,99.31) rectangle (1,153.33);
\fill[E2] (0,153.33) rectangle (1,184.95);
\fill[E2b] (0,184.95) rectangle (1,217.11);
\fill[K3] (0,217.11) rectangle (1,285.90);
\fill[E3] (0,285.90) rectangle (1,334.50);
\draw (0,0) rectangle (1,334.50);
\draw (0,99.31)--(1,99.31);
\draw (0,153.55)--(1,153.55);
\draw[very thin] (0,168.97)--(1,168.97);
\draw (0,184.95)--(1,184.95);
\draw[very thin] (0,192.19)--(1,192.19);
\draw[very thin] (0,200.39)--(1,200.39);
\draw[very thin] (0,208.51)--(1,208.51);
\draw (0,217.11)--(1,217.11);
\draw (0,285.90)--(1,285.90);
\draw[very thin, opacity=0.4] (0.33,285.90)--(0.33,334.50);
\draw[very thin, opacity=0.4] (0.67,285.90)--(0.67,334.50);
\node at (0.5,-15) { deg $5$ };
\end{scope}

\begin{scope}[xshift=7.5cm]
\fill[K] (0,0) rectangle (1,99.94);
\fill[K2] (0,99.94) rectangle (1,153.70);
\fill[E2] (0,153.70) rectangle (1,186.86);
\fill[E2b] (0,186.86) rectangle (1,218.56);
\fill[K3] (0,218.56) rectangle (1,286.06);
\fill[E3] (0,286.06) rectangle (1,333.88);
\draw (0,0) rectangle (1,333.88);
\draw (0,99.94)--(1,99.94);
\draw (0,153.70)--(1,153.70);
\draw[very thin] (0,170.38)--(1,170.38);
\draw (0,186.86)--(1,186.86);
\draw[very thin] (0,194.68)--(1,194.68);
\draw[very thin] (0,202.60)--(1,202.60);
\draw[very thin] (0,209.98)--(1,209.98);
\draw (0,218.56)--(1,218.56);
\draw (0,286.06)--(1,286.06);
\draw[very thin, opacity=0.4] (0.33,286.06)--(0.33,333.88);
\draw[very thin, opacity=0.4] (0.67,286.06)--(0.67,333.88);
\node at (0.5,-15) { deg $6$ };
\end{scope}

\begin{scope}[yscale=20,xshift=-1cm,yshift=12.6cm]
\fill[black!20,rounded corners=1mm] (0.05,-1.4) rectangle (3.05,3.9);
\draw[thick, fill=green!5, rounded corners=1mm] (0,-1.3) rectangle (3,4);
\node at (1.5, 3.4) { \sc Legend };
\draw[fill=K] (0.2,2.2) rectangle (0.6,2.7);
\node[right,scale=0.8] at (0.65,2.45) { $\QQ_2$ };
\draw[fill=K2] (0.2,1.3) rectangle (0.6,1.8);
\node[right,scale=0.8] at (0.65,1.55) { $\QQ_{2^2}$ };
\draw[fill=E2] (0.2,0) rectangle (0.6,1);
\draw[very thin] (0.2,0.5)--(0.6,0.5);
\node[right,scale=0.75] at (0.65,0.25) { $\QQ_2[\sqrt 3]$ };
\node[right,scale=0.75] at (0.65,0.75) { $\QQ_2[\sqrt 7]$ };
\draw[fill=E2b] (1.5,0) rectangle (1.9,2);
\draw[very thin] (1.5,0.5)--(1.9,0.5);
\draw[very thin] (1.5,1)--(1.9,1);
\draw[very thin] (1.5,1.5)--(1.9,1.5);
\node[right,scale=0.75] at (1.95,0.25) { $\QQ_2[\sqrt 2]$ };
\node[right,scale=0.75] at (1.95,0.75) { $\QQ_2[\sqrt 6]$ };
\node[right,scale=0.75] at (1.95,1.25) { $\QQ_2[\sqrt {10}]$ };
\node[right,scale=0.75] at (1.95,1.75) { $\QQ_2[\sqrt {14}]$ };
\draw[fill=K3] (0.2,-0.4) rectangle (0.6,-0.9);
\node[right,scale=0.8] at (0.65,-0.65) { $\QQ_{2^3}$ };
\draw[fill=E3] (1.5,-0.4) rectangle (1.9,-0.9);
\node[right,scale=0.8] at (1.95,-0.65) { $\QQ_2[\sqrt[3] 2]$ };
\draw[very thin, opacity=0.4] (1.633,-0.4)--(1.633,-0.9);
\draw[very thin, opacity=0.4] (1.767,-0.4)--(1.767,-0.9);
\end{scope}

\end{tikzpicture}

%% file: stats5.tex
\begin{tikzpicture}[xscale=1.3,yscale=0.035]
\draw[-latex] (-1.5,0)--(-1.5,340);
\draw (-1.6,0)--(-1.4,0);
\draw (-1.6,100)--(-1.4,100);
\draw (-1.6,200)--(-1.4,200);
\draw (-1.6,300)--(-1.4,300);
\draw[dotted] (-1.4,100)--(0,100);
\draw[dotted] (-1.4,200)--(1.5,200);
\draw[dotted] (-1.4,300)--(3,300);
\node at (-1.8,0) { $0$ };
\node at (-1.8,100) { $1$ };
\node at (-1.8,200) { $2$ };
\node at (-1.8,300) { $3$ };

\draw[transparent] (9,0)--(9,330);

\fill[K] (0,0) rectangle (1,100);
\draw (0,0) rectangle (1,100);
\node at (0.5,-15) { deg $1$ };

\begin{scope}[xshift=1.5cm]
\fill[K] (0,0) rectangle (1,100.54);
\fill[K2] (0,100.54) rectangle (1,168.58);
\fill[E2] (0,168.58) rectangle (1,200);
\draw (0,0) rectangle (1,200);
\draw (0,100.54) rectangle (1,100.54);
\draw (0,168.58) rectangle (1,168.58);
\draw[very thin] (0,183.94) rectangle (1,183.94);
\node at (0.5,-15) { deg $2$ };
\end{scope}

\begin{scope}[xshift=3cm]
\fill[K] (0,0) rectangle (1,100.53);
\fill[K2] (0,100.53) rectangle (1,180.51);
\fill[E2] (0,180.51) rectangle (1,215.25);
\fill[K3] (0,215.25) rectangle (1,290.76);
\fill[E3] (0,290.76) rectangle (1,300);
\draw (0,0) rectangle (1,300);
\draw (0,100.53) rectangle (1,100.53);
\draw (0,180.51) rectangle (1,180.51);
\draw[very thin] (0,198.27) rectangle (1,198.27);
\draw (0,215.25) rectangle (1,215.25);
\draw (0,290.76) rectangle (1,290.76);
\draw[very thin, opacity=0.4] (0.33,290.76)--(0.33,300);
\draw[very thin, opacity=0.4] (0.67,290.76)--(0.67,300);
\node at (0.5,-15) { deg $3$ };
\end{scope}

\begin{scope}[xshift=4.5cm]
\fill[K] (0,0) rectangle (1,101.79);
\fill[K2] (0,101.79) rectangle (1,182.65);
\fill[E2] (0,182.65) rectangle (1,216.01);
\fill[K3] (0,216.01) rectangle (1,306.61);
\fill[E3] (0,306.61) rectangle (1,316.60);
\draw (0,0) rectangle (1,316.60);
\draw (0,101.79) rectangle (1,101.79);
\draw (0,182.65) rectangle (1,182.65);
\draw[very thin] (0,199.67) rectangle (1,199.67);
\draw (0,216.01) rectangle (1,216.01);
\draw (0,306.61) rectangle (1,306.61);
\draw[very thin, opacity=0.4] (0.33,306.61)--(0.33,316.60);
\draw[very thin, opacity=0.4] (0.67,306.61)--(0.67,316.60);
\node at (0.5,-15) { deg $4$ };
\end{scope}

\begin{scope}[xshift=6cm]
\fill[K] (0,0) rectangle (1,99.12);
\fill[K2] (0,99.12) rectangle (1,180.10);
\fill[E2] (0,180.10) rectangle (1,214.44);
\fill[K3] (0,214.44) rectangle (1,309.66);
\fill[E3] (0,309.66) rectangle (1,319.77);
\draw (0,0) rectangle (1,319.77);
\draw (0,99.12) rectangle (1,99.12);
\draw (0,180.10) rectangle (1,180.10);
\draw[very thin] (0,197.06) rectangle (1,197.06);
\draw (0,214.44) rectangle (1,214.44);
\draw (0,309.66) rectangle (1,309.66);
\draw[very thin, opacity=0.4] (0.33,309.66)--(0.33,319.77);
\draw[very thin, opacity=0.4] (0.67,309.66)--(0.67,319.77);
\node at (0.5,-15) { deg $5$ };
\end{scope}

\begin{scope}[xshift=7.5cm]
\fill[K] (0,0) rectangle (1,100.82);
\fill[K2] (0,100.82) rectangle (1,181.16);
\fill[E2] (0,181.16) rectangle (1,213.14);
\fill[K3] (0,213.14) rectangle (1,308.75);
\fill[E3] (0,308.75) rectangle (1,319.25);
\draw (0,0) rectangle (1,319.25);
\draw (0,100.82) rectangle (1,100.82);
\draw (0,181.16) rectangle (1,181.16);
\draw[very thin] (0,197.36) rectangle (1,197.36);
\draw (0,213.14) rectangle (1,213.14);
\draw (0,308.75) rectangle (1,308.75);
\draw[very thin, opacity=0.4] (0.33,308.75)--(0.33,319.25);
\draw[very thin, opacity=0.4] (0.67,308.75)--(0.67,319.25);
\node at (0.5,-15) { deg $6$ };
\end{scope}

\begin{scope}[yscale=20,xshift=-1cm,yshift=12cm]
\fill[black!20,rounded corners=1mm] (0.05,-0.1) rectangle (3.05,3.9);
\draw[thick, fill=green!5, rounded corners=1mm] (0,0) rectangle (3,4);
\node at (1.5, 3.4) { \sc Legend };
\draw[fill=K] (0.2,2.2) rectangle (0.6,2.7);
\node[right,scale=0.8] at (0.65,2.45) { $\QQ_5$ };
\draw[fill=K2] (0.2,1.3) rectangle (0.6,1.8);
\node[right,scale=0.8] at (0.65,1.55) { $\QQ_{5^2}$ };
\draw[fill=E2] (1.4,1.3) rectangle (1.8,2.3);
\draw[very thin] (1.4,1.8)--(1.8,1.8);
\node[right,scale=0.75] at (1.85,1.55) { $\QQ_5[\sqrt 5]$ };
\node[right,scale=0.75] at (1.85,2.05) { $\QQ_5[\sqrt {10}]$ };
\draw[fill=K3] (0.2,0.4) rectangle (0.6,0.9);
\node[right,scale=0.8] at (0.65,0.65) { $\QQ_{5^3}$ };
\draw[fill=E3] (1.4,0.4) rectangle (1.8,0.9);
\node[right,scale=0.8] at (1.85,0.65) { $\QQ_5[\sqrt[3] 5]$ };
\draw[very thin, opacity=0.4] (1.533,0.4)--(1.533,0.9);
\draw[very thin, opacity=0.4] (1.667,0.4)--(1.667,0.9);
\end{scope}

\end{tikzpicture}

%% file: statsunram.tex
\begin{tikzpicture}[scale=0.18]
\draw[->] (-0.5,-2)--(31.5,-2);
\draw[->] (-2,-0.5)--(-2,31.5);
\node[scale=0.8] at (31,-3.5) { $\ZZ_2$ };
\node[scale=0.8] at (-4.5,31) { $\ZZ_2{\cdot} \zeta$ };
\fill[black!4] (0,0) circle (4mm);
\fill[black] (0,16) circle (4mm);
\fill[black!34] (0,8) circle (4mm);
\fill[black] (0,24) circle (4mm);
\fill[black!16] (0,4) circle (4mm);
\fill[black!92] (0,20) circle (4mm);
\fill[black!38] (0,12) circle (4mm);
\fill[black!95] (0,28) circle (4mm);
\fill[black!7] (0,2) circle (4mm);
\fill[black!91] (0,18) circle (4mm);
\fill[black!38] (0,10) circle (4mm);
\fill[black!94] (0,26) circle (4mm);
\fill[black!17] (0,6) circle (4mm);
\fill[black!99] (0,22) circle (4mm);
\fill[black!35] (0,14) circle (4mm);
\fill[black!99] (0,30) circle (4mm);
\fill[black!4] (0,1) circle (4mm);
\fill[black!97] (0,17) circle (4mm);
\fill[black!37] (0,9) circle (4mm);
\fill[black!87] (0,25) circle (4mm);
\fill[black!20] (0,5) circle (4mm);
\fill[black!90] (0,21) circle (4mm);
\fill[black!39] (0,13) circle (4mm);
\fill[black!92] (0,29) circle (4mm);
\fill[black!7] (0,3) circle (4mm);
\fill[black!92] (0,19) circle (4mm);
\fill[black!35] (0,11) circle (4mm);
\fill[black!98] (0,27) circle (4mm);
\fill[black!14] (0,7) circle (4mm);
\fill[black!95] (0,23) circle (4mm);
\fill[black!39] (0,15) circle (4mm);
\fill[black!96] (0,31) circle (4mm);
\fill[black!1] (16,0) circle (4mm);
\fill[black!96] (16,16) circle (4mm);
\fill[black!40] (16,8) circle (4mm);
\fill[black!91] (16,24) circle (4mm);
\fill[black!20] (16,4) circle (4mm);
\fill[black!89] (16,20) circle (4mm);
\fill[black!36] (16,12) circle (4mm);
\fill[black!98] (16,28) circle (4mm);
\fill[black!11] (16,2) circle (4mm);
\fill[black!95] (16,18) circle (4mm);
\fill[black!38] (16,10) circle (4mm);
\fill[black!98] (16,26) circle (4mm);
\fill[black!20] (16,6) circle (4mm);
\fill[black!97] (16,22) circle (4mm);
\fill[black!35] (16,14) circle (4mm);
\fill[black!92] (16,30) circle (4mm);
\fill[black!3] (16,1) circle (4mm);
\fill[black] (16,17) circle (4mm);
\fill[black!31] (16,9) circle (4mm);
\fill[black!99] (16,25) circle (4mm);
\fill[black!15] (16,5) circle (4mm);
\fill[black] (16,21) circle (4mm);
\fill[black!35] (16,13) circle (4mm);
\fill[black!94] (16,29) circle (4mm);
\fill[black!7] (16,3) circle (4mm);
\fill[black] (16,19) circle (4mm);
\fill[black!36] (16,11) circle (4mm);
\fill[black!99] (16,27) circle (4mm);
\fill[black!18] (16,7) circle (4mm);
\fill[black] (16,23) circle (4mm);
\fill[black!32] (16,15) circle (4mm);
\fill[black] (16,31) circle (4mm);
\fill[black!0] (8,0) circle (4mm);
\fill[black] (8,16) circle (4mm);
\fill[black!39] (8,8) circle (4mm);
\fill[black!98] (8,24) circle (4mm);
\fill[black!18] (8,4) circle (4mm);
\fill[black!91] (8,20) circle (4mm);
\fill[black!35] (8,12) circle (4mm);
\fill[black] (8,28) circle (4mm);
\fill[black!9] (8,2) circle (4mm);
\fill[black!94] (8,18) circle (4mm);
\fill[black!36] (8,10) circle (4mm);
\fill[black] (8,26) circle (4mm);
\fill[black!18] (8,6) circle (4mm);
\fill[black] (8,22) circle (4mm);
\fill[black!38] (8,14) circle (4mm);
\fill[black!96] (8,30) circle (4mm);
\fill[black!4] (8,1) circle (4mm);
\fill[black!98] (8,17) circle (4mm);
\fill[black!34] (8,9) circle (4mm);
\fill[black!96] (8,25) circle (4mm);
\fill[black!19] (8,5) circle (4mm);
\fill[black!95] (8,21) circle (4mm);
\fill[black!30] (8,13) circle (4mm);
\fill[black!93] (8,29) circle (4mm);
\fill[black!9] (8,3) circle (4mm);
\fill[black!96] (8,19) circle (4mm);
\fill[black!37] (8,11) circle (4mm);
\fill[black!99] (8,27) circle (4mm);
\fill[black!15] (8,7) circle (4mm);
\fill[black!95] (8,23) circle (4mm);
\fill[black!35] (8,15) circle (4mm);
\fill[black!90] (8,31) circle (4mm);
\fill[black!3] (24,0) circle (4mm);
\fill[black!90] (24,16) circle (4mm);
\fill[black!32] (24,8) circle (4mm);
\fill[black!95] (24,24) circle (4mm);
\fill[black!24] (24,4) circle (4mm);
\fill[black!99] (24,20) circle (4mm);
\fill[black!41] (24,12) circle (4mm);
\fill[black!94] (24,28) circle (4mm);
\fill[black!9] (24,2) circle (4mm);
\fill[black!91] (24,18) circle (4mm);
\fill[black!35] (24,10) circle (4mm);
\fill[black!92] (24,26) circle (4mm);
\fill[black!16] (24,6) circle (4mm);
\fill[black] (24,22) circle (4mm);
\fill[black!40] (24,14) circle (4mm);
\fill[black!98] (24,30) circle (4mm);
\fill[black!4] (24,1) circle (4mm);
\fill[black!93] (24,17) circle (4mm);
\fill[black!34] (24,9) circle (4mm);
\fill[black!91] (24,25) circle (4mm);
\fill[black!19] (24,5) circle (4mm);
\fill[black!95] (24,21) circle (4mm);
\fill[black!39] (24,13) circle (4mm);
\fill[black] (24,29) circle (4mm);
\fill[black!9] (24,3) circle (4mm);
\fill[black] (24,19) circle (4mm);
\fill[black!39] (24,11) circle (4mm);
\fill[black] (24,27) circle (4mm);
\fill[black!18] (24,7) circle (4mm);
\fill[black!99] (24,23) circle (4mm);
\fill[black!38] (24,15) circle (4mm);
\fill[black!90] (24,31) circle (4mm);
\fill[black!1] (4,0) circle (4mm);
\fill[black!90] (4,16) circle (4mm);
\fill[black!35] (4,8) circle (4mm);
\fill[black] (4,24) circle (4mm);
\fill[black!14] (4,4) circle (4mm);
\fill[black!93] (4,20) circle (4mm);
\fill[black!35] (4,12) circle (4mm);
\fill[black!94] (4,28) circle (4mm);
\fill[black!7] (4,2) circle (4mm);
\fill[black!99] (4,18) circle (4mm);
\fill[black!36] (4,10) circle (4mm);
\fill[black!97] (4,26) circle (4mm);
\fill[black!16] (4,6) circle (4mm);
\fill[black!97] (4,22) circle (4mm);
\fill[black!41] (4,14) circle (4mm);
\fill[black!97] (4,30) circle (4mm);
\fill[black!5] (4,1) circle (4mm);
\fill[black] (4,17) circle (4mm);
\fill[black!38] (4,9) circle (4mm);
\fill[black] (4,25) circle (4mm);
\fill[black!19] (4,5) circle (4mm);
\fill[black!90] (4,21) circle (4mm);
\fill[black!36] (4,13) circle (4mm);
\fill[black] (4,29) circle (4mm);
\fill[black!8] (4,3) circle (4mm);
\fill[black] (4,19) circle (4mm);
\fill[black!42] (4,11) circle (4mm);
\fill[black] (4,27) circle (4mm);
\fill[black!17] (4,7) circle (4mm);
\fill[black] (4,23) circle (4mm);
\fill[black!37] (4,15) circle (4mm);
\fill[black!90] (4,31) circle (4mm);
\fill[black!1] (20,0) circle (4mm);
\fill[black!90] (20,16) circle (4mm);
\fill[black!38] (20,8) circle (4mm);
\fill[black!95] (20,24) circle (4mm);
\fill[black!18] (20,4) circle (4mm);
\fill[black!98] (20,20) circle (4mm);
\fill[black!39] (20,12) circle (4mm);
\fill[black] (20,28) circle (4mm);
\fill[black!9] (20,2) circle (4mm);
\fill[black!96] (20,18) circle (4mm);
\fill[black!39] (20,10) circle (4mm);
\fill[black!96] (20,26) circle (4mm);
\fill[black!16] (20,6) circle (4mm);
\fill[black] (20,22) circle (4mm);
\fill[black!42] (20,14) circle (4mm);
\fill[black!91] (20,30) circle (4mm);
\fill[black!5] (20,1) circle (4mm);
\fill[black!96] (20,17) circle (4mm);
\fill[black!39] (20,9) circle (4mm);
\fill[black] (20,25) circle (4mm);
\fill[black!17] (20,5) circle (4mm);
\fill[black!95] (20,21) circle (4mm);
\fill[black!39] (20,13) circle (4mm);
\fill[black] (20,29) circle (4mm);
\fill[black!7] (20,3) circle (4mm);
\fill[black!96] (20,19) circle (4mm);
\fill[black!37] (20,11) circle (4mm);
\fill[black] (20,27) circle (4mm);
\fill[black!19] (20,7) circle (4mm);
\fill[black] (20,23) circle (4mm);
\fill[black!35] (20,15) circle (4mm);
\fill[black!95] (20,31) circle (4mm);
\fill[black!0] (12,0) circle (4mm);
\fill[black!95] (12,16) circle (4mm);
\fill[black!37] (12,8) circle (4mm);
\fill[black!99] (12,24) circle (4mm);
\fill[black!15] (12,4) circle (4mm);
\fill[black!99] (12,20) circle (4mm);
\fill[black!35] (12,12) circle (4mm);
\fill[black!99] (12,28) circle (4mm);
\fill[black!7] (12,2) circle (4mm);
\fill[black!99] (12,18) circle (4mm);
\fill[black!37] (12,10) circle (4mm);
\fill[black] (12,26) circle (4mm);
\fill[black!19] (12,6) circle (4mm);
\fill[black!96] (12,22) circle (4mm);
\fill[black!41] (12,14) circle (4mm);
\fill[black!95] (12,30) circle (4mm);
\fill[black!4] (12,1) circle (4mm);
\fill[black] (12,17) circle (4mm);
\fill[black!33] (12,9) circle (4mm);
\fill[black!82] (12,25) circle (4mm);
\fill[black!21] (12,5) circle (4mm);
\fill[black!93] (12,21) circle (4mm);
\fill[black!34] (12,13) circle (4mm);
\fill[black!90] (12,29) circle (4mm);
\fill[black!7] (12,3) circle (4mm);
\fill[black!95] (12,19) circle (4mm);
\fill[black!35] (12,11) circle (4mm);
\fill[black] (12,27) circle (4mm);
\fill[black!16] (12,7) circle (4mm);
\fill[black] (12,23) circle (4mm);
\fill[black!40] (12,15) circle (4mm);
\fill[black!95] (12,31) circle (4mm);
\fill[black!1] (28,0) circle (4mm);
\fill[black!95] (28,16) circle (4mm);
\fill[black!35] (28,8) circle (4mm);
\fill[black] (28,24) circle (4mm);
\fill[black!18] (28,4) circle (4mm);
\fill[black!99] (28,20) circle (4mm);
\fill[black!36] (28,12) circle (4mm);
\fill[black!92] (28,28) circle (4mm);
\fill[black!7] (28,2) circle (4mm);
\fill[black] (28,18) circle (4mm);
\fill[black!32] (28,10) circle (4mm);
\fill[black!96] (28,26) circle (4mm);
\fill[black!20] (28,6) circle (4mm);
\fill[black!94] (28,22) circle (4mm);
\fill[black!35] (28,14) circle (4mm);
\fill[black!92] (28,30) circle (4mm);
\fill[black!3] (28,1) circle (4mm);
\fill[black] (28,17) circle (4mm);
\fill[black!35] (28,9) circle (4mm);
\fill[black!90] (28,25) circle (4mm);
\fill[black!17] (28,5) circle (4mm);
\fill[black!94] (28,21) circle (4mm);
\fill[black!37] (28,13) circle (4mm);
\fill[black!98] (28,29) circle (4mm);
\fill[black!10] (28,3) circle (4mm);
\fill[black] (28,19) circle (4mm);
\fill[black!36] (28,11) circle (4mm);
\fill[black] (28,27) circle (4mm);
\fill[black!17] (28,7) circle (4mm);
\fill[black!93] (28,23) circle (4mm);
\fill[black!38] (28,15) circle (4mm);
\fill[black] (28,31) circle (4mm);
\fill[black!1] (2,0) circle (4mm);
\fill[black] (2,16) circle (4mm);
\fill[black!40] (2,8) circle (4mm);
\fill[black] (2,24) circle (4mm);
\fill[black!17] (2,4) circle (4mm);
\fill[black] (2,20) circle (4mm);
\fill[black!37] (2,12) circle (4mm);
\fill[black] (2,28) circle (4mm);
\fill[black!7] (2,2) circle (4mm);
\fill[black!91] (2,18) circle (4mm);
\fill[black!39] (2,10) circle (4mm);
\fill[black!99] (2,26) circle (4mm);
\fill[black!17] (2,6) circle (4mm);
\fill[black] (2,22) circle (4mm);
\fill[black!38] (2,14) circle (4mm);
\fill[black!95] (2,30) circle (4mm);
\fill[black!4] (2,1) circle (4mm);
\fill[black!94] (2,17) circle (4mm);
\fill[black!35] (2,9) circle (4mm);
\fill[black!99] (2,25) circle (4mm);
\fill[black!16] (2,5) circle (4mm);
\fill[black!95] (2,21) circle (4mm);
\fill[black!30] (2,13) circle (4mm);
\fill[black] (2,29) circle (4mm);
\fill[black!8] (2,3) circle (4mm);
\fill[black!97] (2,19) circle (4mm);
\fill[black!36] (2,11) circle (4mm);
\fill[black!98] (2,27) circle (4mm);
\fill[black!18] (2,7) circle (4mm);
\fill[black] (2,23) circle (4mm);
\fill[black!41] (2,15) circle (4mm);
\fill[black!91] (2,31) circle (4mm);
\fill[black!3] (18,0) circle (4mm);
\fill[black!91] (18,16) circle (4mm);
\fill[black!38] (18,8) circle (4mm);
\fill[black] (18,24) circle (4mm);
\fill[black!19] (18,4) circle (4mm);
\fill[black] (18,20) circle (4mm);
\fill[black!39] (18,12) circle (4mm);
\fill[black!96] (18,28) circle (4mm);
\fill[black!7] (18,2) circle (4mm);
\fill[black!92] (18,18) circle (4mm);
\fill[black!46] (18,10) circle (4mm);
\fill[black] (18,26) circle (4mm);
\fill[black!17] (18,6) circle (4mm);
\fill[black!99] (18,22) circle (4mm);
\fill[black!37] (18,14) circle (4mm);
\fill[black!95] (18,30) circle (4mm);
\fill[black!3] (18,1) circle (4mm);
\fill[black] (18,17) circle (4mm);
\fill[black!32] (18,9) circle (4mm);
\fill[black!86] (18,25) circle (4mm);
\fill[black!20] (18,5) circle (4mm);
\fill[black!91] (18,21) circle (4mm);
\fill[black!39] (18,13) circle (4mm);
\fill[black] (18,29) circle (4mm);
\fill[black!9] (18,3) circle (4mm);
\fill[black!94] (18,19) circle (4mm);
\fill[black!31] (18,11) circle (4mm);
\fill[black] (18,27) circle (4mm);
\fill[black!17] (18,7) circle (4mm);
\fill[black] (18,23) circle (4mm);
\fill[black!35] (18,15) circle (4mm);
\fill[black] (18,31) circle (4mm);
\fill[black!1] (10,0) circle (4mm);
\fill[black] (10,16) circle (4mm);
\fill[black!41] (10,8) circle (4mm);
\fill[black!93] (10,24) circle (4mm);
\fill[black!16] (10,4) circle (4mm);
\fill[black] (10,20) circle (4mm);
\fill[black!42] (10,12) circle (4mm);
\fill[black] (10,28) circle (4mm);
\fill[black!9] (10,2) circle (4mm);
\fill[black!94] (10,18) circle (4mm);
\fill[black!39] (10,10) circle (4mm);
\fill[black!94] (10,26) circle (4mm);
\fill[black!15] (10,6) circle (4mm);
\fill[black!95] (10,22) circle (4mm);
\fill[black!35] (10,14) circle (4mm);
\fill[black] (10,30) circle (4mm);
\fill[black!4] (10,1) circle (4mm);
\fill[black!99] (10,17) circle (4mm);
\fill[black!40] (10,9) circle (4mm);
\fill[black] (10,25) circle (4mm);
\fill[black!15] (10,5) circle (4mm);
\fill[black!91] (10,21) circle (4mm);
\fill[black!31] (10,13) circle (4mm);
\fill[black!99] (10,29) circle (4mm);
\fill[black!7] (10,3) circle (4mm);
\fill[black!93] (10,19) circle (4mm);
\fill[black!39] (10,11) circle (4mm);
\fill[black!94] (10,27) circle (4mm);
\fill[black!17] (10,7) circle (4mm);
\fill[black!95] (10,23) circle (4mm);
\fill[black!35] (10,15) circle (4mm);
\fill[black!97] (10,31) circle (4mm);
\fill[black!0] (26,0) circle (4mm);
\fill[black!97] (26,16) circle (4mm);
\fill[black!35] (26,8) circle (4mm);
\fill[black] (26,24) circle (4mm);
\fill[black!17] (26,4) circle (4mm);
\fill[black] (26,20) circle (4mm);
\fill[black!37] (26,12) circle (4mm);
\fill[black] (26,28) circle (4mm);
\fill[black!9] (26,2) circle (4mm);
\fill[black!93] (26,18) circle (4mm);
\fill[black!35] (26,10) circle (4mm);
\fill[black!90] (26,26) circle (4mm);
\fill[black!16] (26,6) circle (4mm);
\fill[black!99] (26,22) circle (4mm);
\fill[black!36] (26,14) circle (4mm);
\fill[black!93] (26,30) circle (4mm);
\fill[black!5] (26,1) circle (4mm);
\fill[black!90] (26,17) circle (4mm);
\fill[black!35] (26,9) circle (4mm);
\fill[black!92] (26,25) circle (4mm);
\fill[black!22] (26,5) circle (4mm);
\fill[black!95] (26,21) circle (4mm);
\fill[black!34] (26,13) circle (4mm);
\fill[black!96] (26,29) circle (4mm);
\fill[black!8] (26,3) circle (4mm);
\fill[black!94] (26,19) circle (4mm);
\fill[black!38] (26,11) circle (4mm);
\fill[black] (26,27) circle (4mm);
\fill[black!14] (26,7) circle (4mm);
\fill[black!99] (26,23) circle (4mm);
\fill[black!41] (26,15) circle (4mm);
\fill[black] (26,31) circle (4mm);
\fill[black!1] (6,0) circle (4mm);
\fill[black] (6,16) circle (4mm);
\fill[black!35] (6,8) circle (4mm);
\fill[black] (6,24) circle (4mm);
\fill[black!18] (6,4) circle (4mm);
\fill[black] (6,20) circle (4mm);
\fill[black!35] (6,12) circle (4mm);
\fill[black!96] (6,28) circle (4mm);
\fill[black!8] (6,2) circle (4mm);
\fill[black] (6,18) circle (4mm);
\fill[black!30] (6,10) circle (4mm);
\fill[black] (6,26) circle (4mm);
\fill[black!20] (6,6) circle (4mm);
\fill[black!94] (6,22) circle (4mm);
\fill[black!40] (6,14) circle (4mm);
\fill[black!94] (6,30) circle (4mm);
\fill[black!3] (6,1) circle (4mm);
\fill[black] (6,17) circle (4mm);
\fill[black!35] (6,9) circle (4mm);
\fill[black!87] (6,25) circle (4mm);
\fill[black!20] (6,5) circle (4mm);
\fill[black!95] (6,21) circle (4mm);
\fill[black!35] (6,13) circle (4mm);
\fill[black!91] (6,29) circle (4mm);
\fill[black!9] (6,3) circle (4mm);
\fill[black!99] (6,19) circle (4mm);
\fill[black!44] (6,11) circle (4mm);
\fill[black] (6,27) circle (4mm);
\fill[black!16] (6,7) circle (4mm);
\fill[black!98] (6,23) circle (4mm);
\fill[black!33] (6,15) circle (4mm);
\fill[black!92] (6,31) circle (4mm);
\fill[black!1] (22,0) circle (4mm);
\fill[black!92] (22,16) circle (4mm);
\fill[black!41] (22,8) circle (4mm);
\fill[black!95] (22,24) circle (4mm);
\fill[black!17] (22,4) circle (4mm);
\fill[black!98] (22,20) circle (4mm);
\fill[black!36] (22,12) circle (4mm);
\fill[black!95] (22,28) circle (4mm);
\fill[black!7] (22,2) circle (4mm);
\fill[black] (22,18) circle (4mm);
\fill[black!39] (22,10) circle (4mm);
\fill[black!95] (22,26) circle (4mm);
\fill[black!15] (22,6) circle (4mm);
\fill[black!87] (22,22) circle (4mm);
\fill[black!34] (22,14) circle (4mm);
\fill[black] (22,30) circle (4mm);
\fill[black!3] (22,1) circle (4mm);
\fill[black!92] (22,17) circle (4mm);
\fill[black!36] (22,9) circle (4mm);
\fill[black!96] (22,25) circle (4mm);
\fill[black!16] (22,5) circle (4mm);
\fill[black] (22,21) circle (4mm);
\fill[black!35] (22,13) circle (4mm);
\fill[black!99] (22,29) circle (4mm);
\fill[black!7] (22,3) circle (4mm);
\fill[black] (22,19) circle (4mm);
\fill[black!32] (22,11) circle (4mm);
\fill[black] (22,27) circle (4mm);
\fill[black!17] (22,7) circle (4mm);
\fill[black!87] (22,23) circle (4mm);
\fill[black!38] (22,15) circle (4mm);
\fill[black] (22,31) circle (4mm);
\fill[black!0] (14,0) circle (4mm);
\fill[black] (14,16) circle (4mm);
\fill[black!33] (14,8) circle (4mm);
\fill[black!99] (14,24) circle (4mm);
\fill[black!17] (14,4) circle (4mm);
\fill[black] (14,20) circle (4mm);
\fill[black!36] (14,12) circle (4mm);
\fill[black] (14,28) circle (4mm);
\fill[black!7] (14,2) circle (4mm);
\fill[black] (14,18) circle (4mm);
\fill[black!36] (14,10) circle (4mm);
\fill[black!95] (14,26) circle (4mm);
\fill[black!19] (14,6) circle (4mm);
\fill[black!99] (14,22) circle (4mm);
\fill[black!37] (14,14) circle (4mm);
\fill[black!91] (14,30) circle (4mm);
\fill[black!4] (14,1) circle (4mm);
\fill[black] (14,17) circle (4mm);
\fill[black!35] (14,9) circle (4mm);
\fill[black!92] (14,25) circle (4mm);
\fill[black!14] (14,5) circle (4mm);
\fill[black!99] (14,21) circle (4mm);
\fill[black!36] (14,13) circle (4mm);
\fill[black!93] (14,29) circle (4mm);
\fill[black!8] (14,3) circle (4mm);
\fill[black!91] (14,19) circle (4mm);
\fill[black!39] (14,11) circle (4mm);
\fill[black] (14,27) circle (4mm);
\fill[black!18] (14,7) circle (4mm);
\fill[black!93] (14,23) circle (4mm);
\fill[black!33] (14,15) circle (4mm);
\fill[black!96] (14,31) circle (4mm);
\fill[black!4] (30,0) circle (4mm);
\fill[black!96] (30,16) circle (4mm);
\fill[black!38] (30,8) circle (4mm);
\fill[black!98] (30,24) circle (4mm);
\fill[black!14] (30,4) circle (4mm);
\fill[black!94] (30,20) circle (4mm);
\fill[black!31] (30,12) circle (4mm);
\fill[black!97] (30,28) circle (4mm);
\fill[black!10] (30,2) circle (4mm);
\fill[black!90] (30,18) circle (4mm);
\fill[black!39] (30,10) circle (4mm);
\fill[black!90] (30,26) circle (4mm);
\fill[black!19] (30,6) circle (4mm);
\fill[black!96] (30,22) circle (4mm);
\fill[black!31] (30,14) circle (4mm);
\fill[black!97] (30,30) circle (4mm);
\fill[black!3] (30,1) circle (4mm);
\fill[black!92] (30,17) circle (4mm);
\fill[black!41] (30,9) circle (4mm);
\fill[black!93] (30,25) circle (4mm);
\fill[black!20] (30,5) circle (4mm);
\fill[black] (30,21) circle (4mm);
\fill[black!29] (30,13) circle (4mm);
\fill[black] (30,29) circle (4mm);
\fill[black!6] (30,3) circle (4mm);
\fill[black] (30,19) circle (4mm);
\fill[black!35] (30,11) circle (4mm);
\fill[black!94] (30,27) circle (4mm);
\fill[black!15] (30,7) circle (4mm);
\fill[black!95] (30,23) circle (4mm);
\fill[black!40] (30,15) circle (4mm);
\fill[black!96] (30,31) circle (4mm);
\fill[black!0] (1,0) circle (4mm);
\fill[black!96] (1,16) circle (4mm);
\fill[black!33] (1,8) circle (4mm);
\fill[black!87] (1,24) circle (4mm);
\fill[black!16] (1,4) circle (4mm);
\fill[black] (1,20) circle (4mm);
\fill[black!39] (1,12) circle (4mm);
\fill[black!94] (1,28) circle (4mm);
\fill[black!8] (1,2) circle (4mm);
\fill[black!98] (1,18) circle (4mm);
\fill[black!34] (1,10) circle (4mm);
\fill[black!95] (1,26) circle (4mm);
\fill[black!19] (1,6) circle (4mm);
\fill[black!91] (1,22) circle (4mm);
\fill[black!34] (1,14) circle (4mm);
\fill[black] (1,30) circle (4mm);
\fill[black!4] (1,1) circle (4mm);
\fill[black] (1,17) circle (4mm);
\fill[black!38] (1,9) circle (4mm);
\fill[black] (1,25) circle (4mm);
\fill[black!22] (1,5) circle (4mm);
\fill[black] (1,21) circle (4mm);
\fill[black!42] (1,13) circle (4mm);
\fill[black] (1,29) circle (4mm);
\fill[black!6] (1,3) circle (4mm);
\fill[black!90] (1,19) circle (4mm);
\fill[black!33] (1,11) circle (4mm);
\fill[black!96] (1,27) circle (4mm);
\fill[black!19] (1,7) circle (4mm);
\fill[black!96] (1,23) circle (4mm);
\fill[black!40] (1,15) circle (4mm);
\fill[black!93] (1,31) circle (4mm);
\fill[black!1] (17,0) circle (4mm);
\fill[black!93] (17,16) circle (4mm);
\fill[black!40] (17,8) circle (4mm);
\fill[black!93] (17,24) circle (4mm);
\fill[black!17] (17,4) circle (4mm);
\fill[black] (17,20) circle (4mm);
\fill[black!38] (17,12) circle (4mm);
\fill[black!93] (17,28) circle (4mm);
\fill[black!9] (17,2) circle (4mm);
\fill[black] (17,18) circle (4mm);
\fill[black!37] (17,10) circle (4mm);
\fill[black!93] (17,26) circle (4mm);
\fill[black!17] (17,6) circle (4mm);
\fill[black] (17,22) circle (4mm);
\fill[black!34] (17,14) circle (4mm);
\fill[black!98] (17,30) circle (4mm);
\fill[black!3] (17,1) circle (4mm);
\fill[black!99] (17,17) circle (4mm);
\fill[black!33] (17,9) circle (4mm);
\fill[black!95] (17,25) circle (4mm);
\fill[black!17] (17,5) circle (4mm);
\fill[black] (17,21) circle (4mm);
\fill[black!38] (17,13) circle (4mm);
\fill[black!85] (17,29) circle (4mm);
\fill[black!10] (17,3) circle (4mm);
\fill[black] (17,19) circle (4mm);
\fill[black!38] (17,11) circle (4mm);
\fill[black!94] (17,27) circle (4mm);
\fill[black!20] (17,7) circle (4mm);
\fill[black!90] (17,23) circle (4mm);
\fill[black!40] (17,15) circle (4mm);
\fill[black!94] (17,31) circle (4mm);
\fill[black!2] (9,0) circle (4mm);
\fill[black!94] (9,16) circle (4mm);
\fill[black!40] (9,8) circle (4mm);
\fill[black!95] (9,24) circle (4mm);
\fill[black!18] (9,4) circle (4mm);
\fill[black] (9,20) circle (4mm);
\fill[black!44] (9,12) circle (4mm);
\fill[black!94] (9,28) circle (4mm);
\fill[black!7] (9,2) circle (4mm);
\fill[black] (9,18) circle (4mm);
\fill[black!30] (9,10) circle (4mm);
\fill[black!94] (9,26) circle (4mm);
\fill[black!21] (9,6) circle (4mm);
\fill[black] (9,22) circle (4mm);
\fill[black!38] (9,14) circle (4mm);
\fill[black!93] (9,30) circle (4mm);
\fill[black!4] (9,1) circle (4mm);
\fill[black!92] (9,17) circle (4mm);
\fill[black!35] (9,9) circle (4mm);
\fill[black] (9,25) circle (4mm);
\fill[black!16] (9,5) circle (4mm);
\fill[black] (9,21) circle (4mm);
\fill[black!36] (9,13) circle (4mm);
\fill[black!99] (9,29) circle (4mm);
\fill[black!12] (9,3) circle (4mm);
\fill[black] (9,19) circle (4mm);
\fill[black!38] (9,11) circle (4mm);
\fill[black] (9,27) circle (4mm);
\fill[black!14] (9,7) circle (4mm);
\fill[black!92] (9,23) circle (4mm);
\fill[black!34] (9,15) circle (4mm);
\fill[black] (9,31) circle (4mm);
\fill[black] (25,16) circle (4mm);
\fill[black!40] (25,8) circle (4mm);
\fill[black!96] (25,24) circle (4mm);
\fill[black!15] (25,4) circle (4mm);
\fill[black] (25,20) circle (4mm);
\fill[black!32] (25,12) circle (4mm);
\fill[black!99] (25,28) circle (4mm);
\fill[black!8] (25,2) circle (4mm);
\fill[black!99] (25,18) circle (4mm);
\fill[black!39] (25,10) circle (4mm);
\fill[black!95] (25,26) circle (4mm);
\fill[black!17] (25,6) circle (4mm);
\fill[black!82] (25,22) circle (4mm);
\fill[black!39] (25,14) circle (4mm);
\fill[black] (25,30) circle (4mm);
\fill[black!4] (25,1) circle (4mm);
\fill[black!96] (25,17) circle (4mm);
\fill[black!35] (25,9) circle (4mm);
\fill[black!99] (25,25) circle (4mm);
\fill[black!21] (25,5) circle (4mm);
\fill[black] (25,21) circle (4mm);
\fill[black!38] (25,13) circle (4mm);
\fill[black!92] (25,29) circle (4mm);
\fill[black!8] (25,3) circle (4mm);
\fill[black!96] (25,19) circle (4mm);
\fill[black!33] (25,11) circle (4mm);
\fill[black!94] (25,27) circle (4mm);
\fill[black!15] (25,7) circle (4mm);
\fill[black!87] (25,23) circle (4mm);
\fill[black!35] (25,15) circle (4mm);
\fill[black] (25,31) circle (4mm);
\fill[black!0] (5,0) circle (4mm);
\fill[black] (5,16) circle (4mm);
\fill[black!34] (5,8) circle (4mm);
\fill[black!90] (5,24) circle (4mm);
\fill[black!19] (5,4) circle (4mm);
\fill[black!94] (5,20) circle (4mm);
\fill[black!39] (5,12) circle (4mm);
\fill[black] (5,28) circle (4mm);
\fill[black!9] (5,2) circle (4mm);
\fill[black!96] (5,18) circle (4mm);
\fill[black!31] (5,10) circle (4mm);
\fill[black!91] (5,26) circle (4mm);
\fill[black!16] (5,6) circle (4mm);
\fill[black!90] (5,22) circle (4mm);
\fill[black!33] (5,14) circle (4mm);
\fill[black!96] (5,30) circle (4mm);
\fill[black!5] (5,1) circle (4mm);
\fill[black!98] (5,17) circle (4mm);
\fill[black!38] (5,9) circle (4mm);
\fill[black!97] (5,25) circle (4mm);
\fill[black!17] (5,5) circle (4mm);
\fill[black] (5,21) circle (4mm);
\fill[black!37] (5,13) circle (4mm);
\fill[black!95] (5,29) circle (4mm);
\fill[black!9] (5,3) circle (4mm);
\fill[black!97] (5,19) circle (4mm);
\fill[black!44] (5,11) circle (4mm);
\fill[black!98] (5,27) circle (4mm);
\fill[black!17] (5,7) circle (4mm);
\fill[black!95] (5,23) circle (4mm);
\fill[black!41] (5,15) circle (4mm);
\fill[black] (5,31) circle (4mm);
\fill[black!1] (21,0) circle (4mm);
\fill[black] (21,16) circle (4mm);
\fill[black!35] (21,8) circle (4mm);
\fill[black!92] (21,24) circle (4mm);
\fill[black!20] (21,4) circle (4mm);
\fill[black!96] (21,20) circle (4mm);
\fill[black!35] (21,12) circle (4mm);
\fill[black!91] (21,28) circle (4mm);
\fill[black!7] (21,2) circle (4mm);
\fill[black!91] (21,18) circle (4mm);
\fill[black!34] (21,10) circle (4mm);
\fill[black!91] (21,26) circle (4mm);
\fill[black!20] (21,6) circle (4mm);
\fill[black!99] (21,22) circle (4mm);
\fill[black!35] (21,14) circle (4mm);
\fill[black] (21,30) circle (4mm);
\fill[black!5] (21,1) circle (4mm);
\fill[black!97] (21,17) circle (4mm);
\fill[black!40] (21,9) circle (4mm);
\fill[black] (21,25) circle (4mm);
\fill[black!20] (21,5) circle (4mm);
\fill[black!94] (21,21) circle (4mm);
\fill[black!32] (21,13) circle (4mm);
\fill[black!95] (21,29) circle (4mm);
\fill[black!8] (21,3) circle (4mm);
\fill[black!89] (21,19) circle (4mm);
\fill[black!34] (21,11) circle (4mm);
\fill[black] (21,27) circle (4mm);
\fill[black!18] (21,7) circle (4mm);
\fill[black!95] (21,23) circle (4mm);
\fill[black!32] (21,15) circle (4mm);
\fill[black] (21,31) circle (4mm);
\fill[black!1] (13,0) circle (4mm);
\fill[black] (13,16) circle (4mm);
\fill[black!41] (13,8) circle (4mm);
\fill[black!87] (13,24) circle (4mm);
\fill[black!14] (13,4) circle (4mm);
\fill[black!94] (13,20) circle (4mm);
\fill[black!33] (13,12) circle (4mm);
\fill[black] (13,28) circle (4mm);
\fill[black!8] (13,2) circle (4mm);
\fill[black!99] (13,18) circle (4mm);
\fill[black!35] (13,10) circle (4mm);
\fill[black!95] (13,26) circle (4mm);
\fill[black!15] (13,6) circle (4mm);
\fill[black!86] (13,22) circle (4mm);
\fill[black!35] (13,14) circle (4mm);
\fill[black] (13,30) circle (4mm);
\fill[black!4] (13,1) circle (4mm);
\fill[black!91] (13,17) circle (4mm);
\fill[black!41] (13,9) circle (4mm);
\fill[black] (13,25) circle (4mm);
\fill[black!18] (13,5) circle (4mm);
\fill[black!98] (13,21) circle (4mm);
\fill[black!38] (13,13) circle (4mm);
\fill[black!98] (13,29) circle (4mm);
\fill[black!9] (13,3) circle (4mm);
\fill[black] (13,19) circle (4mm);
\fill[black!39] (13,11) circle (4mm);
\fill[black!94] (13,27) circle (4mm);
\fill[black!15] (13,7) circle (4mm);
\fill[black!98] (13,23) circle (4mm);
\fill[black!34] (13,15) circle (4mm);
\fill[black!88] (13,31) circle (4mm);
\fill[black!2] (29,0) circle (4mm);
\fill[black!88] (29,16) circle (4mm);
\fill[black!32] (29,8) circle (4mm);
\fill[black!95] (29,24) circle (4mm);
\fill[black!15] (29,4) circle (4mm);
\fill[black] (29,20) circle (4mm);
\fill[black!38] (29,12) circle (4mm);
\fill[black!90] (29,28) circle (4mm);
\fill[black!6] (29,2) circle (4mm);
\fill[black!93] (29,18) circle (4mm);
\fill[black!35] (29,10) circle (4mm);
\fill[black!95] (29,26) circle (4mm);
\fill[black!22] (29,6) circle (4mm);
\fill[black] (29,22) circle (4mm);
\fill[black!32] (29,14) circle (4mm);
\fill[black!94] (29,30) circle (4mm);
\fill[black!3] (29,1) circle (4mm);
\fill[black!95] (29,17) circle (4mm);
\fill[black!42] (29,9) circle (4mm);
\fill[black!97] (29,25) circle (4mm);
\fill[black!16] (29,5) circle (4mm);
\fill[black] (29,21) circle (4mm);
\fill[black!38] (29,13) circle (4mm);
\fill[black!91] (29,29) circle (4mm);
\fill[black!10] (29,3) circle (4mm);
\fill[black!96] (29,19) circle (4mm);
\fill[black!38] (29,11) circle (4mm);
\fill[black!89] (29,27) circle (4mm);
\fill[black!15] (29,7) circle (4mm);
\fill[black] (29,23) circle (4mm);
\fill[black!42] (29,15) circle (4mm);
\fill[black] (29,31) circle (4mm);
\fill[black!3] (3,0) circle (4mm);
\fill[black] (3,16) circle (4mm);
\fill[black!34] (3,8) circle (4mm);
\fill[black!95] (3,24) circle (4mm);
\fill[black!17] (3,4) circle (4mm);
\fill[black!94] (3,20) circle (4mm);
\fill[black!38] (3,12) circle (4mm);
\fill[black] (3,28) circle (4mm);
\fill[black!6] (3,2) circle (4mm);
\fill[black] (3,18) circle (4mm);
\fill[black!36] (3,10) circle (4mm);
\fill[black] (3,26) circle (4mm);
\fill[black!20] (3,6) circle (4mm);
\fill[black!92] (3,22) circle (4mm);
\fill[black!40] (3,14) circle (4mm);
\fill[black] (3,30) circle (4mm);
\fill[black!4] (3,1) circle (4mm);
\fill[black!92] (3,17) circle (4mm);
\fill[black!41] (3,9) circle (4mm);
\fill[black] (3,25) circle (4mm);
\fill[black!16] (3,5) circle (4mm);
\fill[black!92] (3,21) circle (4mm);
\fill[black!36] (3,13) circle (4mm);
\fill[black!95] (3,29) circle (4mm);
\fill[black!7] (3,3) circle (4mm);
\fill[black!92] (3,19) circle (4mm);
\fill[black!37] (3,11) circle (4mm);
\fill[black] (3,27) circle (4mm);
\fill[black!16] (3,7) circle (4mm);
\fill[black!97] (3,23) circle (4mm);
\fill[black!36] (3,15) circle (4mm);
\fill[black] (3,31) circle (4mm);
\fill[black!2] (19,0) circle (4mm);
\fill[black] (19,16) circle (4mm);
\fill[black!42] (19,8) circle (4mm);
\fill[black!98] (19,24) circle (4mm);
\fill[black!18] (19,4) circle (4mm);
\fill[black!98] (19,20) circle (4mm);
\fill[black!33] (19,12) circle (4mm);
\fill[black] (19,28) circle (4mm);
\fill[black!10] (19,2) circle (4mm);
\fill[black] (19,18) circle (4mm);
\fill[black!29] (19,10) circle (4mm);
\fill[black!99] (19,26) circle (4mm);
\fill[black!16] (19,6) circle (4mm);
\fill[black!87] (19,22) circle (4mm);
\fill[black!35] (19,14) circle (4mm);
\fill[black!99] (19,30) circle (4mm);
\fill[black!3] (19,1) circle (4mm);
\fill[black!95] (19,17) circle (4mm);
\fill[black!35] (19,9) circle (4mm);
\fill[black!96] (19,25) circle (4mm);
\fill[black!16] (19,5) circle (4mm);
\fill[black!97] (19,21) circle (4mm);
\fill[black!35] (19,13) circle (4mm);
\fill[black!94] (19,29) circle (4mm);
\fill[black!11] (19,3) circle (4mm);
\fill[black!95] (19,19) circle (4mm);
\fill[black!36] (19,11) circle (4mm);
\fill[black] (19,27) circle (4mm);
\fill[black!17] (19,7) circle (4mm);
\fill[black] (19,23) circle (4mm);
\fill[black!36] (19,15) circle (4mm);
\fill[black!92] (19,31) circle (4mm);
\fill[black!1] (11,0) circle (4mm);
\fill[black!92] (11,16) circle (4mm);
\fill[black!36] (11,8) circle (4mm);
\fill[black] (11,24) circle (4mm);
\fill[black!15] (11,4) circle (4mm);
\fill[black] (11,20) circle (4mm);
\fill[black!44] (11,12) circle (4mm);
\fill[black!96] (11,28) circle (4mm);
\fill[black!12] (11,2) circle (4mm);
\fill[black!85] (11,18) circle (4mm);
\fill[black!42] (11,10) circle (4mm);
\fill[black] (11,26) circle (4mm);
\fill[black!14] (11,6) circle (4mm);
\fill[black!96] (11,22) circle (4mm);
\fill[black!35] (11,14) circle (4mm);
\fill[black!90] (11,30) circle (4mm);
\fill[black!4] (11,1) circle (4mm);
\fill[black!95] (11,17) circle (4mm);
\fill[black!38] (11,9) circle (4mm);
\fill[black!94] (11,25) circle (4mm);
\fill[black!19] (11,5) circle (4mm);
\fill[black!96] (11,21) circle (4mm);
\fill[black!36] (11,13) circle (4mm);
\fill[black!91] (11,29) circle (4mm);
\fill[black!9] (11,3) circle (4mm);
\fill[black!98] (11,19) circle (4mm);
\fill[black!38] (11,11) circle (4mm);
\fill[black] (11,27) circle (4mm);
\fill[black!17] (11,7) circle (4mm);
\fill[black!93] (11,23) circle (4mm);
\fill[black!34] (11,15) circle (4mm);
\fill[black!97] (11,31) circle (4mm);
\fill[black!1] (27,0) circle (4mm);
\fill[black!97] (27,16) circle (4mm);
\fill[black!36] (27,8) circle (4mm);
\fill[black!97] (27,24) circle (4mm);
\fill[black!15] (27,4) circle (4mm);
\fill[black!94] (27,20) circle (4mm);
\fill[black!34] (27,12) circle (4mm);
\fill[black!97] (27,28) circle (4mm);
\fill[black!8] (27,2) circle (4mm);
\fill[black!99] (27,18) circle (4mm);
\fill[black!38] (27,10) circle (4mm);
\fill[black] (27,26) circle (4mm);
\fill[black!20] (27,6) circle (4mm);
\fill[black!92] (27,22) circle (4mm);
\fill[black!36] (27,14) circle (4mm);
\fill[black] (27,30) circle (4mm);
\fill[black!5] (27,1) circle (4mm);
\fill[black] (27,17) circle (4mm);
\fill[black!37] (27,9) circle (4mm);
\fill[black] (27,25) circle (4mm);
\fill[black!20] (27,5) circle (4mm);
\fill[black] (27,21) circle (4mm);
\fill[black!39] (27,13) circle (4mm);
\fill[black!99] (27,29) circle (4mm);
\fill[black!9] (27,3) circle (4mm);
\fill[black] (27,19) circle (4mm);
\fill[black!36] (27,11) circle (4mm);
\fill[black!92] (27,27) circle (4mm);
\fill[black!14] (27,7) circle (4mm);
\fill[black!99] (27,23) circle (4mm);
\fill[black!37] (27,15) circle (4mm);
\fill[black!98] (27,31) circle (4mm);
\fill[black!1] (7,0) circle (4mm);
\fill[black!98] (7,16) circle (4mm);
\fill[black!34] (7,8) circle (4mm);
\fill[black] (7,24) circle (4mm);
\fill[black!16] (7,4) circle (4mm);
\fill[black!89] (7,20) circle (4mm);
\fill[black!39] (7,12) circle (4mm);
\fill[black!89] (7,28) circle (4mm);
\fill[black!9] (7,2) circle (4mm);
\fill[black!92] (7,18) circle (4mm);
\fill[black!36] (7,10) circle (4mm);
\fill[black] (7,26) circle (4mm);
\fill[black!22] (7,6) circle (4mm);
\fill[black!93] (7,22) circle (4mm);
\fill[black!35] (7,14) circle (4mm);
\fill[black!92] (7,30) circle (4mm);
\fill[black!3] (7,1) circle (4mm);
\fill[black!93] (7,17) circle (4mm);
\fill[black!35] (7,9) circle (4mm);
\fill[black!99] (7,25) circle (4mm);
\fill[black!17] (7,5) circle (4mm);
\fill[black!96] (7,21) circle (4mm);
\fill[black!37] (7,13) circle (4mm);
\fill[black!96] (7,29) circle (4mm);
\fill[black!7] (7,3) circle (4mm);
\fill[black!94] (7,19) circle (4mm);
\fill[black!35] (7,11) circle (4mm);
\fill[black!89] (7,27) circle (4mm);
\fill[black!16] (7,7) circle (4mm);
\fill[black!96] (7,23) circle (4mm);
\fill[black!38] (7,15) circle (4mm);
\fill[black] (7,31) circle (4mm);
\fill[black!1] (23,0) circle (4mm);
\fill[black] (23,16) circle (4mm);
\fill[black!37] (23,8) circle (4mm);
\fill[black!93] (23,24) circle (4mm);
\fill[black!17] (23,4) circle (4mm);
\fill[black] (23,20) circle (4mm);
\fill[black!38] (23,12) circle (4mm);
\fill[black] (23,28) circle (4mm);
\fill[black!8] (23,2) circle (4mm);
\fill[black!95] (23,18) circle (4mm);
\fill[black!38] (23,10) circle (4mm);
\fill[black] (23,26) circle (4mm);
\fill[black!17] (23,6) circle (4mm);
\fill[black] (23,22) circle (4mm);
\fill[black!41] (23,14) circle (4mm);
\fill[black] (23,30) circle (4mm);
\fill[black!3] (23,1) circle (4mm);
\fill[black!94] (23,17) circle (4mm);
\fill[black!36] (23,9) circle (4mm);
\fill[black!95] (23,25) circle (4mm);
\fill[black!17] (23,5) circle (4mm);
\fill[black!99] (23,21) circle (4mm);
\fill[black!32] (23,13) circle (4mm);
\fill[black!99] (23,29) circle (4mm);
\fill[black!9] (23,3) circle (4mm);
\fill[black!94] (23,19) circle (4mm);
\fill[black!41] (23,11) circle (4mm);
\fill[black!91] (23,27) circle (4mm);
\fill[black!20] (23,7) circle (4mm);
\fill[black] (23,23) circle (4mm);
\fill[black!37] (23,15) circle (4mm);
\fill[black] (23,31) circle (4mm);
\fill[black!1] (15,0) circle (4mm);
\fill[black] (15,16) circle (4mm);
\fill[black!38] (15,8) circle (4mm);
\fill[black!99] (15,24) circle (4mm);
\fill[black!17] (15,4) circle (4mm);
\fill[black] (15,20) circle (4mm);
\fill[black!37] (15,12) circle (4mm);
\fill[black!96] (15,28) circle (4mm);
\fill[black!9] (15,2) circle (4mm);
\fill[black!95] (15,18) circle (4mm);
\fill[black!37] (15,10) circle (4mm);
\fill[black] (15,26) circle (4mm);
\fill[black!16] (15,6) circle (4mm);
\fill[black!95] (15,22) circle (4mm);
\fill[black!38] (15,14) circle (4mm);
\fill[black!92] (15,30) circle (4mm);
\fill[black!4] (15,1) circle (4mm);
\fill[black] (15,17) circle (4mm);
\fill[black!40] (15,9) circle (4mm);
\fill[black!99] (15,25) circle (4mm);
\fill[black!15] (15,5) circle (4mm);
\fill[black] (15,21) circle (4mm);
\fill[black!39] (15,13) circle (4mm);
\fill[black] (15,29) circle (4mm);
\fill[black!7] (15,3) circle (4mm);
\fill[black] (15,19) circle (4mm);
\fill[black!35] (15,11) circle (4mm);
\fill[black!99] (15,27) circle (4mm);
\fill[black!18] (15,7) circle (4mm);
\fill[black!91] (15,23) circle (4mm);
\fill[black!34] (15,15) circle (4mm);
\fill[black!89] (15,31) circle (4mm);
\fill[black!3] (31,0) circle (4mm);
\fill[black!89] (31,16) circle (4mm);
\fill[black!37] (31,8) circle (4mm);
\fill[black!96] (31,24) circle (4mm);
\fill[black!14] (31,4) circle (4mm);
\fill[black] (31,20) circle (4mm);
\fill[black!36] (31,12) circle (4mm);
\fill[black!92] (31,28) circle (4mm);
\fill[black!10] (31,2) circle (4mm);
\fill[black!98] (31,18) circle (4mm);
\fill[black!32] (31,10) circle (4mm);
\fill[black] (31,26) circle (4mm);
\fill[black!21] (31,6) circle (4mm);
\fill[black] (31,22) circle (4mm);
\fill[black!33] (31,14) circle (4mm);
\fill[black] (31,30) circle (4mm);
\fill[black!3] (31,1) circle (4mm);
\fill[black!91] (31,17) circle (4mm);
\fill[black!34] (31,9) circle (4mm);
\fill[black!94] (31,25) circle (4mm);
\fill[black!16] (31,5) circle (4mm);
\fill[black!94] (31,21) circle (4mm);
\fill[black!46] (31,13) circle (4mm);
\fill[black!91] (31,29) circle (4mm);
\fill[black!7] (31,3) circle (4mm);
\fill[black!99] (31,19) circle (4mm);
\fill[black!39] (31,11) circle (4mm);
\fill[black!93] (31,27) circle (4mm);
\fill[black!24] (31,7) circle (4mm);
\fill[black!98] (31,23) circle (4mm);
\fill[black!40] (31,15) circle (4mm);
\fill[black] (31,31) circle (4mm);
\end{tikzpicture}

%% file: statsram.tex
\begin{tikzpicture}[scale=0.18]
\draw[->] (-0.5,-2)--(31.5,-2);
\draw[->] (-2,-0.5)--(-2,31.5);
\node[scale=0.8] at (31,-3.5) { $\ZZ_2$ };
\node[scale=0.8] at (-4.5,31) { $\ZZ_2{\cdot} \pi$ };
\fill[black!11] (0,16) circle (4mm);
\fill[black!4] (0,8) circle (4mm);
\fill[black!9] (0,24) circle (4mm);
\fill[black!1] (0,4) circle (4mm);
\fill[black!10] (0,20) circle (4mm);
\fill[black!3] (0,12) circle (4mm);
\fill[black!10] (0,28) circle (4mm);
\fill[black!0] (0,2) circle (4mm);
\fill[black!10] (0,18) circle (4mm);
\fill[black!5] (0,10) circle (4mm);
\fill[black!9] (0,26) circle (4mm);
\fill[black!1] (0,6) circle (4mm);
\fill[black!7] (0,22) circle (4mm);
\fill[black!3] (0,14) circle (4mm);
\fill[black!11] (0,30) circle (4mm);
\fill[black!0] (0,1) circle (4mm);
\fill[black!11] (0,17) circle (4mm);
\fill[black!3] (0,9) circle (4mm);
\fill[black!7] (0,25) circle (4mm);
\fill[black!1] (0,5) circle (4mm);
\fill[black!9] (0,21) circle (4mm);
\fill[black!5] (0,13) circle (4mm);
\fill[black!10] (0,29) circle (4mm);
\fill[black!0] (0,3) circle (4mm);
\fill[black!10] (0,19) circle (4mm);
\fill[black!3] (0,11) circle (4mm);
\fill[black!10] (0,27) circle (4mm);
\fill[black!1] (0,7) circle (4mm);
\fill[black!9] (0,23) circle (4mm);
\fill[black!4] (0,15) circle (4mm);
\fill[black!11] (0,31) circle (4mm);
\fill[black!13] (16,16) circle (4mm);
\fill[black!7] (16,8) circle (4mm);
\fill[black!11] (16,24) circle (4mm);
\fill[black!2] (16,4) circle (4mm);
\fill[black!10] (16,20) circle (4mm);
\fill[black!4] (16,12) circle (4mm);
\fill[black!8] (16,28) circle (4mm);
\fill[black!1] (16,2) circle (4mm);
\fill[black!11] (16,18) circle (4mm);
\fill[black!2] (16,10) circle (4mm);
\fill[black!11] (16,26) circle (4mm);
\fill[black!2] (16,6) circle (4mm);
\fill[black!9] (16,22) circle (4mm);
\fill[black!4] (16,14) circle (4mm);
\fill[black!12] (16,30) circle (4mm);
\fill[black!0] (16,1) circle (4mm);
\fill[black!12] (16,17) circle (4mm);
\fill[black!4] (16,9) circle (4mm);
\fill[black!9] (16,25) circle (4mm);
\fill[black!2] (16,5) circle (4mm);
\fill[black!11] (16,21) circle (4mm);
\fill[black!2] (16,13) circle (4mm);
\fill[black!11] (16,29) circle (4mm);
\fill[black!1] (16,3) circle (4mm);
\fill[black!8] (16,19) circle (4mm);
\fill[black!4] (16,11) circle (4mm);
\fill[black!10] (16,27) circle (4mm);
\fill[black!2] (16,7) circle (4mm);
\fill[black!11] (16,23) circle (4mm);
\fill[black!7] (16,15) circle (4mm);
\fill[black!13] (16,31) circle (4mm);
\fill[black!7] (8,16) circle (4mm);
\fill[black!4] (8,8) circle (4mm);
\fill[black!11] (8,24) circle (4mm);
\fill[black!2] (8,4) circle (4mm);
\fill[black!11] (8,20) circle (4mm);
\fill[black!5] (8,12) circle (4mm);
\fill[black!8] (8,28) circle (4mm);
\fill[black!0] (8,2) circle (4mm);
\fill[black!12] (8,18) circle (4mm);
\fill[black!5] (8,10) circle (4mm);
\fill[black!12] (8,26) circle (4mm);
\fill[black!1] (8,6) circle (4mm);
\fill[black!10] (8,22) circle (4mm);
\fill[black!4] (8,14) circle (4mm);
\fill[black!10] (8,30) circle (4mm);
\fill[black!0] (8,1) circle (4mm);
\fill[black!10] (8,17) circle (4mm);
\fill[black!4] (8,9) circle (4mm);
\fill[black!10] (8,25) circle (4mm);
\fill[black!1] (8,5) circle (4mm);
\fill[black!12] (8,21) circle (4mm);
\fill[black!5] (8,13) circle (4mm);
\fill[black!12] (8,29) circle (4mm);
\fill[black!0] (8,3) circle (4mm);
\fill[black!8] (8,19) circle (4mm);
\fill[black!5] (8,11) circle (4mm);
\fill[black!11] (8,27) circle (4mm);
\fill[black!2] (8,7) circle (4mm);
\fill[black!11] (8,23) circle (4mm);
\fill[black!4] (8,15) circle (4mm);
\fill[black!7] (8,31) circle (4mm);
\fill[black!0] (24,0) circle (4mm);
\fill[black!13] (24,16) circle (4mm);
\fill[black!4] (24,8) circle (4mm);
\fill[black!9] (24,24) circle (4mm);
\fill[black!1] (24,4) circle (4mm);
\fill[black!7] (24,20) circle (4mm);
\fill[black!5] (24,12) circle (4mm);
\fill[black!9] (24,28) circle (4mm);
\fill[black!0] (24,2) circle (4mm);
\fill[black!9] (24,18) circle (4mm);
\fill[black!4] (24,10) circle (4mm);
\fill[black!12] (24,26) circle (4mm);
\fill[black!2] (24,6) circle (4mm);
\fill[black!12] (24,22) circle (4mm);
\fill[black!4] (24,14) circle (4mm);
\fill[black!8] (24,30) circle (4mm);
\fill[black!1] (24,1) circle (4mm);
\fill[black!8] (24,17) circle (4mm);
\fill[black!4] (24,9) circle (4mm);
\fill[black!12] (24,25) circle (4mm);
\fill[black!2] (24,5) circle (4mm);
\fill[black!12] (24,21) circle (4mm);
\fill[black!4] (24,13) circle (4mm);
\fill[black!9] (24,29) circle (4mm);
\fill[black!0] (24,3) circle (4mm);
\fill[black!9] (24,19) circle (4mm);
\fill[black!5] (24,11) circle (4mm);
\fill[black!7] (24,27) circle (4mm);
\fill[black!1] (24,7) circle (4mm);
\fill[black!9] (24,23) circle (4mm);
\fill[black!4] (24,15) circle (4mm);
\fill[black!13] (24,31) circle (4mm);
\fill[black!1] (4,0) circle (4mm);
\fill[black!9] (4,16) circle (4mm);
\fill[black!4] (4,8) circle (4mm);
\fill[black!11] (4,24) circle (4mm);
\fill[black!3] (4,4) circle (4mm);
\fill[black!11] (4,20) circle (4mm);
\fill[black!5] (4,12) circle (4mm);
\fill[black!11] (4,28) circle (4mm);
\fill[black!0] (4,2) circle (4mm);
\fill[black!10] (4,18) circle (4mm);
\fill[black!5] (4,10) circle (4mm);
\fill[black!10] (4,26) circle (4mm);
\fill[black!2] (4,6) circle (4mm);
\fill[black!7] (4,22) circle (4mm);
\fill[black!4] (4,14) circle (4mm);
\fill[black!7] (4,30) circle (4mm);
\fill[black!0] (4,1) circle (4mm);
\fill[black!7] (4,17) circle (4mm);
\fill[black!4] (4,9) circle (4mm);
\fill[black!7] (4,25) circle (4mm);
\fill[black!2] (4,5) circle (4mm);
\fill[black!10] (4,21) circle (4mm);
\fill[black!5] (4,13) circle (4mm);
\fill[black!10] (4,29) circle (4mm);
\fill[black!0] (4,3) circle (4mm);
\fill[black!11] (4,19) circle (4mm);
\fill[black!5] (4,11) circle (4mm);
\fill[black!11] (4,27) circle (4mm);
\fill[black!3] (4,7) circle (4mm);
\fill[black!11] (4,23) circle (4mm);
\fill[black!4] (4,15) circle (4mm);
\fill[black!9] (4,31) circle (4mm);
\fill[black!0] (20,0) circle (4mm);
\fill[black!12] (20,16) circle (4mm);
\fill[black!5] (20,8) circle (4mm);
\fill[black!9] (20,24) circle (4mm);
\fill[black!0] (20,4) circle (4mm);
\fill[black!10] (20,20) circle (4mm);
\fill[black!3] (20,12) circle (4mm);
\fill[black!9] (20,28) circle (4mm);
\fill[black!1] (20,2) circle (4mm);
\fill[black!10] (20,18) circle (4mm);
\fill[black!5] (20,10) circle (4mm);
\fill[black!9] (20,26) circle (4mm);
\fill[black!3] (20,6) circle (4mm);
\fill[black!9] (20,22) circle (4mm);
\fill[black!4] (20,14) circle (4mm);
\fill[black!12] (20,30) circle (4mm);
\fill[black!0] (20,1) circle (4mm);
\fill[black!12] (20,17) circle (4mm);
\fill[black!4] (20,9) circle (4mm);
\fill[black!9] (20,25) circle (4mm);
\fill[black!3] (20,5) circle (4mm);
\fill[black!9] (20,21) circle (4mm);
\fill[black!5] (20,13) circle (4mm);
\fill[black!10] (20,29) circle (4mm);
\fill[black!1] (20,3) circle (4mm);
\fill[black!9] (20,19) circle (4mm);
\fill[black!3] (20,11) circle (4mm);
\fill[black!10] (20,27) circle (4mm);
\fill[black!0] (20,7) circle (4mm);
\fill[black!9] (20,23) circle (4mm);
\fill[black!5] (20,15) circle (4mm);
\fill[black!12] (20,31) circle (4mm);
\fill[black!0] (12,0) circle (4mm);
\fill[black!10] (12,16) circle (4mm);
\fill[black!5] (12,8) circle (4mm);
\fill[black!8] (12,24) circle (4mm);
\fill[black!1] (12,4) circle (4mm);
\fill[black!11] (12,20) circle (4mm);
\fill[black!6] (12,12) circle (4mm);
\fill[black!10] (12,28) circle (4mm);
\fill[black!1] (12,2) circle (4mm);
\fill[black!12] (12,18) circle (4mm);
\fill[black!4] (12,10) circle (4mm);
\fill[black!13] (12,26) circle (4mm);
\fill[black!2] (12,6) circle (4mm);
\fill[black!11] (12,22) circle (4mm);
\fill[black!2] (12,14) circle (4mm);
\fill[black!13] (12,30) circle (4mm);
\fill[black!1] (12,1) circle (4mm);
\fill[black!13] (12,17) circle (4mm);
\fill[black!2] (12,9) circle (4mm);
\fill[black!11] (12,25) circle (4mm);
\fill[black!2] (12,5) circle (4mm);
\fill[black!13] (12,21) circle (4mm);
\fill[black!4] (12,13) circle (4mm);
\fill[black!12] (12,29) circle (4mm);
\fill[black!1] (12,3) circle (4mm);
\fill[black!10] (12,19) circle (4mm);
\fill[black!6] (12,11) circle (4mm);
\fill[black!11] (12,27) circle (4mm);
\fill[black!1] (12,7) circle (4mm);
\fill[black!8] (12,23) circle (4mm);
\fill[black!5] (12,15) circle (4mm);
\fill[black!10] (12,31) circle (4mm);
\fill[black!0] (28,0) circle (4mm);
\fill[black!13] (28,16) circle (4mm);
\fill[black!4] (28,8) circle (4mm);
\fill[black!11] (28,24) circle (4mm);
\fill[black!2] (28,4) circle (4mm);
\fill[black!9] (28,20) circle (4mm);
\fill[black!4] (28,12) circle (4mm);
\fill[black!10] (28,28) circle (4mm);
\fill[black!1] (28,2) circle (4mm);
\fill[black!9] (28,18) circle (4mm);
\fill[black!5] (28,10) circle (4mm);
\fill[black!11] (28,26) circle (4mm);
\fill[black!3] (28,6) circle (4mm);
\fill[black!15] (28,22) circle (4mm);
\fill[black!4] (28,14) circle (4mm);
\fill[black!10] (28,30) circle (4mm);
\fill[black!0] (28,1) circle (4mm);
\fill[black!10] (28,17) circle (4mm);
\fill[black!4] (28,9) circle (4mm);
\fill[black!15] (28,25) circle (4mm);
\fill[black!3] (28,5) circle (4mm);
\fill[black!11] (28,21) circle (4mm);
\fill[black!5] (28,13) circle (4mm);
\fill[black!9] (28,29) circle (4mm);
\fill[black!1] (28,3) circle (4mm);
\fill[black!10] (28,19) circle (4mm);
\fill[black!4] (28,11) circle (4mm);
\fill[black!9] (28,27) circle (4mm);
\fill[black!2] (28,7) circle (4mm);
\fill[black!11] (28,23) circle (4mm);
\fill[black!4] (28,15) circle (4mm);
\fill[black!13] (28,31) circle (4mm);
\fill[black!10] (2,16) circle (4mm);
\fill[black!4] (2,8) circle (4mm);
\fill[black!11] (2,24) circle (4mm);
\fill[black!3] (2,4) circle (4mm);
\fill[black!8] (2,20) circle (4mm);
\fill[black!4] (2,12) circle (4mm);
\fill[black!9] (2,28) circle (4mm);
\fill[black!0] (2,2) circle (4mm);
\fill[black!11] (2,18) circle (4mm);
\fill[black!2] (2,10) circle (4mm);
\fill[black!8] (2,26) circle (4mm);
\fill[black!2] (2,6) circle (4mm);
\fill[black!9] (2,22) circle (4mm);
\fill[black!5] (2,14) circle (4mm);
\fill[black!11] (2,30) circle (4mm);
\fill[black!1] (2,1) circle (4mm);
\fill[black!11] (2,17) circle (4mm);
\fill[black!5] (2,9) circle (4mm);
\fill[black!9] (2,25) circle (4mm);
\fill[black!2] (2,5) circle (4mm);
\fill[black!8] (2,21) circle (4mm);
\fill[black!2] (2,13) circle (4mm);
\fill[black!11] (2,29) circle (4mm);
\fill[black!0] (2,3) circle (4mm);
\fill[black!9] (2,19) circle (4mm);
\fill[black!4] (2,11) circle (4mm);
\fill[black!8] (2,27) circle (4mm);
\fill[black!3] (2,7) circle (4mm);
\fill[black!11] (2,23) circle (4mm);
\fill[black!4] (2,15) circle (4mm);
\fill[black!10] (2,31) circle (4mm);
\fill[black!9] (18,16) circle (4mm);
\fill[black!4] (18,8) circle (4mm);
\fill[black!10] (18,24) circle (4mm);
\fill[black!1] (18,4) circle (4mm);
\fill[black!9] (18,20) circle (4mm);
\fill[black!5] (18,12) circle (4mm);
\fill[black!8] (18,28) circle (4mm);
\fill[black!1] (18,2) circle (4mm);
\fill[black!8] (18,18) circle (4mm);
\fill[black!4] (18,10) circle (4mm);
\fill[black!12] (18,26) circle (4mm);
\fill[black!2] (18,6) circle (4mm);
\fill[black!8] (18,22) circle (4mm);
\fill[black!3] (18,14) circle (4mm);
\fill[black!12] (18,30) circle (4mm);
\fill[black!12] (18,17) circle (4mm);
\fill[black!3] (18,9) circle (4mm);
\fill[black!8] (18,25) circle (4mm);
\fill[black!2] (18,5) circle (4mm);
\fill[black!12] (18,21) circle (4mm);
\fill[black!4] (18,13) circle (4mm);
\fill[black!8] (18,29) circle (4mm);
\fill[black!1] (18,3) circle (4mm);
\fill[black!8] (18,19) circle (4mm);
\fill[black!5] (18,11) circle (4mm);
\fill[black!9] (18,27) circle (4mm);
\fill[black!1] (18,7) circle (4mm);
\fill[black!10] (18,23) circle (4mm);
\fill[black!4] (18,15) circle (4mm);
\fill[black!9] (18,31) circle (4mm);
\fill[black!10] (10,16) circle (4mm);
\fill[black!3] (10,8) circle (4mm);
\fill[black!14] (10,24) circle (4mm);
\fill[black!1] (10,4) circle (4mm);
\fill[black!11] (10,20) circle (4mm);
\fill[black!3] (10,12) circle (4mm);
\fill[black!11] (10,28) circle (4mm);
\fill[black!0] (10,2) circle (4mm);
\fill[black!12] (10,18) circle (4mm);
\fill[black!3] (10,10) circle (4mm);
\fill[black!7] (10,26) circle (4mm);
\fill[black!1] (10,6) circle (4mm);
\fill[black!9] (10,22) circle (4mm);
\fill[black!5] (10,14) circle (4mm);
\fill[black!9] (10,30) circle (4mm);
\fill[black!9] (10,17) circle (4mm);
\fill[black!5] (10,9) circle (4mm);
\fill[black!9] (10,25) circle (4mm);
\fill[black!1] (10,5) circle (4mm);
\fill[black!7] (10,21) circle (4mm);
\fill[black!3] (10,13) circle (4mm);
\fill[black!12] (10,29) circle (4mm);
\fill[black!0] (10,3) circle (4mm);
\fill[black!11] (10,19) circle (4mm);
\fill[black!3] (10,11) circle (4mm);
\fill[black!11] (10,27) circle (4mm);
\fill[black!1] (10,7) circle (4mm);
\fill[black!14] (10,23) circle (4mm);
\fill[black!3] (10,15) circle (4mm);
\fill[black!10] (10,31) circle (4mm);
\fill[black!0] (26,0) circle (4mm);
\fill[black!9] (26,16) circle (4mm);
\fill[black!7] (26,8) circle (4mm);
\fill[black!12] (26,24) circle (4mm);
\fill[black!1] (26,4) circle (4mm);
\fill[black!11] (26,20) circle (4mm);
\fill[black!5] (26,12) circle (4mm);
\fill[black!9] (26,28) circle (4mm);
\fill[black!0] (26,2) circle (4mm);
\fill[black!11] (26,18) circle (4mm);
\fill[black!3] (26,10) circle (4mm);
\fill[black!14] (26,26) circle (4mm);
\fill[black!2] (26,6) circle (4mm);
\fill[black!10] (26,22) circle (4mm);
\fill[black!7] (26,14) circle (4mm);
\fill[black!12] (26,30) circle (4mm);
\fill[black!0] (26,1) circle (4mm);
\fill[black!12] (26,17) circle (4mm);
\fill[black!7] (26,9) circle (4mm);
\fill[black!10] (26,25) circle (4mm);
\fill[black!2] (26,5) circle (4mm);
\fill[black!14] (26,21) circle (4mm);
\fill[black!3] (26,13) circle (4mm);
\fill[black!11] (26,29) circle (4mm);
\fill[black!0] (26,3) circle (4mm);
\fill[black!9] (26,19) circle (4mm);
\fill[black!5] (26,11) circle (4mm);
\fill[black!11] (26,27) circle (4mm);
\fill[black!1] (26,7) circle (4mm);
\fill[black!12] (26,23) circle (4mm);
\fill[black!7] (26,15) circle (4mm);
\fill[black!9] (26,31) circle (4mm);
\fill[black!10] (6,16) circle (4mm);
\fill[black!5] (6,8) circle (4mm);
\fill[black!6] (6,24) circle (4mm);
\fill[black!1] (6,4) circle (4mm);
\fill[black!9] (6,20) circle (4mm);
\fill[black!4] (6,12) circle (4mm);
\fill[black!9] (6,28) circle (4mm);
\fill[black!1] (6,2) circle (4mm);
\fill[black!7] (6,18) circle (4mm);
\fill[black!5] (6,10) circle (4mm);
\fill[black!10] (6,26) circle (4mm);
\fill[black!2] (6,6) circle (4mm);
\fill[black!8] (6,22) circle (4mm);
\fill[black!5] (6,14) circle (4mm);
\fill[black!11] (6,30) circle (4mm);
\fill[black!0] (6,1) circle (4mm);
\fill[black!11] (6,17) circle (4mm);
\fill[black!5] (6,9) circle (4mm);
\fill[black!8] (6,25) circle (4mm);
\fill[black!2] (6,5) circle (4mm);
\fill[black!10] (6,21) circle (4mm);
\fill[black!5] (6,13) circle (4mm);
\fill[black!7] (6,29) circle (4mm);
\fill[black!1] (6,3) circle (4mm);
\fill[black!9] (6,19) circle (4mm);
\fill[black!4] (6,11) circle (4mm);
\fill[black!9] (6,27) circle (4mm);
\fill[black!1] (6,7) circle (4mm);
\fill[black!6] (6,23) circle (4mm);
\fill[black!5] (6,15) circle (4mm);
\fill[black!10] (6,31) circle (4mm);
\fill[black!0] (22,0) circle (4mm);
\fill[black!8] (22,16) circle (4mm);
\fill[black!5] (22,8) circle (4mm);
\fill[black!8] (22,24) circle (4mm);
\fill[black!1] (22,4) circle (4mm);
\fill[black!11] (22,20) circle (4mm);
\fill[black!3] (22,12) circle (4mm);
\fill[black!9] (22,28) circle (4mm);
\fill[black!0] (22,2) circle (4mm);
\fill[black!10] (22,18) circle (4mm);
\fill[black!4] (22,10) circle (4mm);
\fill[black!8] (22,26) circle (4mm);
\fill[black!1] (22,6) circle (4mm);
\fill[black!10] (22,22) circle (4mm);
\fill[black!4] (22,14) circle (4mm);
\fill[black!11] (22,30) circle (4mm);
\fill[black!1] (22,1) circle (4mm);
\fill[black!11] (22,17) circle (4mm);
\fill[black!4] (22,9) circle (4mm);
\fill[black!10] (22,25) circle (4mm);
\fill[black!1] (22,5) circle (4mm);
\fill[black!8] (22,21) circle (4mm);
\fill[black!4] (22,13) circle (4mm);
\fill[black!10] (22,29) circle (4mm);
\fill[black!0] (22,3) circle (4mm);
\fill[black!9] (22,19) circle (4mm);
\fill[black!3] (22,11) circle (4mm);
\fill[black!11] (22,27) circle (4mm);
\fill[black!1] (22,7) circle (4mm);
\fill[black!8] (22,23) circle (4mm);
\fill[black!5] (22,15) circle (4mm);
\fill[black!8] (22,31) circle (4mm);
\fill[black!12] (14,16) circle (4mm);
\fill[black!5] (14,8) circle (4mm);
\fill[black!12] (14,24) circle (4mm);
\fill[black!2] (14,4) circle (4mm);
\fill[black!11] (14,20) circle (4mm);
\fill[black!2] (14,12) circle (4mm);
\fill[black!9] (14,28) circle (4mm);
\fill[black!2] (14,2) circle (4mm);
\fill[black!12] (14,18) circle (4mm);
\fill[black!5] (14,10) circle (4mm);
\fill[black!9] (14,26) circle (4mm);
\fill[black!1] (14,6) circle (4mm);
\fill[black!12] (14,22) circle (4mm);
\fill[black!4] (14,14) circle (4mm);
\fill[black!10] (14,30) circle (4mm);
\fill[black!1] (14,1) circle (4mm);
\fill[black!10] (14,17) circle (4mm);
\fill[black!4] (14,9) circle (4mm);
\fill[black!12] (14,25) circle (4mm);
\fill[black!1] (14,5) circle (4mm);
\fill[black!9] (14,21) circle (4mm);
\fill[black!5] (14,13) circle (4mm);
\fill[black!12] (14,29) circle (4mm);
\fill[black!2] (14,3) circle (4mm);
\fill[black!9] (14,19) circle (4mm);
\fill[black!2] (14,11) circle (4mm);
\fill[black!11] (14,27) circle (4mm);
\fill[black!2] (14,7) circle (4mm);
\fill[black!12] (14,23) circle (4mm);
\fill[black!5] (14,15) circle (4mm);
\fill[black!12] (14,31) circle (4mm);
\fill[black!8] (30,16) circle (4mm);
\fill[black!3] (30,8) circle (4mm);
\fill[black!13] (30,24) circle (4mm);
\fill[black!4] (30,4) circle (4mm);
\fill[black!12] (30,20) circle (4mm);
\fill[black!4] (30,12) circle (4mm);
\fill[black!9] (30,28) circle (4mm);
\fill[black!2] (30,2) circle (4mm);
\fill[black!9] (30,18) circle (4mm);
\fill[black!3] (30,10) circle (4mm);
\fill[black!9] (30,26) circle (4mm);
\fill[black!2] (30,6) circle (4mm);
\fill[black!10] (30,22) circle (4mm);
\fill[black!4] (30,14) circle (4mm);
\fill[black!12] (30,30) circle (4mm);
\fill[black!0] (30,1) circle (4mm);
\fill[black!12] (30,17) circle (4mm);
\fill[black!4] (30,9) circle (4mm);
\fill[black!10] (30,25) circle (4mm);
\fill[black!2] (30,5) circle (4mm);
\fill[black!9] (30,21) circle (4mm);
\fill[black!3] (30,13) circle (4mm);
\fill[black!9] (30,29) circle (4mm);
\fill[black!2] (30,3) circle (4mm);
\fill[black!9] (30,19) circle (4mm);
\fill[black!4] (30,11) circle (4mm);
\fill[black!12] (30,27) circle (4mm);
\fill[black!4] (30,7) circle (4mm);
\fill[black!13] (30,23) circle (4mm);
\fill[black!3] (30,15) circle (4mm);
\fill[black!8] (30,31) circle (4mm);
\fill[black!0] (1,0) circle (4mm);
\fill[black!11] (1,16) circle (4mm);
\fill[black!2] (1,8) circle (4mm);
\fill[black!10] (1,24) circle (4mm);
\fill[black!2] (1,4) circle (4mm);
\fill[black!8] (1,20) circle (4mm);
\fill[black!5] (1,12) circle (4mm);
\fill[black!13] (1,28) circle (4mm);
\fill[black!10] (1,18) circle (4mm);
\fill[black!4] (1,10) circle (4mm);
\fill[black!11] (1,26) circle (4mm);
\fill[black!3] (1,6) circle (4mm);
\fill[black!12] (1,22) circle (4mm);
\fill[black!4] (1,14) circle (4mm);
\fill[black!11] (1,30) circle (4mm);
\fill[black!0] (1,1) circle (4mm);
\fill[black!11] (1,17) circle (4mm);
\fill[black!4] (1,9) circle (4mm);
\fill[black!12] (1,25) circle (4mm);
\fill[black!3] (1,5) circle (4mm);
\fill[black!11] (1,21) circle (4mm);
\fill[black!4] (1,13) circle (4mm);
\fill[black!10] (1,29) circle (4mm);
\fill[black!13] (1,19) circle (4mm);
\fill[black!5] (1,11) circle (4mm);
\fill[black!8] (1,27) circle (4mm);
\fill[black!2] (1,7) circle (4mm);
\fill[black!10] (1,23) circle (4mm);
\fill[black!2] (1,15) circle (4mm);
\fill[black!11] (1,31) circle (4mm);
\fill[black!13] (17,16) circle (4mm);
\fill[black!3] (17,8) circle (4mm);
\fill[black!7] (17,24) circle (4mm);
\fill[black!2] (17,4) circle (4mm);
\fill[black!9] (17,20) circle (4mm);
\fill[black!5] (17,12) circle (4mm);
\fill[black!9] (17,28) circle (4mm);
\fill[black!11] (17,18) circle (4mm);
\fill[black!5] (17,10) circle (4mm);
\fill[black!12] (17,26) circle (4mm);
\fill[black!2] (17,6) circle (4mm);
\fill[black!11] (17,22) circle (4mm);
\fill[black!5] (17,14) circle (4mm);
\fill[black!12] (17,30) circle (4mm);
\fill[black!12] (17,17) circle (4mm);
\fill[black!5] (17,9) circle (4mm);
\fill[black!11] (17,25) circle (4mm);
\fill[black!2] (17,5) circle (4mm);
\fill[black!12] (17,21) circle (4mm);
\fill[black!5] (17,13) circle (4mm);
\fill[black!11] (17,29) circle (4mm);
\fill[black!9] (17,19) circle (4mm);
\fill[black!5] (17,11) circle (4mm);
\fill[black!9] (17,27) circle (4mm);
\fill[black!2] (17,7) circle (4mm);
\fill[black!7] (17,23) circle (4mm);
\fill[black!3] (17,15) circle (4mm);
\fill[black!13] (17,31) circle (4mm);
\fill[black!11] (9,16) circle (4mm);
\fill[black!4] (9,8) circle (4mm);
\fill[black!11] (9,24) circle (4mm);
\fill[black!1] (9,4) circle (4mm);
\fill[black!7] (9,20) circle (4mm);
\fill[black!4] (9,12) circle (4mm);
\fill[black!10] (9,28) circle (4mm);
\fill[black!2] (9,2) circle (4mm);
\fill[black!8] (9,18) circle (4mm);
\fill[black!2] (9,10) circle (4mm);
\fill[black!10] (9,26) circle (4mm);
\fill[black!2] (9,6) circle (4mm);
\fill[black!11] (9,22) circle (4mm);
\fill[black!6] (9,14) circle (4mm);
\fill[black!10] (9,30) circle (4mm);
\fill[black!0] (9,1) circle (4mm);
\fill[black!10] (9,17) circle (4mm);
\fill[black!6] (9,9) circle (4mm);
\fill[black!11] (9,25) circle (4mm);
\fill[black!2] (9,5) circle (4mm);
\fill[black!10] (9,21) circle (4mm);
\fill[black!2] (9,13) circle (4mm);
\fill[black!8] (9,29) circle (4mm);
\fill[black!2] (9,3) circle (4mm);
\fill[black!10] (9,19) circle (4mm);
\fill[black!4] (9,11) circle (4mm);
\fill[black!7] (9,27) circle (4mm);
\fill[black!1] (9,7) circle (4mm);
\fill[black!11] (9,23) circle (4mm);
\fill[black!4] (9,15) circle (4mm);
\fill[black!11] (9,31) circle (4mm);
\fill[black!0] (25,0) circle (4mm);
\fill[black!11] (25,16) circle (4mm);
\fill[black!4] (25,8) circle (4mm);
\fill[black!8] (25,24) circle (4mm);
\fill[black!2] (25,4) circle (4mm);
\fill[black!7] (25,20) circle (4mm);
\fill[black!4] (25,12) circle (4mm);
\fill[black!8] (25,28) circle (4mm);
\fill[black!1] (25,2) circle (4mm);
\fill[black!14] (25,18) circle (4mm);
\fill[black!4] (25,10) circle (4mm);
\fill[black!14] (25,26) circle (4mm);
\fill[black!1] (25,6) circle (4mm);
\fill[black!6] (25,22) circle (4mm);
\fill[black!3] (25,14) circle (4mm);
\fill[black!10] (25,30) circle (4mm);
\fill[black!0] (25,1) circle (4mm);
\fill[black!10] (25,17) circle (4mm);
\fill[black!3] (25,9) circle (4mm);
\fill[black!6] (25,25) circle (4mm);
\fill[black!1] (25,5) circle (4mm);
\fill[black!14] (25,21) circle (4mm);
\fill[black!4] (25,13) circle (4mm);
\fill[black!14] (25,29) circle (4mm);
\fill[black!1] (25,3) circle (4mm);
\fill[black!8] (25,19) circle (4mm);
\fill[black!4] (25,11) circle (4mm);
\fill[black!7] (25,27) circle (4mm);
\fill[black!2] (25,7) circle (4mm);
\fill[black!8] (25,23) circle (4mm);
\fill[black!4] (25,15) circle (4mm);
\fill[black!11] (25,31) circle (4mm);
\fill[black!8] (5,16) circle (4mm);
\fill[black!5] (5,8) circle (4mm);
\fill[black!11] (5,24) circle (4mm);
\fill[black!2] (5,4) circle (4mm);
\fill[black!13] (5,20) circle (4mm);
\fill[black!5] (5,12) circle (4mm);
\fill[black!11] (5,28) circle (4mm);
\fill[black!2] (5,2) circle (4mm);
\fill[black!9] (5,18) circle (4mm);
\fill[black!4] (5,10) circle (4mm);
\fill[black!7] (5,26) circle (4mm);
\fill[black!1] (5,6) circle (4mm);
\fill[black!9] (5,22) circle (4mm);
\fill[black!5] (5,14) circle (4mm);
\fill[black!10] (5,30) circle (4mm);
\fill[black!0] (5,1) circle (4mm);
\fill[black!10] (5,17) circle (4mm);
\fill[black!5] (5,9) circle (4mm);
\fill[black!9] (5,25) circle (4mm);
\fill[black!1] (5,5) circle (4mm);
\fill[black!7] (5,21) circle (4mm);
\fill[black!4] (5,13) circle (4mm);
\fill[black!9] (5,29) circle (4mm);
\fill[black!2] (5,3) circle (4mm);
\fill[black!11] (5,19) circle (4mm);
\fill[black!5] (5,11) circle (4mm);
\fill[black!13] (5,27) circle (4mm);
\fill[black!2] (5,7) circle (4mm);
\fill[black!11] (5,23) circle (4mm);
\fill[black!5] (5,15) circle (4mm);
\fill[black!8] (5,31) circle (4mm);
\fill[black!11] (21,16) circle (4mm);
\fill[black!4] (21,8) circle (4mm);
\fill[black!10] (21,24) circle (4mm);
\fill[black!2] (21,4) circle (4mm);
\fill[black!9] (21,20) circle (4mm);
\fill[black!4] (21,12) circle (4mm);
\fill[black!10] (21,28) circle (4mm);
\fill[black!1] (21,2) circle (4mm);
\fill[black!10] (21,18) circle (4mm);
\fill[black!4] (21,10) circle (4mm);
\fill[black!13] (21,26) circle (4mm);
\fill[black!2] (21,6) circle (4mm);
\fill[black!10] (21,22) circle (4mm);
\fill[black!3] (21,14) circle (4mm);
\fill[black!11] (21,30) circle (4mm);
\fill[black!1] (21,1) circle (4mm);
\fill[black!11] (21,17) circle (4mm);
\fill[black!3] (21,9) circle (4mm);
\fill[black!10] (21,25) circle (4mm);
\fill[black!2] (21,5) circle (4mm);
\fill[black!13] (21,21) circle (4mm);
\fill[black!4] (21,13) circle (4mm);
\fill[black!10] (21,29) circle (4mm);
\fill[black!1] (21,3) circle (4mm);
\fill[black!10] (21,19) circle (4mm);
\fill[black!4] (21,11) circle (4mm);
\fill[black!9] (21,27) circle (4mm);
\fill[black!2] (21,7) circle (4mm);
\fill[black!10] (21,23) circle (4mm);
\fill[black!4] (21,15) circle (4mm);
\fill[black!11] (21,31) circle (4mm);
\fill[black!9] (13,16) circle (4mm);
\fill[black!5] (13,8) circle (4mm);
\fill[black!11] (13,24) circle (4mm);
\fill[black!1] (13,4) circle (4mm);
\fill[black!9] (13,20) circle (4mm);
\fill[black!2] (13,12) circle (4mm);
\fill[black!12] (13,28) circle (4mm);
\fill[black!1] (13,2) circle (4mm);
\fill[black!9] (13,18) circle (4mm);
\fill[black!6] (13,10) circle (4mm);
\fill[black!13] (13,26) circle (4mm);
\fill[black!2] (13,6) circle (4mm);
\fill[black!10] (13,22) circle (4mm);
\fill[black!3] (13,14) circle (4mm);
\fill[black!14] (13,30) circle (4mm);
\fill[black!14] (13,17) circle (4mm);
\fill[black!3] (13,9) circle (4mm);
\fill[black!10] (13,25) circle (4mm);
\fill[black!2] (13,5) circle (4mm);
\fill[black!13] (13,21) circle (4mm);
\fill[black!6] (13,13) circle (4mm);
\fill[black!9] (13,29) circle (4mm);
\fill[black!1] (13,3) circle (4mm);
\fill[black!12] (13,19) circle (4mm);
\fill[black!2] (13,11) circle (4mm);
\fill[black!9] (13,27) circle (4mm);
\fill[black!1] (13,7) circle (4mm);
\fill[black!11] (13,23) circle (4mm);
\fill[black!5] (13,15) circle (4mm);
\fill[black!9] (13,31) circle (4mm);
\fill[black!10] (29,16) circle (4mm);
\fill[black!4] (29,8) circle (4mm);
\fill[black!9] (29,24) circle (4mm);
\fill[black!2] (29,4) circle (4mm);
\fill[black!10] (29,20) circle (4mm);
\fill[black!5] (29,12) circle (4mm);
\fill[black!11] (29,28) circle (4mm);
\fill[black!0] (29,2) circle (4mm);
\fill[black!11] (29,18) circle (4mm);
\fill[black!3] (29,10) circle (4mm);
\fill[black!9] (29,26) circle (4mm);
\fill[black!0] (29,6) circle (4mm);
\fill[black!8] (29,22) circle (4mm);
\fill[black!4] (29,14) circle (4mm);
\fill[black!9] (29,30) circle (4mm);
\fill[black!0] (29,1) circle (4mm);
\fill[black!9] (29,17) circle (4mm);
\fill[black!4] (29,9) circle (4mm);
\fill[black!8] (29,25) circle (4mm);
\fill[black!0] (29,5) circle (4mm);
\fill[black!9] (29,21) circle (4mm);
\fill[black!3] (29,13) circle (4mm);
\fill[black!11] (29,29) circle (4mm);
\fill[black!0] (29,3) circle (4mm);
\fill[black!11] (29,19) circle (4mm);
\fill[black!5] (29,11) circle (4mm);
\fill[black!10] (29,27) circle (4mm);
\fill[black!2] (29,7) circle (4mm);
\fill[black!9] (29,23) circle (4mm);
\fill[black!4] (29,15) circle (4mm);
\fill[black!10] (29,31) circle (4mm);
\fill[black!10] (3,16) circle (4mm);
\fill[black!4] (3,8) circle (4mm);
\fill[black!10] (3,24) circle (4mm);
\fill[black!3] (3,4) circle (4mm);
\fill[black!10] (3,20) circle (4mm);
\fill[black!5] (3,12) circle (4mm);
\fill[black!11] (3,28) circle (4mm);
\fill[black!0] (3,2) circle (4mm);
\fill[black!11] (3,18) circle (4mm);
\fill[black!4] (3,10) circle (4mm);
\fill[black!10] (3,26) circle (4mm);
\fill[black!1] (3,6) circle (4mm);
\fill[black!11] (3,22) circle (4mm);
\fill[black!7] (3,14) circle (4mm);
\fill[black!11] (3,30) circle (4mm);
\fill[black!11] (3,17) circle (4mm);
\fill[black!7] (3,9) circle (4mm);
\fill[black!11] (3,25) circle (4mm);
\fill[black!1] (3,5) circle (4mm);
\fill[black!10] (3,21) circle (4mm);
\fill[black!4] (3,13) circle (4mm);
\fill[black!11] (3,29) circle (4mm);
\fill[black!0] (3,3) circle (4mm);
\fill[black!11] (3,19) circle (4mm);
\fill[black!5] (3,11) circle (4mm);
\fill[black!10] (3,27) circle (4mm);
\fill[black!3] (3,7) circle (4mm);
\fill[black!10] (3,23) circle (4mm);
\fill[black!4] (3,15) circle (4mm);
\fill[black!10] (3,31) circle (4mm);
\fill[black!7] (19,16) circle (4mm);
\fill[black!4] (19,8) circle (4mm);
\fill[black!9] (19,24) circle (4mm);
\fill[black!1] (19,4) circle (4mm);
\fill[black!10] (19,20) circle (4mm);
\fill[black!6] (19,12) circle (4mm);
\fill[black!9] (19,28) circle (4mm);
\fill[black!1] (19,2) circle (4mm);
\fill[black!10] (19,18) circle (4mm);
\fill[black!3] (19,10) circle (4mm);
\fill[black!10] (19,26) circle (4mm);
\fill[black!2] (19,6) circle (4mm);
\fill[black!9] (19,22) circle (4mm);
\fill[black!4] (19,14) circle (4mm);
\fill[black!7] (19,30) circle (4mm);
\fill[black!7] (19,17) circle (4mm);
\fill[black!4] (19,9) circle (4mm);
\fill[black!9] (19,25) circle (4mm);
\fill[black!2] (19,5) circle (4mm);
\fill[black!10] (19,21) circle (4mm);
\fill[black!3] (19,13) circle (4mm);
\fill[black!10] (19,29) circle (4mm);
\fill[black!1] (19,3) circle (4mm);
\fill[black!9] (19,19) circle (4mm);
\fill[black!6] (19,11) circle (4mm);
\fill[black!10] (19,27) circle (4mm);
\fill[black!1] (19,7) circle (4mm);
\fill[black!9] (19,23) circle (4mm);
\fill[black!4] (19,15) circle (4mm);
\fill[black!7] (19,31) circle (4mm);
\fill[black!7] (11,16) circle (4mm);
\fill[black!4] (11,8) circle (4mm);
\fill[black!6] (11,24) circle (4mm);
\fill[black!2] (11,4) circle (4mm);
\fill[black!9] (11,20) circle (4mm);
\fill[black!4] (11,12) circle (4mm);
\fill[black!10] (11,28) circle (4mm);
\fill[black!0] (11,2) circle (4mm);
\fill[black!9] (11,18) circle (4mm);
\fill[black!4] (11,10) circle (4mm);
\fill[black!9] (11,26) circle (4mm);
\fill[black!2] (11,6) circle (4mm);
\fill[black!10] (11,22) circle (4mm);
\fill[black!4] (11,14) circle (4mm);
\fill[black!7] (11,30) circle (4mm);
\fill[black!0] (11,1) circle (4mm);
\fill[black!7] (11,17) circle (4mm);
\fill[black!4] (11,9) circle (4mm);
\fill[black!10] (11,25) circle (4mm);
\fill[black!2] (11,5) circle (4mm);
\fill[black!9] (11,21) circle (4mm);
\fill[black!4] (11,13) circle (4mm);
\fill[black!9] (11,29) circle (4mm);
\fill[black!0] (11,3) circle (4mm);
\fill[black!10] (11,19) circle (4mm);
\fill[black!4] (11,11) circle (4mm);
\fill[black!9] (11,27) circle (4mm);
\fill[black!2] (11,7) circle (4mm);
\fill[black!6] (11,23) circle (4mm);
\fill[black!4] (11,15) circle (4mm);
\fill[black!7] (11,31) circle (4mm);
\fill[black!0] (27,0) circle (4mm);
\fill[black!11] (27,16) circle (4mm);
\fill[black!5] (27,8) circle (4mm);
\fill[black!8] (27,24) circle (4mm);
\fill[black!2] (27,4) circle (4mm);
\fill[black!10] (27,20) circle (4mm);
\fill[black!3] (27,12) circle (4mm);
\fill[black!10] (27,28) circle (4mm);
\fill[black!0] (27,2) circle (4mm);
\fill[black!12] (27,18) circle (4mm);
\fill[black!3] (27,10) circle (4mm);
\fill[black!9] (27,26) circle (4mm);
\fill[black!1] (27,6) circle (4mm);
\fill[black!12] (27,22) circle (4mm);
\fill[black!4] (27,14) circle (4mm);
\fill[black!11] (27,30) circle (4mm);
\fill[black!0] (27,1) circle (4mm);
\fill[black!11] (27,17) circle (4mm);
\fill[black!4] (27,9) circle (4mm);
\fill[black!12] (27,25) circle (4mm);
\fill[black!1] (27,5) circle (4mm);
\fill[black!9] (27,21) circle (4mm);
\fill[black!3] (27,13) circle (4mm);
\fill[black!12] (27,29) circle (4mm);
\fill[black!0] (27,3) circle (4mm);
\fill[black!10] (27,19) circle (4mm);
\fill[black!3] (27,11) circle (4mm);
\fill[black!10] (27,27) circle (4mm);
\fill[black!2] (27,7) circle (4mm);
\fill[black!8] (27,23) circle (4mm);
\fill[black!5] (27,15) circle (4mm);
\fill[black!11] (27,31) circle (4mm);
\fill[black!9] (7,16) circle (4mm);
\fill[black!3] (7,8) circle (4mm);
\fill[black!10] (7,24) circle (4mm);
\fill[black!2] (7,4) circle (4mm);
\fill[black!11] (7,20) circle (4mm);
\fill[black!3] (7,12) circle (4mm);
\fill[black!10] (7,28) circle (4mm);
\fill[black!0] (7,2) circle (4mm);
\fill[black!6] (7,18) circle (4mm);
\fill[black!5] (7,10) circle (4mm);
\fill[black!9] (7,26) circle (4mm);
\fill[black!2] (7,6) circle (4mm);
\fill[black!11] (7,22) circle (4mm);
\fill[black!4] (7,14) circle (4mm);
\fill[black!10] (7,30) circle (4mm);
\fill[black!0] (7,1) circle (4mm);
\fill[black!10] (7,17) circle (4mm);
\fill[black!4] (7,9) circle (4mm);
\fill[black!11] (7,25) circle (4mm);
\fill[black!2] (7,5) circle (4mm);
\fill[black!9] (7,21) circle (4mm);
\fill[black!5] (7,13) circle (4mm);
\fill[black!6] (7,29) circle (4mm);
\fill[black!0] (7,3) circle (4mm);
\fill[black!10] (7,19) circle (4mm);
\fill[black!3] (7,11) circle (4mm);
\fill[black!11] (7,27) circle (4mm);
\fill[black!2] (7,7) circle (4mm);
\fill[black!10] (7,23) circle (4mm);
\fill[black!3] (7,15) circle (4mm);
\fill[black!9] (7,31) circle (4mm);
\fill[black!12] (23,16) circle (4mm);
\fill[black!2] (23,8) circle (4mm);
\fill[black!7] (23,24) circle (4mm);
\fill[black!3] (23,4) circle (4mm);
\fill[black!10] (23,20) circle (4mm);
\fill[black!5] (23,12) circle (4mm);
\fill[black!8] (23,28) circle (4mm);
\fill[black!0] (23,2) circle (4mm);
\fill[black!6] (23,18) circle (4mm);
\fill[black!5] (23,10) circle (4mm);
\fill[black!10] (23,26) circle (4mm);
\fill[black!2] (23,6) circle (4mm);
\fill[black!10] (23,22) circle (4mm);
\fill[black!6] (23,14) circle (4mm);
\fill[black!11] (23,30) circle (4mm);
\fill[black!0] (23,1) circle (4mm);
\fill[black!11] (23,17) circle (4mm);
\fill[black!6] (23,9) circle (4mm);
\fill[black!10] (23,25) circle (4mm);
\fill[black!2] (23,5) circle (4mm);
\fill[black!10] (23,21) circle (4mm);
\fill[black!5] (23,13) circle (4mm);
\fill[black!6] (23,29) circle (4mm);
\fill[black!0] (23,3) circle (4mm);
\fill[black!8] (23,19) circle (4mm);
\fill[black!5] (23,11) circle (4mm);
\fill[black!10] (23,27) circle (4mm);
\fill[black!3] (23,7) circle (4mm);
\fill[black!7] (23,23) circle (4mm);
\fill[black!2] (23,15) circle (4mm);
\fill[black!12] (23,31) circle (4mm);
\fill[black!14] (15,16) circle (4mm);
\fill[black!7] (15,8) circle (4mm);
\fill[black!10] (15,24) circle (4mm);
\fill[black!2] (15,4) circle (4mm);
\fill[black!7] (15,20) circle (4mm);
\fill[black!3] (15,12) circle (4mm);
\fill[black!10] (15,28) circle (4mm);
\fill[black!0] (15,2) circle (4mm);
\fill[black!11] (15,18) circle (4mm);
\fill[black!4] (15,10) circle (4mm);
\fill[black!10] (15,26) circle (4mm);
\fill[black!3] (15,6) circle (4mm);
\fill[black!9] (15,22) circle (4mm);
\fill[black!3] (15,14) circle (4mm);
\fill[black!10] (15,30) circle (4mm);
\fill[black!10] (15,17) circle (4mm);
\fill[black!3] (15,9) circle (4mm);
\fill[black!9] (15,25) circle (4mm);
\fill[black!3] (15,5) circle (4mm);
\fill[black!10] (15,21) circle (4mm);
\fill[black!4] (15,13) circle (4mm);
\fill[black!11] (15,29) circle (4mm);
\fill[black!0] (15,3) circle (4mm);
\fill[black!10] (15,19) circle (4mm);
\fill[black!3] (15,11) circle (4mm);
\fill[black!7] (15,27) circle (4mm);
\fill[black!2] (15,7) circle (4mm);
\fill[black!10] (15,23) circle (4mm);
\fill[black!7] (15,15) circle (4mm);
\fill[black!14] (15,31) circle (4mm);
\fill[black!1] (31,0) circle (4mm);
\fill[black!11] (31,16) circle (4mm);
\fill[black!3] (31,8) circle (4mm);
\fill[black!9] (31,24) circle (4mm);
\fill[black!2] (31,4) circle (4mm);
\fill[black!12] (31,20) circle (4mm);
\fill[black!2] (31,12) circle (4mm);
\fill[black!10] (31,28) circle (4mm);
\fill[black!2] (31,2) circle (4mm);
\fill[black!10] (31,18) circle (4mm);
\fill[black!5] (31,10) circle (4mm);
\fill[black!10] (31,26) circle (4mm);
\fill[black!2] (31,6) circle (4mm);
\fill[black!13] (31,22) circle (4mm);
\fill[black!4] (31,14) circle (4mm);
\fill[black!11] (31,30) circle (4mm);
\fill[black!1] (31,1) circle (4mm);
\fill[black!11] (31,17) circle (4mm);
\fill[black!4] (31,9) circle (4mm);
\fill[black!13] (31,25) circle (4mm);
\fill[black!2] (31,5) circle (4mm);
\fill[black!10] (31,21) circle (4mm);
\fill[black!5] (31,13) circle (4mm);
\fill[black!10] (31,29) circle (4mm);
\fill[black!2] (31,3) circle (4mm);
\fill[black!10] (31,19) circle (4mm);
\fill[black!2] (31,11) circle (4mm);
\fill[black!12] (31,27) circle (4mm);
\fill[black!2] (31,7) circle (4mm);
\fill[black!9] (31,23) circle (4mm);
\fill[black!3] (31,15) circle (4mm);
\fill[black!11] (31,31) circle (4mm);
\end{tikzpicture}

%% file: statspairs.tex
\begin{tikzpicture}[scale=0.18]
\draw[->] (-0.5,-2)--(31.5,-2);
\draw[->] (-2,-0.5)--(-2,31.5);
\node[scale=0.8] at (31,-3.5) { $\ZZ_2$ };
\node[scale=0.8] at (-3.5,31) { $\ZZ_2$ };
\fill[black!1] (0,0) circle (4mm);
\fill[black!55] (0,16) circle (4mm);
\fill[black!34] (0,8) circle (4mm);
\fill[black!57] (0,24) circle (4mm);
\fill[black!18] (0,4) circle (4mm);
\fill[black!57] (0,20) circle (4mm);
\fill[black!34] (0,12) circle (4mm);
\fill[black!52] (0,28) circle (4mm);
\fill[black!9] (0,2) circle (4mm);
\fill[black!56] (0,18) circle (4mm);
\fill[black!35] (0,10) circle (4mm);
\fill[black!53] (0,26) circle (4mm);
\fill[black!14] (0,6) circle (4mm);
\fill[black!54] (0,22) circle (4mm);
\fill[black!34] (0,14) circle (4mm);
\fill[black!52] (0,30) circle (4mm);
\fill[black!5] (0,1) circle (4mm);
\fill[black!55] (0,17) circle (4mm);
\fill[black!30] (0,9) circle (4mm);
\fill[black!47] (0,25) circle (4mm);
\fill[black!17] (0,5) circle (4mm);
\fill[black!54] (0,21) circle (4mm);
\fill[black!35] (0,13) circle (4mm);
\fill[black!53] (0,29) circle (4mm);
\fill[black!10] (0,3) circle (4mm);
\fill[black!50] (0,19) circle (4mm);
\fill[black!33] (0,11) circle (4mm);
\fill[black!50] (0,27) circle (4mm);
\fill[black!20] (0,7) circle (4mm);
\fill[black!57] (0,23) circle (4mm);
\fill[black!32] (0,15) circle (4mm);
\fill[black!61] (0,31) circle (4mm);
\fill[black!55] (16,0) circle (4mm);
\fill[black!0] (16,16) circle (4mm);
\fill[black!60] (16,8) circle (4mm);
\fill[black!33] (16,24) circle (4mm);
\fill[black!50] (16,4) circle (4mm);
\fill[black!17] (16,20) circle (4mm);
\fill[black!46] (16,12) circle (4mm);
\fill[black!35] (16,28) circle (4mm);
\fill[black!49] (16,2) circle (4mm);
\fill[black!9] (16,18) circle (4mm);
\fill[black!53] (16,10) circle (4mm);
\fill[black!35] (16,26) circle (4mm);
\fill[black!54] (16,6) circle (4mm);
\fill[black!14] (16,22) circle (4mm);
\fill[black!55] (16,14) circle (4mm);
\fill[black!31] (16,30) circle (4mm);
\fill[black!50] (16,1) circle (4mm);
\fill[black!3] (16,17) circle (4mm);
\fill[black!54] (16,9) circle (4mm);
\fill[black!36] (16,25) circle (4mm);
\fill[black!56] (16,5) circle (4mm);
\fill[black!15] (16,21) circle (4mm);
\fill[black!51] (16,13) circle (4mm);
\fill[black!34] (16,29) circle (4mm);
\fill[black!54] (16,3) circle (4mm);
\fill[black!9] (16,19) circle (4mm);
\fill[black!52] (16,11) circle (4mm);
\fill[black!34] (16,27) circle (4mm);
\fill[black!52] (16,7) circle (4mm);
\fill[black!18] (16,23) circle (4mm);
\fill[black!51] (16,15) circle (4mm);
\fill[black!33] (16,31) circle (4mm);
\fill[black!34] (8,0) circle (4mm);
\fill[black!60] (8,16) circle (4mm);
\fill[black!1] (8,8) circle (4mm);
\fill[black!53] (8,24) circle (4mm);
\fill[black!34] (8,4) circle (4mm);
\fill[black!57] (8,20) circle (4mm);
\fill[black!15] (8,12) circle (4mm);
\fill[black!52] (8,28) circle (4mm);
\fill[black!34] (8,2) circle (4mm);
\fill[black!49] (8,18) circle (4mm);
\fill[black!10] (8,10) circle (4mm);
\fill[black!60] (8,26) circle (4mm);
\fill[black!30] (8,6) circle (4mm);
\fill[black!52] (8,22) circle (4mm);
\fill[black!20] (8,14) circle (4mm);
\fill[black!55] (8,30) circle (4mm);
\fill[black!36] (8,1) circle (4mm);
\fill[black!54] (8,17) circle (4mm);
\fill[black!3] (8,9) circle (4mm);
\fill[black!57] (8,25) circle (4mm);
\fill[black!37] (8,5) circle (4mm);
\fill[black!52] (8,21) circle (4mm);
\fill[black!19] (8,13) circle (4mm);
\fill[black!54] (8,29) circle (4mm);
\fill[black!32] (8,3) circle (4mm);
\fill[black!61] (8,19) circle (4mm);
\fill[black!11] (8,11) circle (4mm);
\fill[black!53] (8,27) circle (4mm);
\fill[black!36] (8,7) circle (4mm);
\fill[black!58] (8,23) circle (4mm);
\fill[black!19] (8,15) circle (4mm);
\fill[black!53] (8,31) circle (4mm);
\fill[black!57] (24,0) circle (4mm);
\fill[black!33] (24,16) circle (4mm);
\fill[black!53] (24,8) circle (4mm);
\fill[black!3] (24,24) circle (4mm);
\fill[black!61] (24,4) circle (4mm);
\fill[black!32] (24,20) circle (4mm);
\fill[black!60] (24,12) circle (4mm);
\fill[black!18] (24,28) circle (4mm);
\fill[black!49] (24,2) circle (4mm);
\fill[black!30] (24,18) circle (4mm);
\fill[black!53] (24,10) circle (4mm);
\fill[black!9] (24,26) circle (4mm);
\fill[black!53] (24,6) circle (4mm);
\fill[black!31] (24,22) circle (4mm);
\fill[black!53] (24,14) circle (4mm);
\fill[black!14] (24,30) circle (4mm);
\fill[black!52] (24,1) circle (4mm);
\fill[black!32] (24,17) circle (4mm);
\fill[black!50] (24,9) circle (4mm);
\fill[black!4] (24,25) circle (4mm);
\fill[black!54] (24,5) circle (4mm);
\fill[black!29] (24,21) circle (4mm);
\fill[black!58] (24,13) circle (4mm);
\fill[black!17] (24,29) circle (4mm);
\fill[black!57] (24,3) circle (4mm);
\fill[black!35] (24,19) circle (4mm);
\fill[black!47] (24,11) circle (4mm);
\fill[black!8] (24,27) circle (4mm);
\fill[black!51] (24,7) circle (4mm);
\fill[black!32] (24,23) circle (4mm);
\fill[black!50] (24,15) circle (4mm);
\fill[black!16] (24,31) circle (4mm);
\fill[black!18] (4,0) circle (4mm);
\fill[black!50] (4,16) circle (4mm);
\fill[black!34] (4,8) circle (4mm);
\fill[black!61] (4,24) circle (4mm);
\fill[black!0] (4,4) circle (4mm);
\fill[black!56] (4,20) circle (4mm);
\fill[black!28] (4,12) circle (4mm);
\fill[black!59] (4,28) circle (4mm);
\fill[black!15] (4,2) circle (4mm);
\fill[black!51] (4,18) circle (4mm);
\fill[black!33] (4,10) circle (4mm);
\fill[black!54] (4,26) circle (4mm);
\fill[black!7] (4,6) circle (4mm);
\fill[black!55] (4,22) circle (4mm);
\fill[black!33] (4,14) circle (4mm);
\fill[black!53] (4,30) circle (4mm);
\fill[black!17] (4,1) circle (4mm);
\fill[black!55] (4,17) circle (4mm);
\fill[black!31] (4,9) circle (4mm);
\fill[black!50] (4,25) circle (4mm);
\fill[black!3] (4,5) circle (4mm);
\fill[black!52] (4,21) circle (4mm);
\fill[black!34] (4,13) circle (4mm);
\fill[black!52] (4,29) circle (4mm);
\fill[black!12] (4,3) circle (4mm);
\fill[black!52] (4,19) circle (4mm);
\fill[black!38] (4,11) circle (4mm);
\fill[black!52] (4,27) circle (4mm);
\fill[black!10] (4,7) circle (4mm);
\fill[black!59] (4,23) circle (4mm);
\fill[black!30] (4,15) circle (4mm);
\fill[black!50] (4,31) circle (4mm);
\fill[black!57] (20,0) circle (4mm);
\fill[black!17] (20,16) circle (4mm);
\fill[black!57] (20,8) circle (4mm);
\fill[black!32] (20,24) circle (4mm);
\fill[black!56] (20,4) circle (4mm);
\fill[black!52] (20,12) circle (4mm);
\fill[black!33] (20,28) circle (4mm);
\fill[black!59] (20,2) circle (4mm);
\fill[black!16] (20,18) circle (4mm);
\fill[black!53] (20,10) circle (4mm);
\fill[black!31] (20,26) circle (4mm);
\fill[black!51] (20,6) circle (4mm);
\fill[black!8] (20,22) circle (4mm);
\fill[black!55] (20,14) circle (4mm);
\fill[black!33] (20,30) circle (4mm);
\fill[black!61] (20,1) circle (4mm);
\fill[black!16] (20,17) circle (4mm);
\fill[black!50] (20,9) circle (4mm);
\fill[black!35] (20,25) circle (4mm);
\fill[black!53] (20,5) circle (4mm);
\fill[black!5] (20,21) circle (4mm);
\fill[black!60] (20,13) circle (4mm);
\fill[black!32] (20,29) circle (4mm);
\fill[black!52] (20,3) circle (4mm);
\fill[black!18] (20,19) circle (4mm);
\fill[black!55] (20,11) circle (4mm);
\fill[black!30] (20,27) circle (4mm);
\fill[black!57] (20,7) circle (4mm);
\fill[black!11] (20,23) circle (4mm);
\fill[black!54] (20,15) circle (4mm);
\fill[black!34] (20,31) circle (4mm);
\fill[black!34] (12,0) circle (4mm);
\fill[black!46] (12,16) circle (4mm);
\fill[black!15] (12,8) circle (4mm);
\fill[black!60] (12,24) circle (4mm);
\fill[black!28] (12,4) circle (4mm);
\fill[black!52] (12,20) circle (4mm);
\fill[black!2] (12,12) circle (4mm);
\fill[black!55] (12,28) circle (4mm);
\fill[black!37] (12,2) circle (4mm);
\fill[black!51] (12,18) circle (4mm);
\fill[black!16] (12,10) circle (4mm);
\fill[black!54] (12,26) circle (4mm);
\fill[black!34] (12,6) circle (4mm);
\fill[black!56] (12,22) circle (4mm);
\fill[black!9] (12,14) circle (4mm);
\fill[black!56] (12,30) circle (4mm);
\fill[black!33] (12,1) circle (4mm);
\fill[black!55] (12,17) circle (4mm);
\fill[black!23] (12,9) circle (4mm);
\fill[black!51] (12,25) circle (4mm);
\fill[black!36] (12,5) circle (4mm);
\fill[black!50] (12,21) circle (4mm);
\fill[black!3] (12,13) circle (4mm);
\fill[black!53] (12,29) circle (4mm);
\fill[black!34] (12,3) circle (4mm);
\fill[black!52] (12,19) circle (4mm);
\fill[black!16] (12,11) circle (4mm);
\fill[black!56] (12,27) circle (4mm);
\fill[black!29] (12,7) circle (4mm);
\fill[black!50] (12,23) circle (4mm);
\fill[black!9] (12,15) circle (4mm);
\fill[black!52] (12,31) circle (4mm);
\fill[black!52] (28,0) circle (4mm);
\fill[black!35] (28,16) circle (4mm);
\fill[black!52] (28,8) circle (4mm);
\fill[black!18] (28,24) circle (4mm);
\fill[black!59] (28,4) circle (4mm);
\fill[black!33] (28,20) circle (4mm);
\fill[black!55] (28,12) circle (4mm);
\fill[black!0] (28,28) circle (4mm);
\fill[black!54] (28,2) circle (4mm);
\fill[black!37] (28,18) circle (4mm);
\fill[black!54] (28,10) circle (4mm);
\fill[black!17] (28,26) circle (4mm);
\fill[black!62] (28,6) circle (4mm);
\fill[black!29] (28,22) circle (4mm);
\fill[black!49] (28,14) circle (4mm);
\fill[black!7] (28,30) circle (4mm);
\fill[black!59] (28,1) circle (4mm);
\fill[black!36] (28,17) circle (4mm);
\fill[black!57] (28,9) circle (4mm);
\fill[black!19] (28,25) circle (4mm);
\fill[black!60] (28,5) circle (4mm);
\fill[black!34] (28,21) circle (4mm);
\fill[black!61] (28,13) circle (4mm);
\fill[black!2] (28,29) circle (4mm);
\fill[black!52] (28,3) circle (4mm);
\fill[black!29] (28,19) circle (4mm);
\fill[black!51] (28,11) circle (4mm);
\fill[black!22] (28,27) circle (4mm);
\fill[black!53] (28,7) circle (4mm);
\fill[black!31] (28,23) circle (4mm);
\fill[black!52] (28,15) circle (4mm);
\fill[black!11] (28,31) circle (4mm);
\fill[black!9] (2,0) circle (4mm);
\fill[black!49] (2,16) circle (4mm);
\fill[black!34] (2,8) circle (4mm);
\fill[black!49] (2,24) circle (4mm);
\fill[black!15] (2,4) circle (4mm);
\fill[black!59] (2,20) circle (4mm);
\fill[black!37] (2,12) circle (4mm);
\fill[black!54] (2,28) circle (4mm);
\fill[black!2] (2,2) circle (4mm);
\fill[black!55] (2,18) circle (4mm);
\fill[black!35] (2,10) circle (4mm);
\fill[black!59] (2,26) circle (4mm);
\fill[black!11] (2,6) circle (4mm);
\fill[black!57] (2,22) circle (4mm);
\fill[black!28] (2,14) circle (4mm);
\fill[black!54] (2,30) circle (4mm);
\fill[black!11] (2,1) circle (4mm);
\fill[black!49] (2,17) circle (4mm);
\fill[black!37] (2,9) circle (4mm);
\fill[black!54] (2,25) circle (4mm);
\fill[black!15] (2,5) circle (4mm);
\fill[black!54] (2,21) circle (4mm);
\fill[black!33] (2,13) circle (4mm);
\fill[black!55] (2,29) circle (4mm);
\fill[black!4] (2,3) circle (4mm);
\fill[black!54] (2,19) circle (4mm);
\fill[black!32] (2,11) circle (4mm);
\fill[black!53] (2,27) circle (4mm);
\fill[black!15] (2,7) circle (4mm);
\fill[black!55] (2,23) circle (4mm);
\fill[black!34] (2,15) circle (4mm);
\fill[black!50] (2,31) circle (4mm);
\fill[black!56] (18,0) circle (4mm);
\fill[black!9] (18,16) circle (4mm);
\fill[black!49] (18,8) circle (4mm);
\fill[black!30] (18,24) circle (4mm);
\fill[black!51] (18,4) circle (4mm);
\fill[black!16] (18,20) circle (4mm);
\fill[black!51] (18,12) circle (4mm);
\fill[black!37] (18,28) circle (4mm);
\fill[black!55] (18,2) circle (4mm);
\fill[black!55] (18,10) circle (4mm);
\fill[black!34] (18,26) circle (4mm);
\fill[black!63] (18,6) circle (4mm);
\fill[black!14] (18,22) circle (4mm);
\fill[black!57] (18,14) circle (4mm);
\fill[black!35] (18,30) circle (4mm);
\fill[black!51] (18,1) circle (4mm);
\fill[black!7] (18,17) circle (4mm);
\fill[black!55] (18,9) circle (4mm);
\fill[black!37] (18,25) circle (4mm);
\fill[black!54] (18,5) circle (4mm);
\fill[black!18] (18,21) circle (4mm);
\fill[black!53] (18,13) circle (4mm);
\fill[black!33] (18,29) circle (4mm);
\fill[black!60] (18,3) circle (4mm);
\fill[black!4] (18,19) circle (4mm);
\fill[black!53] (18,11) circle (4mm);
\fill[black!33] (18,27) circle (4mm);
\fill[black!57] (18,7) circle (4mm);
\fill[black!19] (18,23) circle (4mm);
\fill[black!52] (18,15) circle (4mm);
\fill[black!28] (18,31) circle (4mm);
\fill[black!35] (10,0) circle (4mm);
\fill[black!53] (10,16) circle (4mm);
\fill[black!10] (10,8) circle (4mm);
\fill[black!53] (10,24) circle (4mm);
\fill[black!33] (10,4) circle (4mm);
\fill[black!53] (10,20) circle (4mm);
\fill[black!16] (10,12) circle (4mm);
\fill[black!54] (10,28) circle (4mm);
\fill[black!35] (10,2) circle (4mm);
\fill[black!55] (10,18) circle (4mm);
\fill[black!0] (10,10) circle (4mm);
\fill[black!55] (10,26) circle (4mm);
\fill[black!34] (10,6) circle (4mm);
\fill[black!48] (10,22) circle (4mm);
\fill[black!17] (10,14) circle (4mm);
\fill[black!50] (10,30) circle (4mm);
\fill[black!28] (10,1) circle (4mm);
\fill[black!54] (10,17) circle (4mm);
\fill[black!5] (10,9) circle (4mm);
\fill[black!52] (10,25) circle (4mm);
\fill[black!34] (10,5) circle (4mm);
\fill[black!51] (10,21) circle (4mm);
\fill[black!16] (10,13) circle (4mm);
\fill[black!51] (10,29) circle (4mm);
\fill[black!32] (10,3) circle (4mm);
\fill[black!46] (10,19) circle (4mm);
\fill[black!5] (10,11) circle (4mm);
\fill[black!60] (10,27) circle (4mm);
\fill[black!32] (10,7) circle (4mm);
\fill[black!57] (10,23) circle (4mm);
\fill[black!17] (10,15) circle (4mm);
\fill[black!54] (10,31) circle (4mm);
\fill[black!53] (26,0) circle (4mm);
\fill[black!35] (26,16) circle (4mm);
\fill[black!60] (26,8) circle (4mm);
\fill[black!9] (26,24) circle (4mm);
\fill[black!54] (26,4) circle (4mm);
\fill[black!31] (26,20) circle (4mm);
\fill[black!54] (26,12) circle (4mm);
\fill[black!17] (26,28) circle (4mm);
\fill[black!59] (26,2) circle (4mm);
\fill[black!34] (26,18) circle (4mm);
\fill[black!55] (26,10) circle (4mm);
\fill[black!1] (26,26) circle (4mm);
\fill[black!57] (26,6) circle (4mm);
\fill[black!29] (26,22) circle (4mm);
\fill[black!54] (26,14) circle (4mm);
\fill[black!20] (26,30) circle (4mm);
\fill[black!55] (26,1) circle (4mm);
\fill[black!28] (26,17) circle (4mm);
\fill[black!54] (26,9) circle (4mm);
\fill[black!9] (26,25) circle (4mm);
\fill[black!64] (26,5) circle (4mm);
\fill[black!31] (26,21) circle (4mm);
\fill[black!54] (26,13) circle (4mm);
\fill[black!20] (26,29) circle (4mm);
\fill[black!53] (26,3) circle (4mm);
\fill[black!35] (26,19) circle (4mm);
\fill[black!52] (26,11) circle (4mm);
\fill[black!4] (26,27) circle (4mm);
\fill[black!50] (26,7) circle (4mm);
\fill[black!32] (26,23) circle (4mm);
\fill[black!52] (26,15) circle (4mm);
\fill[black!15] (26,31) circle (4mm);
\fill[black!14] (6,0) circle (4mm);
\fill[black!54] (6,16) circle (4mm);
\fill[black!30] (6,8) circle (4mm);
\fill[black!53] (6,24) circle (4mm);
\fill[black!7] (6,4) circle (4mm);
\fill[black!51] (6,20) circle (4mm);
\fill[black!34] (6,12) circle (4mm);
\fill[black!62] (6,28) circle (4mm);
\fill[black!11] (6,2) circle (4mm);
\fill[black!63] (6,18) circle (4mm);
\fill[black!34] (6,10) circle (4mm);
\fill[black!57] (6,26) circle (4mm);
\fill[black!1] (6,6) circle (4mm);
\fill[black!55] (6,22) circle (4mm);
\fill[black!36] (6,14) circle (4mm);
\fill[black!51] (6,30) circle (4mm);
\fill[black!18] (6,1) circle (4mm);
\fill[black!57] (6,17) circle (4mm);
\fill[black!30] (6,9) circle (4mm);
\fill[black!50] (6,25) circle (4mm);
\fill[black!7] (6,5) circle (4mm);
\fill[black!61] (6,21) circle (4mm);
\fill[black!36] (6,13) circle (4mm);
\fill[black!56] (6,29) circle (4mm);
\fill[black!16] (6,3) circle (4mm);
\fill[black!54] (6,19) circle (4mm);
\fill[black!34] (6,11) circle (4mm);
\fill[black!56] (6,27) circle (4mm);
\fill[black!6] (6,7) circle (4mm);
\fill[black!54] (6,23) circle (4mm);
\fill[black!34] (6,15) circle (4mm);
\fill[black!55] (6,31) circle (4mm);
\fill[black!54] (22,0) circle (4mm);
\fill[black!14] (22,16) circle (4mm);
\fill[black!52] (22,8) circle (4mm);
\fill[black!31] (22,24) circle (4mm);
\fill[black!55] (22,4) circle (4mm);
\fill[black!8] (22,20) circle (4mm);
\fill[black!56] (22,12) circle (4mm);
\fill[black!29] (22,28) circle (4mm);
\fill[black!57] (22,2) circle (4mm);
\fill[black!14] (22,18) circle (4mm);
\fill[black!48] (22,10) circle (4mm);
\fill[black!29] (22,26) circle (4mm);
\fill[black!55] (22,6) circle (4mm);
\fill[black!1] (22,22) circle (4mm);
\fill[black!51] (22,14) circle (4mm);
\fill[black!33] (22,30) circle (4mm);
\fill[black!52] (22,1) circle (4mm);
\fill[black!15] (22,17) circle (4mm);
\fill[black!53] (22,9) circle (4mm);
\fill[black!34] (22,25) circle (4mm);
\fill[black!56] (22,5) circle (4mm);
\fill[black!7] (22,21) circle (4mm);
\fill[black!57] (22,13) circle (4mm);
\fill[black!29] (22,29) circle (4mm);
\fill[black!50] (22,3) circle (4mm);
\fill[black!18] (22,19) circle (4mm);
\fill[black!51] (22,11) circle (4mm);
\fill[black!28] (22,27) circle (4mm);
\fill[black!52] (22,7) circle (4mm);
\fill[black!5] (22,23) circle (4mm);
\fill[black!58] (22,15) circle (4mm);
\fill[black!35] (22,31) circle (4mm);
\fill[black!34] (14,0) circle (4mm);
\fill[black!55] (14,16) circle (4mm);
\fill[black!20] (14,8) circle (4mm);
\fill[black!53] (14,24) circle (4mm);
\fill[black!33] (14,4) circle (4mm);
\fill[black!55] (14,20) circle (4mm);
\fill[black!9] (14,12) circle (4mm);
\fill[black!49] (14,28) circle (4mm);
\fill[black!28] (14,2) circle (4mm);
\fill[black!57] (14,18) circle (4mm);
\fill[black!17] (14,10) circle (4mm);
\fill[black!54] (14,26) circle (4mm);
\fill[black!36] (14,6) circle (4mm);
\fill[black!51] (14,22) circle (4mm);
\fill[black!2] (14,14) circle (4mm);
\fill[black!54] (14,30) circle (4mm);
\fill[black!31] (14,1) circle (4mm);
\fill[black!53] (14,17) circle (4mm);
\fill[black!19] (14,9) circle (4mm);
\fill[black!52] (14,25) circle (4mm);
\fill[black!35] (14,5) circle (4mm);
\fill[black!53] (14,21) circle (4mm);
\fill[black!10] (14,13) circle (4mm);
\fill[black!52] (14,29) circle (4mm);
\fill[black!33] (14,3) circle (4mm);
\fill[black!53] (14,19) circle (4mm);
\fill[black!11] (14,11) circle (4mm);
\fill[black!55] (14,27) circle (4mm);
\fill[black!34] (14,7) circle (4mm);
\fill[black!56] (14,23) circle (4mm);
\fill[black!5] (14,15) circle (4mm);
\fill[black!57] (14,31) circle (4mm);
\fill[black!52] (30,0) circle (4mm);
\fill[black!31] (30,16) circle (4mm);
\fill[black!55] (30,8) circle (4mm);
\fill[black!14] (30,24) circle (4mm);
\fill[black!53] (30,4) circle (4mm);
\fill[black!33] (30,20) circle (4mm);
\fill[black!56] (30,12) circle (4mm);
\fill[black!7] (30,28) circle (4mm);
\fill[black!54] (30,2) circle (4mm);
\fill[black!35] (30,18) circle (4mm);
\fill[black!50] (30,10) circle (4mm);
\fill[black!20] (30,26) circle (4mm);
\fill[black!51] (30,6) circle (4mm);
\fill[black!33] (30,22) circle (4mm);
\fill[black!54] (30,14) circle (4mm);
\fill[black!2] (30,30) circle (4mm);
\fill[black!50] (30,1) circle (4mm);
\fill[black!29] (30,17) circle (4mm);
\fill[black!51] (30,9) circle (4mm);
\fill[black!17] (30,25) circle (4mm);
\fill[black!54] (30,5) circle (4mm);
\fill[black!37] (30,21) circle (4mm);
\fill[black!55] (30,13) circle (4mm);
\fill[black!10] (30,29) circle (4mm);
\fill[black!53] (30,3) circle (4mm);
\fill[black!30] (30,19) circle (4mm);
\fill[black!57] (30,11) circle (4mm);
\fill[black!16] (30,27) circle (4mm);
\fill[black!57] (30,7) circle (4mm);
\fill[black!35] (30,23) circle (4mm);
\fill[black!52] (30,15) circle (4mm);
\fill[black!5] (30,31) circle (4mm);
\fill[black!5] (1,0) circle (4mm);
\fill[black!50] (1,16) circle (4mm);
\fill[black!36] (1,8) circle (4mm);
\fill[black!52] (1,24) circle (4mm);
\fill[black!17] (1,4) circle (4mm);
\fill[black!61] (1,20) circle (4mm);
\fill[black!33] (1,12) circle (4mm);
\fill[black!59] (1,28) circle (4mm);
\fill[black!11] (1,2) circle (4mm);
\fill[black!51] (1,18) circle (4mm);
\fill[black!28] (1,10) circle (4mm);
\fill[black!55] (1,26) circle (4mm);
\fill[black!18] (1,6) circle (4mm);
\fill[black!52] (1,22) circle (4mm);
\fill[black!31] (1,14) circle (4mm);
\fill[black!50] (1,30) circle (4mm);
\fill[black!0] (1,1) circle (4mm);
\fill[black!51] (1,17) circle (4mm);
\fill[black!35] (1,9) circle (4mm);
\fill[black!57] (1,25) circle (4mm);
\fill[black!20] (1,5) circle (4mm);
\fill[black!54] (1,21) circle (4mm);
\fill[black!33] (1,13) circle (4mm);
\fill[black!48] (1,29) circle (4mm);
\fill[black!7] (1,3) circle (4mm);
\fill[black!51] (1,19) circle (4mm);
\fill[black!37] (1,11) circle (4mm);
\fill[black!46] (1,27) circle (4mm);
\fill[black!15] (1,7) circle (4mm);
\fill[black!49] (1,23) circle (4mm);
\fill[black!32] (1,15) circle (4mm);
\fill[black!58] (1,31) circle (4mm);
\fill[black!55] (17,0) circle (4mm);
\fill[black!3] (17,16) circle (4mm);
\fill[black!54] (17,8) circle (4mm);
\fill[black!32] (17,24) circle (4mm);
\fill[black!55] (17,4) circle (4mm);
\fill[black!16] (17,20) circle (4mm);
\fill[black!55] (17,12) circle (4mm);
\fill[black!36] (17,28) circle (4mm);
\fill[black!49] (17,2) circle (4mm);
\fill[black!7] (17,18) circle (4mm);
\fill[black!54] (17,10) circle (4mm);
\fill[black!28] (17,26) circle (4mm);
\fill[black!57] (17,6) circle (4mm);
\fill[black!15] (17,22) circle (4mm);
\fill[black!53] (17,14) circle (4mm);
\fill[black!29] (17,30) circle (4mm);
\fill[black!51] (17,1) circle (4mm);
\fill[black!0] (17,17) circle (4mm);
\fill[black!53] (17,9) circle (4mm);
\fill[black!39] (17,25) circle (4mm);
\fill[black!53] (17,5) circle (4mm);
\fill[black!13] (17,21) circle (4mm);
\fill[black!59] (17,13) circle (4mm);
\fill[black!33] (17,29) circle (4mm);
\fill[black!57] (17,3) circle (4mm);
\fill[black!6] (17,19) circle (4mm);
\fill[black!54] (17,11) circle (4mm);
\fill[black!35] (17,27) circle (4mm);
\fill[black!55] (17,7) circle (4mm);
\fill[black!16] (17,23) circle (4mm);
\fill[black!56] (17,15) circle (4mm);
\fill[black!35] (17,31) circle (4mm);
\fill[black!30] (9,0) circle (4mm);
\fill[black!54] (9,16) circle (4mm);
\fill[black!3] (9,8) circle (4mm);
\fill[black!50] (9,24) circle (4mm);
\fill[black!31] (9,4) circle (4mm);
\fill[black!50] (9,20) circle (4mm);
\fill[black!23] (9,12) circle (4mm);
\fill[black!57] (9,28) circle (4mm);
\fill[black!37] (9,2) circle (4mm);
\fill[black!55] (9,18) circle (4mm);
\fill[black!5] (9,10) circle (4mm);
\fill[black!54] (9,26) circle (4mm);
\fill[black!30] (9,6) circle (4mm);
\fill[black!53] (9,22) circle (4mm);
\fill[black!19] (9,14) circle (4mm);
\fill[black!51] (9,30) circle (4mm);
\fill[black!35] (9,1) circle (4mm);
\fill[black!53] (9,17) circle (4mm);
\fill[black!3] (9,9) circle (4mm);
\fill[black!56] (9,25) circle (4mm);
\fill[black!34] (9,5) circle (4mm);
\fill[black!49] (9,21) circle (4mm);
\fill[black!19] (9,13) circle (4mm);
\fill[black!54] (9,29) circle (4mm);
\fill[black!30] (9,3) circle (4mm);
\fill[black!60] (9,19) circle (4mm);
\fill[black!10] (9,11) circle (4mm);
\fill[black!54] (9,27) circle (4mm);
\fill[black!34] (9,7) circle (4mm);
\fill[black!49] (9,23) circle (4mm);
\fill[black!16] (9,15) circle (4mm);
\fill[black!52] (9,31) circle (4mm);
\fill[black!47] (25,0) circle (4mm);
\fill[black!36] (25,16) circle (4mm);
\fill[black!57] (25,8) circle (4mm);
\fill[black!4] (25,24) circle (4mm);
\fill[black!50] (25,4) circle (4mm);
\fill[black!35] (25,20) circle (4mm);
\fill[black!51] (25,12) circle (4mm);
\fill[black!19] (25,28) circle (4mm);
\fill[black!54] (25,2) circle (4mm);
\fill[black!37] (25,18) circle (4mm);
\fill[black!52] (25,10) circle (4mm);
\fill[black!9] (25,26) circle (4mm);
\fill[black!50] (25,6) circle (4mm);
\fill[black!34] (25,22) circle (4mm);
\fill[black!52] (25,14) circle (4mm);
\fill[black!17] (25,30) circle (4mm);
\fill[black!57] (25,1) circle (4mm);
\fill[black!39] (25,17) circle (4mm);
\fill[black!56] (25,9) circle (4mm);
\fill[black!1] (25,25) circle (4mm);
\fill[black!52] (25,5) circle (4mm);
\fill[black!32] (25,21) circle (4mm);
\fill[black!57] (25,13) circle (4mm);
\fill[black!14] (25,29) circle (4mm);
\fill[black!60] (25,3) circle (4mm);
\fill[black!36] (25,19) circle (4mm);
\fill[black!56] (25,11) circle (4mm);
\fill[black!7] (25,27) circle (4mm);
\fill[black!56] (25,7) circle (4mm);
\fill[black!34] (25,23) circle (4mm);
\fill[black!56] (25,15) circle (4mm);
\fill[black!18] (25,31) circle (4mm);
\fill[black!17] (5,0) circle (4mm);
\fill[black!56] (5,16) circle (4mm);
\fill[black!37] (5,8) circle (4mm);
\fill[black!54] (5,24) circle (4mm);
\fill[black!3] (5,4) circle (4mm);
\fill[black!53] (5,20) circle (4mm);
\fill[black!36] (5,12) circle (4mm);
\fill[black!60] (5,28) circle (4mm);
\fill[black!15] (5,2) circle (4mm);
\fill[black!54] (5,18) circle (4mm);
\fill[black!34] (5,10) circle (4mm);
\fill[black!64] (5,26) circle (4mm);
\fill[black!7] (5,6) circle (4mm);
\fill[black!56] (5,22) circle (4mm);
\fill[black!35] (5,14) circle (4mm);
\fill[black!54] (5,30) circle (4mm);
\fill[black!20] (5,1) circle (4mm);
\fill[black!53] (5,17) circle (4mm);
\fill[black!34] (5,9) circle (4mm);
\fill[black!52] (5,25) circle (4mm);
\fill[black!1] (5,5) circle (4mm);
\fill[black!58] (5,21) circle (4mm);
\fill[black!35] (5,13) circle (4mm);
\fill[black!57] (5,29) circle (4mm);
\fill[black!17] (5,3) circle (4mm);
\fill[black!54] (5,19) circle (4mm);
\fill[black!32] (5,11) circle (4mm);
\fill[black!45] (5,27) circle (4mm);
\fill[black!9] (5,7) circle (4mm);
\fill[black!50] (5,23) circle (4mm);
\fill[black!31] (5,15) circle (4mm);
\fill[black!54] (5,31) circle (4mm);
\fill[black!54] (21,0) circle (4mm);
\fill[black!15] (21,16) circle (4mm);
\fill[black!52] (21,8) circle (4mm);
\fill[black!29] (21,24) circle (4mm);
\fill[black!52] (21,4) circle (4mm);
\fill[black!5] (21,20) circle (4mm);
\fill[black!50] (21,12) circle (4mm);
\fill[black!34] (21,28) circle (4mm);
\fill[black!54] (21,2) circle (4mm);
\fill[black!18] (21,18) circle (4mm);
\fill[black!51] (21,10) circle (4mm);
\fill[black!31] (21,26) circle (4mm);
\fill[black!61] (21,6) circle (4mm);
\fill[black!7] (21,22) circle (4mm);
\fill[black!53] (21,14) circle (4mm);
\fill[black!37] (21,30) circle (4mm);
\fill[black!54] (21,1) circle (4mm);
\fill[black!13] (21,17) circle (4mm);
\fill[black!49] (21,9) circle (4mm);
\fill[black!32] (21,25) circle (4mm);
\fill[black!58] (21,5) circle (4mm);
\fill[black!0] (21,21) circle (4mm);
\fill[black!59] (21,13) circle (4mm);
\fill[black!34] (21,29) circle (4mm);
\fill[black!46] (21,3) circle (4mm);
\fill[black!14] (21,19) circle (4mm);
\fill[black!56] (21,11) circle (4mm);
\fill[black!33] (21,27) circle (4mm);
\fill[black!49] (21,7) circle (4mm);
\fill[black!9] (21,23) circle (4mm);
\fill[black!55] (21,15) circle (4mm);
\fill[black!29] (21,31) circle (4mm);
\fill[black!35] (13,0) circle (4mm);
\fill[black!51] (13,16) circle (4mm);
\fill[black!19] (13,8) circle (4mm);
\fill[black!58] (13,24) circle (4mm);
\fill[black!34] (13,4) circle (4mm);
\fill[black!60] (13,20) circle (4mm);
\fill[black!3] (13,12) circle (4mm);
\fill[black!61] (13,28) circle (4mm);
\fill[black!33] (13,2) circle (4mm);
\fill[black!53] (13,18) circle (4mm);
\fill[black!16] (13,10) circle (4mm);
\fill[black!54] (13,26) circle (4mm);
\fill[black!36] (13,6) circle (4mm);
\fill[black!57] (13,22) circle (4mm);
\fill[black!10] (13,14) circle (4mm);
\fill[black!55] (13,30) circle (4mm);
\fill[black!33] (13,1) circle (4mm);
\fill[black!59] (13,17) circle (4mm);
\fill[black!19] (13,9) circle (4mm);
\fill[black!57] (13,25) circle (4mm);
\fill[black!35] (13,5) circle (4mm);
\fill[black!59] (13,21) circle (4mm);
\fill[black!1] (13,13) circle (4mm);
\fill[black!52] (13,29) circle (4mm);
\fill[black!33] (13,3) circle (4mm);
\fill[black!55] (13,19) circle (4mm);
\fill[black!17] (13,11) circle (4mm);
\fill[black!56] (13,27) circle (4mm);
\fill[black!30] (13,7) circle (4mm);
\fill[black!59] (13,23) circle (4mm);
\fill[black!8] (13,15) circle (4mm);
\fill[black!55] (13,31) circle (4mm);
\fill[black!53] (29,0) circle (4mm);
\fill[black!34] (29,16) circle (4mm);
\fill[black!54] (29,8) circle (4mm);
\fill[black!17] (29,24) circle (4mm);
\fill[black!52] (29,4) circle (4mm);
\fill[black!32] (29,20) circle (4mm);
\fill[black!53] (29,12) circle (4mm);
\fill[black!2] (29,28) circle (4mm);
\fill[black!55] (29,2) circle (4mm);
\fill[black!33] (29,18) circle (4mm);
\fill[black!51] (29,10) circle (4mm);
\fill[black!20] (29,26) circle (4mm);
\fill[black!56] (29,6) circle (4mm);
\fill[black!29] (29,22) circle (4mm);
\fill[black!52] (29,14) circle (4mm);
\fill[black!10] (29,30) circle (4mm);
\fill[black!48] (29,1) circle (4mm);
\fill[black!33] (29,17) circle (4mm);
\fill[black!54] (29,9) circle (4mm);
\fill[black!14] (29,25) circle (4mm);
\fill[black!57] (29,5) circle (4mm);
\fill[black!34] (29,21) circle (4mm);
\fill[black!52] (29,13) circle (4mm);
\fill[black!0] (29,29) circle (4mm);
\fill[black!55] (29,3) circle (4mm);
\fill[black!34] (29,19) circle (4mm);
\fill[black!55] (29,11) circle (4mm);
\fill[black!23] (29,27) circle (4mm);
\fill[black!55] (29,7) circle (4mm);
\fill[black!30] (29,23) circle (4mm);
\fill[black!57] (29,15) circle (4mm);
\fill[black!10] (29,31) circle (4mm);
\fill[black!10] (3,0) circle (4mm);
\fill[black!54] (3,16) circle (4mm);
\fill[black!32] (3,8) circle (4mm);
\fill[black!57] (3,24) circle (4mm);
\fill[black!12] (3,4) circle (4mm);
\fill[black!52] (3,20) circle (4mm);
\fill[black!34] (3,12) circle (4mm);
\fill[black!52] (3,28) circle (4mm);
\fill[black!4] (3,2) circle (4mm);
\fill[black!60] (3,18) circle (4mm);
\fill[black!32] (3,10) circle (4mm);
\fill[black!53] (3,26) circle (4mm);
\fill[black!16] (3,6) circle (4mm);
\fill[black!50] (3,22) circle (4mm);
\fill[black!33] (3,14) circle (4mm);
\fill[black!53] (3,30) circle (4mm);
\fill[black!7] (3,1) circle (4mm);
\fill[black!57] (3,17) circle (4mm);
\fill[black!30] (3,9) circle (4mm);
\fill[black!60] (3,25) circle (4mm);
\fill[black!17] (3,5) circle (4mm);
\fill[black!46] (3,21) circle (4mm);
\fill[black!33] (3,13) circle (4mm);
\fill[black!55] (3,29) circle (4mm);
\fill[black!1] (3,3) circle (4mm);
\fill[black!60] (3,19) circle (4mm);
\fill[black!31] (3,11) circle (4mm);
\fill[black!56] (3,27) circle (4mm);
\fill[black!23] (3,7) circle (4mm);
\fill[black!54] (3,23) circle (4mm);
\fill[black!31] (3,15) circle (4mm);
\fill[black!63] (3,31) circle (4mm);
\fill[black!50] (19,0) circle (4mm);
\fill[black!9] (19,16) circle (4mm);
\fill[black!61] (19,8) circle (4mm);
\fill[black!35] (19,24) circle (4mm);
\fill[black!52] (19,4) circle (4mm);
\fill[black!18] (19,20) circle (4mm);
\fill[black!52] (19,12) circle (4mm);
\fill[black!29] (19,28) circle (4mm);
\fill[black!54] (19,2) circle (4mm);
\fill[black!4] (19,18) circle (4mm);
\fill[black!46] (19,10) circle (4mm);
\fill[black!35] (19,26) circle (4mm);
\fill[black!54] (19,6) circle (4mm);
\fill[black!18] (19,22) circle (4mm);
\fill[black!53] (19,14) circle (4mm);
\fill[black!30] (19,30) circle (4mm);
\fill[black!51] (19,1) circle (4mm);
\fill[black!6] (19,17) circle (4mm);
\fill[black!60] (19,9) circle (4mm);
\fill[black!36] (19,25) circle (4mm);
\fill[black!54] (19,5) circle (4mm);
\fill[black!14] (19,21) circle (4mm);
\fill[black!55] (19,13) circle (4mm);
\fill[black!34] (19,29) circle (4mm);
\fill[black!60] (19,3) circle (4mm);
\fill[black!0] (19,19) circle (4mm);
\fill[black!56] (19,11) circle (4mm);
\fill[black!32] (19,27) circle (4mm);
\fill[black!49] (19,7) circle (4mm);
\fill[black!21] (19,23) circle (4mm);
\fill[black!57] (19,15) circle (4mm);
\fill[black!32] (19,31) circle (4mm);
\fill[black!33] (11,0) circle (4mm);
\fill[black!52] (11,16) circle (4mm);
\fill[black!11] (11,8) circle (4mm);
\fill[black!47] (11,24) circle (4mm);
\fill[black!38] (11,4) circle (4mm);
\fill[black!55] (11,20) circle (4mm);
\fill[black!16] (11,12) circle (4mm);
\fill[black!51] (11,28) circle (4mm);
\fill[black!32] (11,2) circle (4mm);
\fill[black!53] (11,18) circle (4mm);
\fill[black!5] (11,10) circle (4mm);
\fill[black!52] (11,26) circle (4mm);
\fill[black!34] (11,6) circle (4mm);
\fill[black!51] (11,22) circle (4mm);
\fill[black!11] (11,14) circle (4mm);
\fill[black!57] (11,30) circle (4mm);
\fill[black!37] (11,1) circle (4mm);
\fill[black!54] (11,17) circle (4mm);
\fill[black!10] (11,9) circle (4mm);
\fill[black!56] (11,25) circle (4mm);
\fill[black!32] (11,5) circle (4mm);
\fill[black!56] (11,21) circle (4mm);
\fill[black!17] (11,13) circle (4mm);
\fill[black!55] (11,29) circle (4mm);
\fill[black!31] (11,3) circle (4mm);
\fill[black!56] (11,19) circle (4mm);
\fill[black!1] (11,11) circle (4mm);
\fill[black!56] (11,27) circle (4mm);
\fill[black!31] (11,7) circle (4mm);
\fill[black!50] (11,23) circle (4mm);
\fill[black!13] (11,15) circle (4mm);
\fill[black!49] (11,31) circle (4mm);
\fill[black!50] (27,0) circle (4mm);
\fill[black!34] (27,16) circle (4mm);
\fill[black!53] (27,8) circle (4mm);
\fill[black!8] (27,24) circle (4mm);
\fill[black!52] (27,4) circle (4mm);
\fill[black!30] (27,20) circle (4mm);
\fill[black!56] (27,12) circle (4mm);
\fill[black!22] (27,28) circle (4mm);
\fill[black!53] (27,2) circle (4mm);
\fill[black!33] (27,18) circle (4mm);
\fill[black!60] (27,10) circle (4mm);
\fill[black!4] (27,26) circle (4mm);
\fill[black!56] (27,6) circle (4mm);
\fill[black!28] (27,22) circle (4mm);
\fill[black!55] (27,14) circle (4mm);
\fill[black!16] (27,30) circle (4mm);
\fill[black!46] (27,1) circle (4mm);
\fill[black!35] (27,17) circle (4mm);
\fill[black!54] (27,9) circle (4mm);
\fill[black!7] (27,25) circle (4mm);
\fill[black!45] (27,5) circle (4mm);
\fill[black!33] (27,21) circle (4mm);
\fill[black!56] (27,13) circle (4mm);
\fill[black!23] (27,29) circle (4mm);
\fill[black!56] (27,3) circle (4mm);
\fill[black!32] (27,19) circle (4mm);
\fill[black!56] (27,11) circle (4mm);
\fill[black!56] (27,7) circle (4mm);
\fill[black!34] (27,23) circle (4mm);
\fill[black!55] (27,15) circle (4mm);
\fill[black!20] (27,31) circle (4mm);
\fill[black!20] (7,0) circle (4mm);
\fill[black!52] (7,16) circle (4mm);
\fill[black!36] (7,8) circle (4mm);
\fill[black!51] (7,24) circle (4mm);
\fill[black!10] (7,4) circle (4mm);
\fill[black!57] (7,20) circle (4mm);
\fill[black!29] (7,12) circle (4mm);
\fill[black!53] (7,28) circle (4mm);
\fill[black!15] (7,2) circle (4mm);
\fill[black!57] (7,18) circle (4mm);
\fill[black!32] (7,10) circle (4mm);
\fill[black!50] (7,26) circle (4mm);
\fill[black!6] (7,6) circle (4mm);
\fill[black!52] (7,22) circle (4mm);
\fill[black!34] (7,14) circle (4mm);
\fill[black!57] (7,30) circle (4mm);
\fill[black!15] (7,1) circle (4mm);
\fill[black!55] (7,17) circle (4mm);
\fill[black!34] (7,9) circle (4mm);
\fill[black!56] (7,25) circle (4mm);
\fill[black!9] (7,5) circle (4mm);
\fill[black!49] (7,21) circle (4mm);
\fill[black!30] (7,13) circle (4mm);
\fill[black!55] (7,29) circle (4mm);
\fill[black!23] (7,3) circle (4mm);
\fill[black!49] (7,19) circle (4mm);
\fill[black!31] (7,11) circle (4mm);
\fill[black!56] (7,27) circle (4mm);
\fill[black!2] (7,7) circle (4mm);
\fill[black!56] (7,23) circle (4mm);
\fill[black!35] (7,15) circle (4mm);
\fill[black!57] (7,31) circle (4mm);
\fill[black!57] (23,0) circle (4mm);
\fill[black!18] (23,16) circle (4mm);
\fill[black!58] (23,8) circle (4mm);
\fill[black!32] (23,24) circle (4mm);
\fill[black!59] (23,4) circle (4mm);
\fill[black!11] (23,20) circle (4mm);
\fill[black!50] (23,12) circle (4mm);
\fill[black!31] (23,28) circle (4mm);
\fill[black!55] (23,2) circle (4mm);
\fill[black!19] (23,18) circle (4mm);
\fill[black!57] (23,10) circle (4mm);
\fill[black!32] (23,26) circle (4mm);
\fill[black!54] (23,6) circle (4mm);
\fill[black!5] (23,22) circle (4mm);
\fill[black!56] (23,14) circle (4mm);
\fill[black!35] (23,30) circle (4mm);
\fill[black!49] (23,1) circle (4mm);
\fill[black!16] (23,17) circle (4mm);
\fill[black!49] (23,9) circle (4mm);
\fill[black!34] (23,25) circle (4mm);
\fill[black!50] (23,5) circle (4mm);
\fill[black!9] (23,21) circle (4mm);
\fill[black!59] (23,13) circle (4mm);
\fill[black!30] (23,29) circle (4mm);
\fill[black!54] (23,3) circle (4mm);
\fill[black!21] (23,19) circle (4mm);
\fill[black!50] (23,11) circle (4mm);
\fill[black!34] (23,27) circle (4mm);
\fill[black!56] (23,7) circle (4mm);
\fill[black!56] (23,15) circle (4mm);
\fill[black!30] (23,31) circle (4mm);
\fill[black!32] (15,0) circle (4mm);
\fill[black!51] (15,16) circle (4mm);
\fill[black!19] (15,8) circle (4mm);
\fill[black!50] (15,24) circle (4mm);
\fill[black!30] (15,4) circle (4mm);
\fill[black!54] (15,20) circle (4mm);
\fill[black!9] (15,12) circle (4mm);
\fill[black!52] (15,28) circle (4mm);
\fill[black!34] (15,2) circle (4mm);
\fill[black!52] (15,18) circle (4mm);
\fill[black!17] (15,10) circle (4mm);
\fill[black!52] (15,26) circle (4mm);
\fill[black!34] (15,6) circle (4mm);
\fill[black!58] (15,22) circle (4mm);
\fill[black!5] (15,14) circle (4mm);
\fill[black!52] (15,30) circle (4mm);
\fill[black!32] (15,1) circle (4mm);
\fill[black!56] (15,17) circle (4mm);
\fill[black!16] (15,9) circle (4mm);
\fill[black!56] (15,25) circle (4mm);
\fill[black!31] (15,5) circle (4mm);
\fill[black!55] (15,21) circle (4mm);
\fill[black!8] (15,13) circle (4mm);
\fill[black!57] (15,29) circle (4mm);
\fill[black!31] (15,3) circle (4mm);
\fill[black!57] (15,19) circle (4mm);
\fill[black!13] (15,11) circle (4mm);
\fill[black!55] (15,27) circle (4mm);
\fill[black!35] (15,7) circle (4mm);
\fill[black!56] (15,23) circle (4mm);
\fill[black!2] (15,15) circle (4mm);
\fill[black!48] (15,31) circle (4mm);
\fill[black!61] (31,0) circle (4mm);
\fill[black!33] (31,16) circle (4mm);
\fill[black!53] (31,8) circle (4mm);
\fill[black!16] (31,24) circle (4mm);
\fill[black!50] (31,4) circle (4mm);
\fill[black!34] (31,20) circle (4mm);
\fill[black!52] (31,12) circle (4mm);
\fill[black!11] (31,28) circle (4mm);
\fill[black!50] (31,2) circle (4mm);
\fill[black!28] (31,18) circle (4mm);
\fill[black!54] (31,10) circle (4mm);
\fill[black!15] (31,26) circle (4mm);
\fill[black!55] (31,6) circle (4mm);
\fill[black!35] (31,22) circle (4mm);
\fill[black!57] (31,14) circle (4mm);
\fill[black!5] (31,30) circle (4mm);
\fill[black!58] (31,1) circle (4mm);
\fill[black!35] (31,17) circle (4mm);
\fill[black!52] (31,9) circle (4mm);
\fill[black!18] (31,25) circle (4mm);
\fill[black!54] (31,5) circle (4mm);
\fill[black!29] (31,21) circle (4mm);
\fill[black!55] (31,13) circle (4mm);
\fill[black!10] (31,29) circle (4mm);
\fill[black!63] (31,3) circle (4mm);
\fill[black!32] (31,19) circle (4mm);
\fill[black!49] (31,11) circle (4mm);
\fill[black!20] (31,27) circle (4mm);
\fill[black!57] (31,7) circle (4mm);
\fill[black!30] (31,23) circle (4mm);
\fill[black!48] (31,15) circle (4mm);
\fill[black!1] (31,31) circle (4mm);
\end{tikzpicture}